\documentclass[twoside]{irmaems}
\usepackage{epsf}

\usepackage{amssymb} 
\usepackage{amsmath} 
\usepackage{amscd} 
\usepackage{latexsym} 

\usepackage{xy}
\xyoption{all}

\usepackage{makeidx}
\makeindex

\begin{document}  

\title
{Trace coordinates on Fricke spaces of some simple hyperbolic surfaces}

\author{William M.  Goldman\thanks{ Work partially supported by
National Science Foundation grant DMS070781, a Semester Research Award
(Fall 2005) from the General Research Board of the University of
Maryland, and the Oswald Veblen Fund at the Institute for Advanced
Study.} }

\address{
Department of Mathematics \\
University of Maryland \\
College Park, MD 20742 USA \\
email:\,\tt{wmg@math.umd.edu}}

\maketitle

\renewcommand\theequation{\arabic{section}.\arabic{equation}}
 
\theoremstyle{definition} 

\newtheorem{definition}{Definition}[subsection]
\newtheorem{remark}[definition]{Remark}
\newtheorem{example}[definition]{Example}


\newtheorem*{theoremA}{Theorem A}  
\newtheorem*{theoremB}{Theorem B}  
\newtheorem*{theoremC}{Theorem C}  
\newtheorem*{theoremD}{Theorem D}  

\newtheorem*{notation}{Notation}  

\theoremstyle{plain}      

\newtheorem{proposition}[definition]{Proposition}
\newtheorem{theorem}[definition]{Theorem}
\newtheorem{corollary}[definition]{Corollary}
\newtheorem{lemma}[definition]{Lemma}
\newtheorem*{conjecture}{Conjecture}

\numberwithin{equation}{subsection}

\date{\today}


\newcommand{\R}{{\mathbb R}}
\newcommand{\C}{{\mathbb C}}
\newcommand{\Z}{\mathbb{Z}}
\newcommand{\F}{\mathbb{F}}
\newcommand{\kk}{\mathsf{k}}
\newcommand{\HH}{\mathbb{H}}

\newcommand{\interior}{\mathsf{int}}

\newcommand{\Id}{{\mathbb I}}
\newcommand{\Ker}{\mathsf{Ker}}
\newcommand{\Aut}{\mathsf{Aut}}
\newcommand{\Ad}{\mathsf{Ad}}
\newcommand{\Inn}{\mathsf{Inn}}
\newcommand{\tr}{\mathsf{tr}}
\renewcommand{\det}{\mathsf{det}}
\newcommand{\Fix}{\mathsf{Fix}}
\newcommand{\Sym}{\mathsf{Sym}}
\newcommand{\Hom}{\mathsf{Hom}}

\newcommand{\GL}[1]{{\mathsf{GL}}({#1})}
\newcommand{\GLt}{\mathsf{GL}(2)}
\newcommand{\GLtC}{{\GL{2,\C}}}
\newcommand{\GLthC}{{\GL{3,\C}}}
\newcommand{\GLtR}{\mathsf{GL}(2,\R)}

\newcommand{\SL}[1]{{\mathsf{SL}}({#1})}
\newcommand{\SLt}{{\SL{2}}}
\newcommand{\SLtC}{{\SL{2,\C}}}
\newcommand{\SLthC}{{\SL{3,\C}}}
\newcommand{\SLtR}{{\SL{2,\R}}}

\newcommand{\PGL}[1]{{\mathsf{PGL}}({#1})}
\newcommand{\PGLt}{{\PGL{2}}}
\newcommand{\PGLtR}{\mathsf{PGL}(2,\R)}
\newcommand{\PGLtC}{{\PGL{2,\C}}}

\newcommand{\PSL}[1]{{\mathsf{PSL}}({#1})}
\newcommand{\PSLt}{{\PSL{2}}}
\newcommand{\PSLtC}{\mathsf{PSL}(2,\C)}
\newcommand{\PSLtR}{\mathsf{PSL}(2,\R)}

\newcommand{\SUt}{{\mathsf{SU}(2)}}
\newcommand{\SOt}{{\mathsf{SO}(2)}}
\newcommand{\SOot}{{\mathsf{SO}(1,2)}}
\newcommand{\SOto}{{\mathsf{SO}(2,1)}}
\newcommand{\SOoo}{{\mathsf{SO}(1,1)}}
\newcommand{\SOthC}{\mathsf{SO}(3,\C)}

\newcommand{\Mat}{{\mathsf{M}_2(\C)}}
\newcommand{\Oto}{{\mathsf{O}(2,1)}}
\newcommand{\OB}{\mathsf{O}(\C^3,\BB)}
\newcommand{\SOB}{\mathsf{SO}(\C^3,\BB)}

\renewcommand{\sl}{\mathfrak{sl}}
\newcommand{\slt}{\sl(2)}
\newcommand{\gl}{\mathfrak{gl}}
\newcommand{\gltR}{\gl(2,\R)}
\newcommand{\sltR}{\sl(2,\R)}
\newcommand{\sltC}{\sl(2,\C)}

\renewcommand{\P}{\mathbb{P}}
\renewcommand{\Id}{\mathbb{I}}

\newcommand{\CP}{\mathbb{CP}}
\newcommand{\RP}{\mathbb{RP}}

\newcommand{\rxy}{\rho_{XY}}
\newcommand{\ryz}{\rho_{YZ}}
\newcommand{\rzx}{\rho_{ZX}}
\newcommand{\lxy}{l_{XY}}
\newcommand{\lyz}{l_{YZ}}
\newcommand{\lzx}{l_{ZX}}

\newcommand{\bt}{{\mathsf{t}}}
\newcommand{\hmg}{\Hom(\pi ,G)}
\newcommand{\W}{{\mathcal{W}}}
\newcommand{\BB}{{\mathbb{B}}}

\newcommand{\Ht}{\mathsf{H}^2}
\newcommand{\Hth}{\mathsf{H}^3}
\newcommand{\Rto}{\R^3_1}

\newcommand{\Inv}{\mathsf{Inv}}
\newcommand{\jj}{\mathsf{j}}

\newcommand{\Ss}{\mathfrak{S}}
\newcommand{\dst}{\mathsf{dS}^2_1}

\newcommand{\Hh}{\mathcal{H}}
\newcommand{\hh}{\mathfrak{H}}
\newcommand{\Eta}{\mathbf{\eta}}
\newcommand{\Rho}{\mathbf{\rho}}

\newcommand{\hexagon}{\mathsf{Hex}}
\newcommand{\rank}{\mathsf{rank}}
\newcommand{\Lie}{\mathsf{Lie}}

\newcommand{\Hx}{\Hh_{\widehat{X}}}
\newcommand{\Hy}{\Hh_{\widehat{Y}}}
\newcommand{\Hz}{\Hh_{\widehat{Z}}}

\newcommand{\hS}{\hat{\Sigma}}
\newcommand{\hP}{\hat{\Pi}}

\newcommand{\csch}{\operatorname{csch}}

\newcommand{\Fricke}{\mathfrak{F}}

\newcommand{\trho}{\tilde{\rho}}
\newcommand{\tG}{\tilde{G}}
\newcommand{\zz}{\mathfrak{z}}
\newcommand{\RR}{\mathfrak{R}}

\setcounter{page}{1}
\tableofcontents   

\section{Introduction}

The work of Fricke-Klein~\cite{FrickeKlein} develops the deformation
theory of hyperbolic structures on a surface $\Sigma$ in terms of the space
of representations of its fundamental group $\pi=\pi_1(\Sigma)$ in
$\SLtC$. This leads to an algebraic structure on the deformation spaces.
Here we expound this theory from a modern viewpoint.

We emphasize the close relationship between algebra and 
geometry. In particular algebraic properties of $2\times 2$ matrices
are applied to hyperbolic geometry in low dimensions. Our main object
of interest is the {\em deformation space\/} of hyperbolic structures
on a fixed compact surface-with-boundary $\Sigma$. The points of 
this deformation space correspond to equivalence classes of 
{\em marked\/} hyperbolic structures on $\interior(\Sigma)$
where the ends are either cusps (complete ends of finite area)
or are collar neighborhoods of closed geodesics. Such deformation
spaces have been named {\em Fricke spaces\/} by 
Bers-Gardiner~\cite{BersGardiner}. When $\Sigma$ is closed,
then the uniformization theorem identifies hyperbolic structures
with conformal structures and the Fricke space is commonly
identified with the {\em Teichm\"uller space\/} of marked 
{\em conformal\/} structures on $\Sigma$.

Hyperbolic structures are a special case of {\em locally homogeneous
geometric structures\/} modelled on a homogeneous space of a Lie
group $G$. These structures were first systematically defined by
Ehresmann~\cite{Ehresmann}, and they determine representations
of the fundamental group $\pi_1(\Sigma)$ in $G$. Equivalence
classes of structures determine {\em equivalence classes\/}
of representations, and the first part of this chapter deals
with the algebraic problem of determining the moduli space
of equivalence classes of pairs of unimodular $2\times 2$
matrices.

Our starting point is the following well-known yet fundamental fact
when $\pi$ is a free group $\F_2$ of rank two. This fact may be found
in the book of Fricke and Klein~\cite{FrickeKlein} and the even
earlier paper of Vogt~\cite{Vogt}. 
Perhaps 
much was known at the time about
invariants of $2\times 2$ matrices among the early practitioners
of what has since become known as ``classical invariant theory''.
Now this algebraic work is contained in the powerful general theory
developed by Procesi~\cite{Procesi} and others, 
which in a sense completes the work begun in the 19th century. 

Procesi's theorem implies that the ring of invariants on the
space of representations $\pi\xrightarrow{\rho} \SLtC$ is generated
by {\em characters\/}
\begin{equation*}
\rho \stackrel{t_\gamma}\longmapsto \tr\big(\rho(\gamma)\big), 
\end{equation*}
where $\gamma\in\pi$, and hence we call this ring the {\em character
ring.\/} We begin by proving the elementary fact that character ring
$\RR_1$ of a cyclic group is the polynomial ring $\C[\tr]$ generated
by the {\em trace function\/} $\SLtC\xrightarrow{\tr}\C$. From this we
proceed to the basic fact, that the 
character ring $\RR_2$ of the rank two free group $\F_2$ is a
polynomial ring on {\em three variables:\/} 

\index{character ring}

\begin{theoremA}[Vogt~\cite{Vogt},\ Fricke~\cite{Fricke}]
\label{thm:vogt}
Let $\SLtC\times \SLtC\xrightarrow{f}\C$ be a regular function which is invariant under
the diagonal action of $\SLtC$ by conjugation. 
There exists a polynomial function 
$F(x,y,z)\in\C[x,y,z]$ such that 
\begin{equation*}
f(\xi,\eta) = F(\tr(\xi),\tr(\eta),\tr(\xi\eta)). 
\end{equation*}
Furthermore, for all  $(x,y,z)\in\C^3$, there exists $(\xi,\eta)\in \SLtC\times\SLtC$ such that 
\begin{equation*}
\bmatrix x \\ y \\z \endbmatrix = 
\bmatrix \tr(\xi) \\ \tr(\eta) \\\tr(\xi\eta) \endbmatrix.  
\end{equation*}
Conversely, if 
$x^2 + y^2 + z^2 - x y z \neq 4$ 
and 
$(\xi,\eta),(\xi',\eta')\in\SLtC\times\SLtC $ satisfy 
\begin{equation*}
\bmatrix \tr(\xi) \\ \tr(\eta) \\ \tr(\xi\eta) \endbmatrix = 
\bmatrix \tr(\xi') \\ \tr(\eta') \\ \tr(\xi'\eta') \endbmatrix = 
\bmatrix x \\ y \\ z \endbmatrix,  
\end{equation*}
then $(\xi',\eta') = g.(\xi,\eta)$ for some $g\in G$.
\end{theoremA}

Algebro-geometrically, 
Theorem~A asserts that the 
{\em $\SLtC$\/}-character variety $V_2$ of a free group of rank two equals
$\C^3$. This will be our basic algebraic tool for describing moduli spaces
of structures on the surface $\Sigma$ and their automorphisms arising
from transformations of $\Sigma$.  

The 
condition $x^2 + y^2 + z^2 - x y z \neq 4$ 
also means that the matrix group
$\langle\xi,\eta\rangle$ acts {\em irreducibly\/} on $\C^2$. That is,
$\langle\xi,\eta\rangle$ preserves no proper nonzero linear subspace
of $\C^2$.  The condition that $\xi,\eta$ generate an irreducible
representation is crucial in several alternate descriptions of
$\SLtC$-representations of $\F_2$.  In particular, it is equivalent to
the condition that the $\PGLtC$-orbit is closed in $\Hom(\F_2,\SLtC)$.
This condition is in turn equivalent to the orbit being 
{\em stable\/} in the sense of Geometric Invariant Theory.


A more geometric description involves the action of the subgroup
$\langle \xi,\eta\rangle\subset\SLtC$ on hyperbolic $3$-space
$\Hth$. The group $\PSLtC$ acts by orientation-preserving isometries
of $\Hth$. An {\em involution,\/} that is, an element $g\in \PSLtC$
having order two, is reflection in a unique geodesic
$\Fix(g)\subset\Hth$. Denote the space of such involutions by $\Inv$.

\begin{theoremB}[Coxeter extension]
Suppose that $\xi,\eta\in\SLtC$ generate an irreducible representation
and let $\zeta = \eta^{-1}\xi^{-1}$ so that
\begin{equation*}
\xi \eta \zeta = \Id.
\end{equation*}
Then there exists a unique triple of involutions 
\begin{equation*}
\iota_{\xi\eta},\iota_{\eta\zeta},\iota_{\zeta\xi}\in\Inv 
\end{equation*}
such that the corresponding elements $\P(\xi),\P(\eta),\P(\zeta)\in\PSLtC$
satisfy:
\begin{align*}
\P(\xi)   & = \iota_{\zeta\xi} \iota_{\xi\eta} \\
\P(\eta)  & = \iota_{\xi\eta}\iota_{\eta\zeta} \\
\P(\zeta) & = \iota_{\eta\zeta}\iota_{\zeta\xi}.
\end{align*} 
\end{theoremB}
From Theorem~A follows the identification of the
Fricke space of the three-holed sphere
in terms of trace coordinates as $\big(-\infty,-2]^3$.
The three trace parameters correspond to the three
boundary components of $\Sigma$.
From Theorem~B follows the identification of the
Fricke space of the three-holed sphere with the
space of (mildly degenerate) right-angled hexagons
in the hyperbolic plane $\Ht$. (Right-angled hexagons
are allowed to degenerate when some of the alternate
edges covering boundary components degenerate to ideal points.)

The condition 
$x^2 + y^2 + z^2 - x y z \neq 4$  means that 
$\langle\xi,\eta\rangle$ defines an irreducible representation
on $\C^2$. This is equivalent to the condition that
\begin{equation*}
  \tr [\xi,\eta] \neq 2.
\end{equation*}
Thus the commutator trace plays an important role, partially because
the fundamental group of the one-holed torus admits free generators
$X,Y$ such that the boundary component corresponds to $[X,Y]$.  In
particular trace coordinates identify the Fricke space of the
one-holed torus with
\begin{equation*}
\{ (x,y,z)\in \big(2,\infty) \mid x^2 + y^2 + z^2 - xyz \le 0 \}
\end{equation*}
where the boundary trace equals
\begin{equation*}
\tr [\xi,\eta] = x^2 + y^2 + z^2 - xyz  \le -2.
\end{equation*}
The trace coordinates are related to Fenchel-Nielsen coordinates.
Similar descriptions of the Fricke spaces of the two-holed
cross-surface (projective plane) and the one-holed Klein bottle are
also given.  

The character variety of $\F_3$ is more complicated.
Let $X_1,X_2,X_3$ be free generators. The traces of
the words 
\begin{equation*}
X_1, X_2, X_3,
X_1 X_2, X_1X_3, X_2X_3,
X_1 X_2 X_3, X_1X_3 X_2
\end{equation*}
generate the $\SLtC$-character ring of $\F_3$. 
We denote these functions by
\begin{equation*}
x_1, x_2, x_3,
x_{12}, x_{13}, x_{23},
x_{123}, x_{132}
\end{equation*}
respectively.
However, the character ring is
not a polynomial ring on these generators, due to the
trace identities expressing the {\em triple traces\/}
$x_{123}$ and $x_{132}$ as the roots of a monic
quadratic polynomial whose coefficients are polynomials
in the {\em single traces\/} $x_i$ and 
{\em double traces\/} $x_{ij}$:
\begin{align*}
x_{123} +  x_{132} & \;=\; x_{12}x_3 + x_{13}x_2 + x_{23}x_1 - x_1x_2x_3 \\
x_{123}\,\,\,  x_{132} & \;=\;  (x_1^2 + x_2^2 +x_3^2) \,+\,
(x_{12}^2 + x_{23}^2 \,+\, x_{13}^2) \  -  \\  
& \qquad \  ( x_1x_2x_{12} + x_2x_3x_{23} + x_3x_1x_{13} ) \,+\,
x_{12} x_{23} x_{13} \,-\, 4. 
\end{align*}
Furthermore the character variety is a hypersurface in $\C^7$
which is a double branched covering of $\C^6$. 
In particular its coordinate ring, the {\em character ring,\/}
is the quotient 
\begin{equation*}
\RR_3 := \C[x_1, x_2, x_3, x_{12}, x_{13}, x_{23},x_{123}]/\mathfrak{I}
\end{equation*}
by the principal ideal $\mathfrak{I}$ generated by the polynomial
\begin{align*}
\Phi(x_1, x_2, x_3,  & x_{12}, x_{13}, x_{23},x_{123})  \;:=\; \\
& x_1 x_2 x_3 x_{123}   + 
x_{12}x_{13}x_{23}  \\  
& \quad - x_1x_2x_{12} - x_1x_3x_{13} - x_2x_3x_{23} \\ &
\qquad - x_1 x_{23} x_{123}
- x_2 x_{13} x_{123}
- x_3 x_{12} x_{123}  \\ & 
\qquad \quad + x_1^2 + x_2^2 + x_3^2 + x_{12}^2 + x_{13}^2 +  x_{23}^2 \\ 
& \qquad \qquad + x_{123}^2 - 4 
\end{align*}

We use this description to discuss the Fricke spaces
of the 4-holed sphere $\Sigma_{0,4}$ and the 2-holed torus
$\Sigma_{1,2}$. In these cases, the generators $X_i$ and their
products correspond to curves on the surface, and we pay special
attention to the elements corresponding to the boundary $\partial \Sigma$.

In particular we describe the homomorphisms on character rings
induced by the orientable double coverings of the 2-holed cross-cap 
$C_{0,2}$ 
\begin{equation*}
\Sigma_{0,4}  \longrightarrow C_{0,2} 
\end{equation*}
and the $1$-holed Klein bottle $C_{1,1}$
\begin{equation*}
\Sigma_{1,2}  \longrightarrow C_{1,1} 
\end{equation*}
respectively.

Finally we end with the important observation (see Vogt~\cite{Vogt})
that the $\SLtC$-character ring $\RR_n$ of a free group $\F_n$ where $n\ge 4$,
is generated by traces of words of length $\le 3$. 

This chapter began as an effort~\cite{fricke}, to provide a 
self-contained exposition of Theorem~A. Later it grew to include
several results on hyperbolic geometry, which were used, for example
in \cite{puncturedtorus} but with neither adequate proofs nor 
references to the literature. In this version, we have tried to give a
leisurely and elementary description of basic results on moduli of
hyperbolic structures using trace coordinates.

\subsubsection*{Acknowlegements.}
I am grateful to Hyman Bass, Steve Boyer, Richard Brown, Serge Cantat,
Virginie Charette, Daryl Cooper, John Conway, Marc Culler, Todd Drumm,
Art Dupre,
Elisha Falbel, Carlos Florentino, Charlie Frohman, Alexander Gamburd,
Jane Gilman, Ryan Hoban, Misha Kapovich, Linda Keen, Fran\c cois
Labourie, Albert Marden, John Millson, Maryam Mirzakhani, Greg
McShane, Walter Neumann, John Parker, Julien Paupert, Bob Penner,
Peter Sarnak, Caroline Series, Peter Shalen, Adam
Sikora, George Stantchev, Ser-Peow Tan, Domingo Toledo, Richard
Wentworth, Anna Wienhard, Scott Wolpert and Eugene Xia for their
interest in this manuscript, and for providing valuable
suggestions. In particular I wish to thank Sean Lawton,
Elisha Peterson, 
and Ying Zhang
for carefully reading parts of this manuscript and pointing out
several technical mistakes.  I would like to thank Guillaume Th\'eret 
for help with the illustrations.  I also wish to express my gratitude
for the hospitality of the Mathematical Sciences Research Institute in
Fall 2007, and the Institute for Advanced Study and Princeton
University in Spring 2008, where this work was completed.

\subsubsection*{Notation and terminology.}
We mainly work over the field $\C$ of complex numbers and its subfield
$\R$ of real numbers.  Denote the ring of rational integers by $\Z$.
We denote projectivization by $\P$, so that if $V$ is a $\C$-vector
space (respectively an $\R$-vector space), then $\P(V)$ denotes the
set of all complex (respectively real) lines in $V$.  Similarly if
$V\xrightarrow{\xi} W$ is a linear transformation between vector
spaces $V,W$, denote the corresponding projective transformation by
$\P(\xi)$, wherever it is defined.  For example the {\em complex
projective line\/} $\CP^1 = \P(\C^2)$.  The noncommutative field of
Hamilton quaternions is denoted $\HH$.  The set of positive real
numbers is denoted $\R_+$.

Denote the algebra of $2\times 2$ matrices over $\C$ by $\Mat$. 

The trace and determinant functions are denoted $\tr$ and $\det$ respectively.
Denote the transpose of a matrix $A$ by $A^{\dag}$.

Let $\kk$ be a field (either $\R$ or $\C$).
Denote the multiplicative group of $\kk$ 
(the group of nonzero elements) by $\kk^*$.

Let $n>0$ be an integer.  The {\em general linear group\/} is denoted
$\GL{n,\kk}$; for example $\GLtC$ is the group of all invertible
$2\times 2$ complex matrices.  We also denote the group of scalar
matrices 
\begin{equation*}
\kk^*\Id\subset \GL{n,\kk} 
\end{equation*}
by $\kk^*$.  The {\em special linear group\/} consists of all matrices in $\GL{n,\kk}$ having
determinant one, and is denoted $\SL{n,\kk}$.  The {\em projective
linear groups\/} $\PGL{n,\kk}$ (and respectively $\PSL{n,\kk}$) are
the quotients of $\GL{n,\kk}$ (respectively $\SL{n,\kk}$) by the
central subgroup $\{\lambda\Id \mid \lambda \in\kk^*\}$ of scalar
matrices, which we also denote $\kk^*$. 

If $A,B$ are matrices, then their multiplicative commutator is 
denoted $[A,B] := ABA^{-1}B^{-1}$ and their additive commutator 
(their {\em Lie product\/}) is denoted $\Lie(A,B) := AB - BA$.

\renewcommand{\star}{\ast} 

If $A$ is a transformation, denote its set of fixed points
by $\Fix(A)$.
Denote the relation of conjugacy in a group by $\sim$.
Denote free product of two groups $A,B$ by $A\star B$.
If $a_1,\dots, a_n$ are elements of a group, then 
$\langle a_1,\dots a_n\rangle$ denotes the subgroup generated
by $a_1,\dots, a_n$. The presentation of a group with
generators $g_1,\dots g_m$ and relations $r_1(g_1,\dots g_m), \dots
r_n(g_1,\dots g_m)$ is denoted
\begin{equation*}
\langle  g_1,\dots g_m \mid r_1, \dots r_n \rangle.
\end{equation*}
Denote the free group of rank $n$ by $\F_n$.
Denote the symmetric group on $n$ letters by $\mathfrak{S}_n$.

Denote the (real) hyperbolic $n$-space by $\mathsf{H}^n$.

We briefly summarize the topology of surfaces.

A compact surface with $n$ boundary components will be called 
{\em $n$-holed.\/} If $M$ is a closed surface, then the complement
in $M$ of $n$ open discs will be called an {\em ``$n$-holed $M$.''\/} 
For example a $1$-holed sphere is a disc and a $2$-holed sphere
is an annulus.

We adopt the following notation for topological types of connected compact
surfaces, beginning with orientable surfaces.
$\Sigma_{g,n}$ denotes the $n$-holed (orientable) surface of genus $g$.
Thus $\Sigma_{0,0}$ is a sphere, $\Sigma_{1,0}$ is a torus,
$\Sigma_{0,1}$ is a disc and $\Sigma_{0,2}$ is an annulus.  

The connected sum operation $\#$ satisifies:
\begin{equation*}
\Sigma_{g_1,n_1} \#  \Sigma_{g_2,n_2} \approx 
\Sigma_{g_1+ g_2,n_1+n_2}.
\end{equation*}
Other basic facts about orientable surfaces involve 
the Euler characteristic and the fundamental group:
\begin{equation*}
\chi(\Sigma_{g,n}) = 2 - 2g - n 
\end{equation*}
and if $n>0$, the fundamental group $\pi_1(\Sigma_{g,n})$ is free of rank $2g+ n - 1$.

For non-orientable surfaces, our starting point is the 
topological 
surface $C_{0,0}$ homeomorphic to the real projective plane, 
which J.\ H.\ Conway has proposed calling a {\em cross-surface.\/}
\index{cross-surface}
We denote the $n$-holed $k+1$-fold connected sum of cross-surfaces
by $C_{k,n}$. Thus the M\"obius band is represented by $C_{0,1}$ and the
Klein bottle by 
\begin{equation*}
C_{1,0} \approx C_{0,0}\# C_{0,0} 
\end{equation*}
The operation of connected sum satisfies:
\begin{align*}
\Sigma_{g,n_1}\# C_{k,n_2} &\approx C_{2g+k,n_1+n_2} \\ 
C_{k_1,n_1}\# C_{k_2,n_2} &\approx C_{k_1+k_2+1,n_1+n_2}.
\end{align*}
The Euler characteristic and the fundamental group satisfy:
\begin{equation*}
\chi(C_{k,n}) = 1 - n - k 
\end{equation*}
and $\pi_1(C_{k,n})$ is free of rank $n+k$ if $n > 0$.

The orientable double covering space of $C_{g,n}$ is $\Sigma_{g,2n}$.
\index{topology of surfaces}
\section{Traces in $\SLtC$}\index{traces}

The purpose of this section is 
an elementary and relatively
self-contained proof of Theorem~A.
This basic result explicitly describes the $\SLtC$-character variety of 
a rank-two free group as the affine space $\C^3$, parametrized
by the traces of the free generators $X,Y$ and the trace of their
product $XY$.
Apparently due to Vogt~\cite{Vogt}, it is also in the work of 
Fricke~\cite{Fricke} and Fricke-Klein~\cite{FrickeKlein}.

We motivate the discussion by starting with the simpler case of conjugacy 
classes of single elements, that is  
{\em cyclic groups\/}  (free groups of {\em rank one\/}). 
In this case the $\SLtC$-character variety $V_1$ is the {\em affine line\/} $\C^1$,
parametrized by the trace.

\subsection{Cyclic groups}\index{cyclic groups.}
\begin{theorem}
Let $\SLtC\xrightarrow{f}\C$ be a polynomial function invariant under
inner automorphisms of $\SLtC$. Then there exists a polynomial $F(t)\in\C[t]$
such that $f(g) = F(\tr(g))$. Conversely, if $g,g'\in \SLtC$ satisfy
\begin{equation*}
\tr(g) = \tr(g') \neq \pm 2, 
\end{equation*}
then $g' = h g h^{-1}$ for some $h\in \SLtC$.
\end{theorem}
\begin{proof}
Suppose $f$ is an invariant function. For $t\in\C$, define
\begin{equation*}
\xi_t := \bmatrix t & -1 \\ 1 & 0 \endbmatrix 
\end{equation*}
and define $F(t)$ by
\begin{equation*}
F(t) = f(\xi_t).
\end{equation*}
Suppose that $t\neq\pm 2$ and $\tr(g) = t$. 
Then 
$g$ and $\xi_t$ each have distinct eigenvalues 
\begin{equation*}
\lambda_{\pm} = \frac12 \left( t \pm ( t^2 - 4)^{1/2} \right)
\end{equation*}
and $hgh^{-1} = \xi_t$ for some $h\in \SLtC$.
Thus 
\begin{equation*}
f(g) = f(h^{-1}\xi_th) = f(\xi_t) = F(t) 
\end{equation*}
as desired. If $t=\pm 2$, then by taking Jordan normal
form, either $g=\pm\Id$ or $g$ is conjugate to $\xi_t$. In the latter case,
$f(g) = F(t)$ follows from invariance. Otherwise $g$ lies in the closure
of the $\SLtC$-orbit of $\xi_t$ and $f(g) = f(\xi_t) = F(t)$ follows by
continuity of $f$.

The converse direction follows from Jordan normal form as already used above.
\end{proof}
The map
\begin{equation*}
\SLtC \xrightarrow{\tr} \C 
\end{equation*}
is a {\em categorical\/} quotient map in the sense of algebraic geometry,
although it fails to be a quotient map in the usual sense. 
The 
discrepancy
occurs at the critical level sets $\tr^{-1}(\pm 2)$.
The critical values of $\tr$ are $\pm 2$, and the restriction of
$\tr$ to the regular set
\begin{equation*}
\tr^{-1}( \C \setminus \{\pm 2\}) 
\end{equation*}
is a quotient map (indeed a holomorphic submersion). The critical
level set $\tr^{-1}(2)$ consists of all unipotent matrices, and these
are conjugate to the one-parameter subgroup
\begin{equation*}
\bmatrix 1 & t \\ 0 & 1 \endbmatrix 
\end{equation*}
where $t\in\C$. For $t\neq 0$, these matrices comprise a single orbit. 
This orbit does {\em not\/} contain the identity matrix $\Id$ (where $t = 0$),
although its closure does.
Any regular function cannot separate a non-identity unipotent
matrix from $\Id$. Thus $\tr^{-1}(2)$ contains two orbits: the non-identity
unipotent matrices, and the identity matrix $\Id$. Similar remarks apply to
the other critical level set  $\tr^{-1}(-2) = -\tr^{-1}(2)$.

For example, 
\begin{align*}
\SLtC & \longrightarrow \C \\
\xi & \longmapsto \tr(\xi^2)
\end{align*}
is an invariant function and can be expressed in terms of $\tr(\xi)$ by:
\begin{equation}\label{eq:traceofsquare}
\tr(\xi^2) = \tr(\xi)^2 - 2
\end{equation}
which follows from the Cayley-Hamilon theorem 
(see \eqref{eq:CH} below) by taking traces.

\subsection{Two-generator groups.}
We begin by recording the first (trivial)
normalization for computing traces:
\begin{equation}\label{eq:trid}
\tr(\Id) = 2.
\end{equation}
This will be the first of three properties of the trace
function which enables the computation of traces of
arbitrary words in elements of $\SLtC$.

\subsubsection*{The Cayley-Hamilton theorem.}
\index{Cayley-Hamilton theorem}

If $\xi$ is a $2\times 2$-matrix,
\begin{equation}\label{eq:CH}
\xi^2 - \tr(\xi) \xi + \det(\xi)\Id = 0. 
\end{equation}
Suppose $\xi,\eta\in\SLtC$.
Multiplying \eqref{eq:CH} by $\xi^{-1}$ and rearranging,
\begin{equation}\label{eq:SumInvo}
\xi + \xi^{-1} = \tr(\xi) \Id 
\end{equation}
from which follows (using \eqref{eq:trid}):
\begin{equation}\label{eq:inverse}
\tr(\xi) =  \tr(\xi^{-1}).
\end{equation}
Multiplying \eqref{eq:SumInvo} by $\eta$ and taking traces, we obtain
(switching $\xi$ and $\eta$):
\begin{theorem}[The Basic Identity]
Let $\xi,\eta\in\SLtC$. Then
\begin{equation}\label{eq:basic}
\tr(\xi\eta) + \tr(\xi\eta^{-1})  
= \tr(\xi)\tr(\eta).
\end{equation}
\end{theorem}
As we shall see, the three identities
\eqref{eq:trid},\eqref{eq:SumInvo} and 
\eqref{eq:inverse}
apply to
compute the trace of any word $w(\xi,\eta)$ for $\xi,\eta\in\SLtC$.

\subsubsection*{Traces of reduced words: an algorithm.}\label{sec:wordtraces}
\index{algorithm for $\SLtC$-traces of reduced words in $\F_2$}
Here is an important special case of 
Theorem~A. 
Namely, let $w(X,Y)\in\pi$
be a reduced word. Then
\begin{align*}
\SLtC \times \SLtC & \longrightarrow \C \\
(\xi,\eta) & \longmapsto \tr\big(w(\xi,\eta)\big)
\end{align*}
is an $\SLtC$-invariant function on $\SLtC\times\SLtC$. 
Theorem~A guarantees
a polynomial
\begin{equation*}
f_w(x,y,z)\in\C[x,y,z] 
\end{equation*}
such that 
\begin{equation}\label{eq:fw}
\tr\big(w(\xi,\eta)\big) \;=\; f_w\big(\tr(\xi),\tr(\eta),\tr(\xi\eta)\big)
\end{equation}
for all $\xi,\eta\in \SLtC$. We describe an algorithm for computing
$f_w(x,y,z)$. For notational convenience we write 
\begin{equation*}
\tr\big(w(\xi,\eta)\big) := f_{w(X,Y)}(x,y,z). 
\end{equation*}
For example,
\begin{align*}
\tr(\Id) & = 2, \\
\tr(\xi^{-1})  = \tr(\xi) & = x, \\
\tr(\eta^{-1})  = \tr(\eta) & = y,
\end{align*}
verifying assertion \eqref{eq:fw} for words $w$ of length $\ell(w)\le 1$.
For symmetry, we write $Z = Y^{-1} X^{-1}$,
so that $X,Y,Z$ satisfy the relation
\begin{equation*}
X Y Z = \Id.
\end{equation*}
(For a geometric interpretation of this presentation
in terms of the three-holed sphere $\Sigma_{0,3}$,
compare \S\ref{sec:threeholedsphere}.)
Write $\zeta = (\xi\eta)^{-1}$ so that
$\xi\eta \zeta = \Id$. Then 
\begin{align*}
\tr(\xi\eta)   = \tr(\eta\xi)  & =  \tr(\xi^{-1}\eta^{-1})   =  \\
\tr(\eta^{-1}\xi^{-1}) & =  \tr(\zeta) 
  = \tr(\zeta^{-1})  = z,
\end{align*}

The reduced words of length two are 
\begin{align*}
X^2,Y^2,XY, & XY^{-1},YX,YX^{-1}, \\
X^{-2},Y^{-2},X^{-1}Y^{-1}, &X^{-1}Y,Y^{-1}X^{-1},Y^{-1}X.
\end{align*}
As mentioned above, the trace of a square \eqref{eq:traceofsquare} follows immediately
by taking the trace of \eqref{eq:CH}. Thus:
\begin{align*}
\tr(\xi^2) & =  x^2 - 2 \\
\tr(\eta^2) & =  y^2 - 2 \\
\tr\big((\xi\eta)^2\big) & =  z^2 - 2 
\end{align*}
Further applications of the trace identities imply:
\begin{align*}
\tr(\xi\eta^{-1}) & =  xy - z \\
\tr\big(\eta(\xi\eta)\big) = \tr(\eta\zeta^{-1}) & =  yz - x \\
\tr\big((\xi\eta)^{-1})\xi^{-1}\big) = \tr(\zeta\xi^{-1}) & =  zx - y 
\end{align*}
For example, taking $w(X,Y) = XY^{-1}$,
\begin{equation*}
\tr(\xi\eta^{-1}) = \tr(\xi)\tr(\eta) - \tr(\xi\eta).
\end{equation*}
Furthermore
\begin{align}\label{eq:xyxiy}
\tr(\xi\eta\xi^{-1}\eta) & = \tr(\xi\eta) \tr(\xi^{-1}\eta) - \tr(\xi^2)  \\ 
& = z(xy - z) - (x^2 -2) \notag \\ 
&  = 2 - x^2  - z^2  + xyz. \notag
\end{align}
An extremely important example is the {\em commutator word\/}
\begin{equation*}
k(X,Y) :=  X Y X^{-1} Y^{-1}.
\end{equation*}
Computation of its trace polynomial $\kappa = f_k$ 
follows easily from applying \eqref{eq:basic} to \eqref{eq:xyxiy}:
\begin{align*}
\tr(\xi\eta\xi^{-1}\eta^{-1}) & =  
\tr(\xi\eta\xi^{-1}) \tr(\eta) - \tr(\xi\eta\xi^{-1}\eta) \\ & =  
y^2 - (2 - x^2  - z^2  + xyz) \\ & = x^2 + y^2 + z^2 - x y z - 2
\end{align*}
whence
\begin{equation}
\label{eq:commutator}
\kappa(x,y,z) = f_k(x,y,z) =
x^2 + y^2 + z^2 - x y z - 2.
\end{equation}
\index{commutator trace}

Assume inductively that for all reduced words $w(X,Y)\in\pi$ with
$\ell(w)<m$, 
there exists a polynomial $f_w(x,y,z) = \tr(w(\xi,\eta))$ satisfying \eqref{eq:fw}.
Suppose that $u(X,Y)\in\F_2$ is a reduced word of length $\ell(u) = m$.

The explicit calculations above begin the induction for $m\le 2$. 
Thus we assume $m>2$.

Furthermore, we can assume that $u$ is {\em cyclically reduced,\/} that is
the initial symbol of $u$ is not inverse to the terminal symbol of $u$.
For otherwise 
\begin{equation*}
u(X,Y) = S u'(X,Y) S^{-1}, 
\end{equation*}
where $S$ is one of the four symbols 
\begin{equation*}
X, Y,X^{-1},Y^{-1} 
\end{equation*}
and $\ell(u') = m-2$. Then $u(X,Y)$ and $u'(X,Y)$ are conjugate
and 
\begin{equation*}
\tr(u(X,Y)) = \tr(u'(X,Y)). 
\end{equation*}

If $m>2$, and $u$ is cyclically reduced,
then $u(X,Y)$ has a repeated letter,
which we may assume to
equal $X$. That is, we may write, after conjugating
by a subword,
\begin{equation*}
u(X,Y) = u_1(X,Y) u_2(X,Y) 
\end{equation*}
where $u_1$ and $u_2$ are reduced words each ending in $X^{\pm 1}$. 
Furthermore we may assume that 
\begin{equation*}
\ell(u_1) + \ell(u_2) = \ell(u) = m, 
\end{equation*}
so that $\ell(u_1) < m$ and $\ell(u_2) < m$.
Suppose first that $u_1$ and $u_2$ both end in $X$. Then 
\begin{equation*}
u(X,Y) = \left( u_1(X,Y)X^{-1}\right)X\ 
\left( u_2(X,Y)X^{-1}\right)X 
\end{equation*}
and each of 
\begin{equation*}
u_1(X,Y)X^{-1},\;  u_2(X,Y)X^{-1} 
\end{equation*}
has a terminal $XX^{-1}$, 
which we cancel to obtain corresponding reduced
words $u_1'(X,Y), u_2'(X,Y)$ respectively
with 
\begin{equation*}
\ell(u_i'), \ell(u_i) 
\end{equation*}
for $i=1,2$ and
\begin{equation*}
u(X,Y) = u_1(X,Y) u_2(X,Y) =  u_1'(X,Y) X u_2'(X,Y) X 
\end{equation*}
in $\F_2$.
Then 
\begin{equation*}
(u_1(X,Y)X^{-1})(u_2(X,Y)X^{-1})^{-1} = 
u_1'(X,Y)u_2'(X,Y)^{-1}
\end{equation*}
is represented by a reduced word $u_3(X,Y)$ 
satisfying $\ell(u_3) < m$.
By the induction hypothesis, there exist polynomials
\begin{equation*}
f_{u_1(X,Y)},\, f_{u_2(X,Y)},\, f_{u_3(X,Y)} \;\in\; \C[x,y,z]
\end{equation*}
such that, for all $\xi,\eta\in \SLtC$, $i=1,2,3$,
\begin{equation*}
\tr\big(u_i(\xi,\eta)\big) = 
f_{u_i(X,Y)}\big(\tr(\xi),\tr(\eta),\tr(\xi\eta)\big).
\end{equation*}
By \eqref{eq:basic}, 
\begin{equation*}
f_u = f_{u_1} f_{u_2} - f_{u_3}
\end{equation*}
is a polynomial in $\C[x,y,z]$.
The cases when $u_1$ and $u_2$ both end in the symbols $X^{-1}, Y, Y^{-1}$ 
are completely analogous.
Since there are only four symbols, the only cyclically reduced words without
repeated symbols are commutators of the symbols, for example 
$XYX^{-1}Y^{-1}$. Repeated applications of the trace identities 
evaluate this trace polynomial as $\kappa(x,y,z)$ 
defined in \eqref{eq:commutator}.
The other commutators of distinct symbols also have trace $\kappa(x,y,z)$
by identical arguments.


\subsubsection*{Surjectivity of characters of pairs: a normal form.}
\index{surjectivity of character map for $\F_2$}
We first show that
\begin{align*}
\tau: \SLtC\times\SLtC & \longrightarrow \C^3 \\
(\xi,\eta) & 
\longmapsto \bmatrix \tr(\xi) \\  \tr(\eta) \\ \tr(\xi\eta)\endbmatrix 
\end{align*}
is surjective.
Let $(x,y,z)\in\C^3$.
Choose $\zz\in\C$ so that
\begin{equation*}
\zz + \zz^{-1} = z, 
\end{equation*}
that is, $\zz = \frac12 (z \pm \sqrt{z^2 - 4})$.
Let
\begin{equation}\label{eq:explicitrep}
\xi_x = \bmatrix x &  -1  \\ 1 & 0 \endbmatrix,\;
\eta_{(y,\zz)} = \bmatrix 0 &  \zz^{-1} \\ -\zz & y \endbmatrix.
\end{equation}
Then $\tau(\xi_x,\eta_{(y,\zz)}) = (x,y,z)$.

Next we show that every $\SLtC$-invariant regular function
\begin{equation*}
\SLtC\times\SLtC\xrightarrow{f}\C 
\end{equation*}
factors through $\tau$.
To this end we need the following elementary lemma on symmetric functions:

\begin{lemma}\label{lem:even}
Let $R$ be an integral domain where $2$ is invertible,
and let $R' = R[\zz,\zz^{-1}]$ be the ring of Laurent
polynomials over $R$. Let $R'\xrightarrow{\sigma} R'$ be the involution
which fixes $R$  
and interchanges $\zz$ and $\zz^{-1}$. Then the subring 
of $\sigma$-invariants is the polynomial ring $R[\zz+\zz^{-1}]$.
\end{lemma}
\begin{proof}
Let $F(\zz,\zz^{-1})\in R[\zz,\zz^{-1}]$ be a $\sigma$-invariant
Laurent polynomial. 
Begin by rewriting $R'$ as the quotient of the polynomial ring
$R[x,y]$ by the ideal generated by $xy-1$. Then $\sigma$ is induced by
the involution $\tilde\sigma$ of $R[x,y]$ interchanging $x$ and $y$.
Let $f(x,y)\in R[x,y]$ be a polynomial whose image in $R'$ is $F$.
Then there exists a polynomial $g(x,y)$ such that 
\begin{equation*}
f(x,y) - f(y,x) = g(x,y) (xy - 1). 
\end{equation*}
Clearly $g(x,y) = - g(y,x)$. 
Let 
\begin{equation*}
\tilde f(x,y) = f(x,y) - \frac12\; g(x,y) (x y - 1)
\end{equation*}
so that $\tilde f(x,y)  = \tilde f(y,x)$. By the theorem on elementary
symmetric functions,
\begin{equation*}
\tilde f(x,y) = h(x + y, xy) 
\end{equation*}
for some polynomial $h(u,v)$.
Therefore $F(\zz,\zz^{-1}) = h(\zz + \zz^{-1},1)$ as desired.
\end{proof}
\noindent
By definition $f(\xi,\eta)$ is a polynomial in the matrix entries of
$\xi$ and $\eta$; regard two polynomials differing by elements in the
ideal generated by $\det(\xi)-1$ and $\det(\eta)-1$ as equal. 
Thus $f(\xi_x,\eta_{(y,\zz)})$ equals a function $g(x,y,\zz)$
which is a polynomial in $x,y\in\C$ and a Laurent polynomial in
$\zz\in\C^*$,   
"where $\xi_x$ and $\eta_{(y,\zz)}$ were defined 
in \eqref{eq:explicitrep}.

\begin{lemma}\label{lem:commonperp}
Let $\xi,\eta\in \SLtC$ such that $\kappa(\tau(\xi,\eta))\neq 2$. 
Then there exists
$h\in \SLtC$ such that 
\begin{equation*}
h\cdot (\xi,\eta) =  (\xi^{-1},\eta^{-1}).
\end{equation*}
\end{lemma}
\begin{proof}
Let $(x,y,z) = \tau(\xi,\eta)$. By  
the commutator trace formula
\eqref{eq:commutator}, 
\begin{equation*}
\tr [\xi,\eta] = \kappa(x,y,z) 
\end{equation*}
where $[\xi,\eta] = \xi\eta\xi^{-1}\eta^{-1}$.

Let $L = \xi\eta-\eta\xi$. 
(Compare \S 4 of J\o rgensen~\cite{J} or Fenchel~\cite{Fenchel}.)
Then 
\begin{equation*}
\tr(L) = \tr(\xi\eta) -\tr(\eta\xi) = 0.  
\end{equation*}
Furthermore
for any $2\times 2$ matrix $M$, the characteristic polynomial
\begin{equation*}
\lambda_M(t) := \det(t\Id - M) = t^2 - \tr(M) t + \det(M).
\end{equation*}
Thus
\begin{align*}
\det(L) & = \det( [\xi,\eta] -\Id)\det(\eta\xi) \\ & = \det([\xi,\eta]-\Id) \\ 
& = - \lambda_{[\xi,\eta]}(1) \\ 
& = - 2 + \tr [\xi,\eta] \\ 
& = - 2 + \kappa(x,y,z) \neq 0.
\end{align*}
Choose $\mu\in\C^*$ such that $\mu^2 \det(L) = 1$ and let $h = \mu L\in \SLtC$.

Since $\tr(h) = 0$ and $\det(h) =1$, the Cayley-Hamilton Theorem  
\begin{equation*}
\lambda_M(M) = 0  
\end{equation*}
implies that $h^2 = -\Id$. Similarly
\begin{align*}
\det(h\xi) &= \det(h) = 1 \\ 
\det(h\eta) &= \det(h) = 1, 
\end{align*}
and 
\begin{align*}
\tr(h\xi) & = \mu (\tr((\xi\eta)\xi) - \tr((\eta\xi)\xi)) \\
& = \mu (\tr(\xi(\eta\xi)) - \tr((\eta\xi)\xi)) = 0 
\end{align*}
and 
\begin{align*}
\tr(h\eta) & = \mu (\tr((\xi\eta)\eta) - \tr((\eta\xi)\eta)) \\ 
& = \mu (\tr((\xi\eta)\eta) - \tr(\eta(\xi\eta)) = 0
\end{align*}
so $(h\xi)^2 = (h\eta)^2 = -\Id$. Thus
\begin{equation*}
h \xi h^{-1} \xi = - h\xi h\xi = \Id 
\end{equation*}
whence $h \xi h^{-1} = \xi^{-1}.$ 
Similarly $h \eta h^{-1} = \eta^{-1}$, 
concluding the proof of the lemma. \end{proof}
\index{common orthogonal geodesic}

Apply Lemma~\ref{lem:commonperp} 
to $\xi = \xi_x$ and $\eta = \eta_{(y,\zz)}$ as above to obtain
$h$ such that conjugation by $h$ maps
\begin{equation*}
\xi  \longmapsto
\xi^{-1} =  \bmatrix 0 &  1  \\ -1 & x \endbmatrix
\end{equation*}
and 
\begin{equation*}
\eta  \longmapsto
\eta^{-1} =  \bmatrix y &  - 1/\zz  \\ \zz & 0 \endbmatrix.
\end{equation*}
If 
\begin{equation*}
u = \bmatrix 0 & 1 \\ 1 & 0 \endbmatrix,
\end{equation*}
then 
\begin{equation*}
uh \xi (uh)^{-1} = u \xi^{-1} u^{-1} =  
\bmatrix x &  -1  \\ 1 & 0 \endbmatrix = \xi
\end{equation*}
and 
\begin{equation*}
uh \eta (uh)^{-1} = 
u \eta^{-1} u^{-1} =  \bmatrix 0 &  \zz  \\ -1/\zz & y \endbmatrix.
\end{equation*}
Thus
\begin{align*}
g(x,y,\zz) & =  f(\xi,\eta) \\ & 
= f(uh \xi (uh)^{-1},uh \eta (uh)^{-1}) \\ 
& =  g(x,y,\zz^{-1}).
\end{align*}
Lemma~\ref{lem:even} 
implies that 
\begin{equation}\label{eq:even}
g(x,y,\zz) = F(x,y,\zz + 1/\zz)
\end{equation}
for some polynomial $F(x,y,z)\in\C[x,y,z]$, 
whenever $\kappa(x,y,\zz+1/\zz)\neq 2$.
Since this condition defines a nonempty Zariski-dense open set,
\eqref{eq:even} holds on all of $\C^2\times\C^*$ and
\begin{equation*}
f(\xi,\eta) = F(\tr(\xi),\tr(\eta), \tr(\xi\eta)) 
\end{equation*}
as claimed.

\subsubsection*{Injectivity of $\SLtC$-characters of pairs.}
\index{injectivity of character map for $\F_2$}

Finally we show that if $(\xi,\eta),(\xi',\eta')\in H$ 
satisfy
\begin{equation}\label{eq:equalchars}
\bmatrix \tr(\xi) \\ \tr(\eta) \\ \tr(\xi\eta) \endbmatrix = 
\bmatrix \tr(\xi') \\ \tr(\eta') \\ \tr(\xi'\eta') \endbmatrix = 
\bmatrix x \\ y \\ z \endbmatrix,  
\end{equation}
and $\kappa(x,y,z)\neq 2$, then $(\xi,\eta)$ and $(\xi',\eta')$ 
are $\SLtC$-equivalent. 
By \S\ref{sec:wordtraces}, the triple 
\begin{equation*}
\bmatrix x \\ y \\ z\endbmatrix  = 
\bmatrix \tr(\xi) \\ \tr(\eta) \\ \tr (\xi\eta)\endbmatrix 
\end{equation*}
determines the character function
\begin{align*}
\pi &\longrightarrow \C  \\
w(X,Y) & \longmapsto \tr\big( w(\xi,\eta)\big) = f_w(x,y,z).
\end{align*}
Let $\rho$ and $\rho'$ denote the representations $\pi\longrightarrow \SLtC$
taking $X,Y$  to $\xi,\eta$ and $\xi',\eta'$ respectively
and let $\chi,\chi'$ denote their respective characters. 
Then our hypothesis 
\eqref{eq:equalchars} implies that $\chi=\chi'$.

\subsection{Injectivity of the character map: the general case.}

The conjugacy of representations (one of which is irreducible) 
having the same character 
follows from a general argument using the Burnside theorem.
I am grateful to Hyman Bass~\cite{Bass} for explaining this to me.

Suppose $\rho$ and $\rho'$ are irreducible representations on $\C^2$.
Burnside's Theorem (see Lang~\cite{Lang}, p.445) implies 
the corresponding representations 
(also denoted $\rho, \rho'$ respectively)
of the group algebra $\C\pi$ into $\Mat$ are surjective. 
Since the trace form
\begin{align*}
\Mat \times  \Mat &\longrightarrow \C \\
(A,B) & \longmapsto \tr(AB)
\end{align*}
is nondegenerate, the kernel $K$ of $\C\pi\xrightarrow{\rho} \Mat$
consists of all 
\begin{equation*}
\sum_{\alpha\in\pi} a_\alpha \alpha \in \C\pi
\end{equation*}
such that 
\begin{align*}
0 & = \tr \Bigg( \Big(\sum_{\alpha\in\pi} a_\alpha \rho(\alpha)\Big) 
\rho(\beta) \Bigg) \\ 
& = \sum_{\alpha\in\pi} a_\alpha\, \tr\big(\rho(\alpha\beta)\big) \\
& = \sum_{\alpha\in\pi} a_\alpha\, \chi(\alpha\beta)
\end{align*}
for all $\beta\in\pi$.
Thus the kernels of both representations of $\C\pi$ are equal, and
$\rho$ and $\rho'$ respectively induce algebra isomorphisms
\begin{equation*}
\C\pi/K \longrightarrow \Mat,
\end{equation*}
denoted $\tilde\rho, \tilde\rho'$.

The composition $\tilde\rho'\circ\tilde\rho^{-1}$ is an
automorphism of the algebra $\Mat$, which must be induced by
conjugation by $g\in\GLtC$. 
(See, for example, Corollary 9.122, p.734
of Rotman~\cite{Rotman}.) 
In particular 
$\rho'(\gamma) = g \rho(\gamma)g^{-1}$
as desired. 

\subsubsection*{Irreducibility.}
\index{irreducible representations}
The theory is significantly different for reducible representations.
Representations 
\begin{equation*}
\rho_1,\rho_2\in\Hom(\pi,\SLtC) 
\end{equation*}
are {\em equivalent\/} 
$\Longleftrightarrow$ they define the same point in the character variety,
that is, for all regular functions $f$ in the character ring, 
\begin{equation*}
f(\rho_1) = f(\rho_2). 
\end{equation*}
If both are irreducible, then $\rho_1$ and $\rho_2$ are conjugate.
Closely related is the fact that the conjugacy class of an irreducible
representation is closed.
Here are several equivalent conditions for irreducibility of two-generator 
subgroups of $\SLtC$:

\begin{proposition}\label{prop:irreducibility}
Let $\xi,\eta\in\SLtC$. The following are equivalent:
\begin{enumerate}
\item $\xi,\eta$ generate an irreducible representation on $\C^2$;\label{item:irr}
\item $\tr(\xi\eta\xi^{-1}\eta^{-1}) \neq 2$;\label{item:trcomm}
\item $\det(\xi\eta-\eta\xi) \neq 0$;\label{item:detLie}
\item The pair $(\xi,\eta)$ is not $\SLtC$-conjugate to a representation by
upper-triangular matrices
\begin{equation*}
\bmatrix a & b \\ 0 & a^{-1} \endbmatrix, 
\end{equation*}
where $a\in\C^*, b\in\C$;\label{item:uppertri}
\item Either the group $\langle \xi,\eta\rangle$ is not solvable, or 
there exists a decomposition 
\begin{equation*}
\C^2 = L_1 \oplus L_2 
\end{equation*}
into an {\em invariant pair\/} of lines $L_i$ such that one of $\xi,\eta$ interchanges
$L_1$ and $L_2$;\label{item:solv}
\item\label{item:algebrabasis}
$\{\Id, \xi, \eta, \xi\eta\}$ is a basis for $\Mat$.
\end{enumerate}
\end{proposition}
\noindent
In the next section we will find a further condition
(Theorem~\ref{thm:factor})
involving extending the representation to a representation
of the free product $\Z/2\star\Z/2\star\Z/2$.
\begin{proof}
The equivalence \eqref{item:irr} $\Longleftrightarrow$ \eqref{item:trcomm}
is due to 
Culler-Shalen~\cite{CullerShalen}. 
For completeness we give the proof here.

To prove \eqref{item:trcomm}$\Longrightarrow$\eqref{item:irr}, 
suppose that $\rho$ is reducible.
If $\xi,\eta$ generate a representation
with an invariant subspace of $\C^2$ of dimension one, this representation
is conjugate to one in which $\xi$ and $\eta$ are 
upper-triangular. 
Denoting their diagonal entries by $a,a^{-1}$ and $b,b^{-1}$ respectively,
the diagonal entries of $\xi\eta$ are 
$ab,a^{-1}b^{-1}$.
Thus
\begin{align*}
x & = a + a^{-1}, \\ y & = b  + b^{-1}, \\
z & = a b + a^{-1} b^{-1}.
\end{align*}
By direct computation, $\kappa(x,y,z) = 2$. 

To prove \eqref{item:irr}$\Longrightarrow$\eqref{item:trcomm}, 
suppose that $\kappa(x,y,z) = 2$. 
Let $\mathfrak{A}\subset \Mat$ denote the linear span of
$\Id,\xi,\eta,\xi\eta$. 
Identities derived from the Cayley-Halmilton theorem
\eqref{eq:CH} such as \eqref{eq:SumInvo} imply that
$\mathfrak{A}$ is a subalgebra of $\Mat$. 
For example, $\xi^2$ equals the linear combination 
\begin{equation}\label{eq:square}
\xi^2 = -\Id + x \xi 
\end{equation}
and
\begin{equation}\label{eq:commutation}
\eta\xi = (z - xy) \Id +  y \xi + x \eta - \xi\eta.
\end{equation}
The latter identity follows by writing
\begin{equation*}
\xi^{-1} \eta + \eta^{-1}\xi = \tr(\xi^{-1}\eta)\Id = 
(x y - z)\Id
\end{equation*}
and summing
\begin{align*}
x\eta & = \xi\eta + \xi^{-1}\eta \\
y\xi & = \eta\xi + \eta^{-1}\xi 
\end{align*}
to obtain:
\begin{equation*}
\xi \eta + \eta \xi = (z - x y) \Id + x\eta + y\xi
\end{equation*}
as desired.

In the basis of $\Mat$ by elementary matrices, the map
\begin{align*}
\C^4 & \longrightarrow \Mat \\
\bmatrix x_1 \\ x_2 \\ x_3 \\ x_4 \endbmatrix & \longmapsto
x_1 \Id + x_2 \xi + x_3 \eta + x_4 \xi\eta
\end{align*}
has determinant $2-\kappa(x,y,z) = 0$ and is not surjective. 
Thus $\mathfrak{A}$ is a proper subalgebra of $\Mat$ and 
the representation is reducible, as desired.

\eqref{item:trcomm} $\Longleftrightarrow$ 
\eqref{item:detLie} follows from the suggestive
formula, valid for $\xi,\eta\in\SLtC$,
\begin{equation}\label{eq:twokindsofcommutator}
\tr(\xi\eta\xi^{-1}\eta^{-1}) + \det(\xi\eta-\eta\xi) = 2,
\end{equation}
whose proof is left as an exercise.

The equivalence \eqref{item:irr}$\Longleftrightarrow$ \eqref{item:uppertri}
is essentially the definition of reducibility. If $L\subset\C^2$ is an invariant
subspace, then conjugating by a linear automorphism which maps $L$  to the
first coordinate line $\C\times\{0\}$ makes the representation upper triangular.

\eqref{item:uppertri}$\Longleftrightarrow$ \eqref{item:solv}
follows from the classification of solvable subgroups
of $\SLtC$: a solvable subgroup is either conjugate to a group of 
upper-triangular matrices, 
or is conjugate to a {\em dihedral representation,\/} 
where one of $\xi,\eta$ is a diagonal matrix
and the other is the {\em involution\/}
\begin{equation*}
i \bmatrix 0 & 1 \\ 1 & 0 \endbmatrix
\end{equation*}
(where the coefficient $i$ is required for unimodularity).
A dihedral representation is one which interchanges an invariant pair  
of lines although the lines themselves are not invariant.
For a descripton of these representations in terms of hyperbolic geometry,
see \S\ref{sec:dihedral}.

\eqref{item:irr}$\Longleftrightarrow$ \eqref{item:algebrabasis}
follows from the Burnside lemma, and identities such as
\eqref{eq:square} and \eqref{eq:commutation} to express 
products 
of $\Id,\xi,\eta,\xi\eta$ with the generators $\xi,\eta$ as linear
combinations of $\Id,\xi,\eta,\xi\eta$.
\end{proof}
\index{irreducible representations}

\section{Coxeter triangle groups in hyperbolic $3$-space} \label{sec:Coxeter}
\index{Coxeter triangle group}

An alternate geometric approach to the algebraic parametrization using
traces involves right-angled hexagons in $\Hth$. 
Specifically, a marked 2-generator group corresponds to an ordered triple
of lines in $\Ht$, no two of which are asymptotic. This triple completes
to a right-angled hexagon by including the three common orthogonal lines.
We use this geometric construction to identify, in terms of traces,
which representations correspond to geometric structures on surfaces.
However, since the trace is only defined on $\SLtC$, and not on $\PSLtC$,
we must first discuss the conditions which ensure that a representation
into $\PSLtC$ lifts to $\SLtC$.

\subsection{Lifting representations to $\SLtC$.}
The group of orientation-preserving 
isometries of $\Hth$ identifies with $\PSLtC$, which is doubly covered
by $\SLtC$. In general, a representation $\Gamma\longrightarrow\PSLtC$ may or may not 
lift to a representation to $\SLtC$. Clearly if $\Gamma$ is a free group, every
representation lifts, since lifting each generator suffices to define a lifted
representation. In general the obstruction to lifting a representation
$\Gamma\longrightarrow\PSLtC$ is a cohomology class $\mathfrak{o}\in H^2(\Gamma,\Z/2)$.
Furthermore there exists a central $\Z/2$-extension  $\hat\Gamma \longrightarrow \Gamma$ 
(corresponding to $\mathfrak{o}$) and a lifted representation $\hat\Gamma$ such that
\begin{equation*}
\begin{CD} 
\hat\Gamma @>>>\SLtC  \\
@VVV @VVV \\
\Gamma  @>>> \PSLtC
\end{CD} 
\end{equation*}
commutes. This lift is {\em not unique;\/} the various lifts differ by multiplication by
homomorphisms 
\begin{equation*}
\Gamma\longrightarrow \{\pm \Id \} = \mathsf{center}\big(\SLtC\big)
\end{equation*}
which comprise the group
\begin{equation*}
\Hom\big(\pi_1(\Sigma),\{\pm \Id\} \big) \cong H^1(\Sigma;\Z/2).
\end{equation*}

The cohomology class in $H^2(\Gamma,\Z/2)$ may be understood in terms
of {\em Hopf's formula\/} for the second homology of a group. 
(See, for example,  Brown~\cite{Brown}.) Consider a  
presentation $\Gamma = F/R$ where $F$ is a finitely generated free
group and $R\lhd F$ 
is a normal subgroup. A set
$\{f_1,\dots,f_N\}$ of free generators for $F$ corresponds to the
generators of $\Gamma$ and $R$ corresponds to the {\em relations\/}
among these generators. Then Hopf's formula identifies $H_2(\Gamma)$
with the quotient group 
\begin{equation*}
\big([F,F]\cap R\big)/ [F,R],  
\end{equation*}
where $[F,F]\lhd F$ is the commutator subgroup and $[F,R]$ is the (normal)
subgroup of $F$ generated by commutators $[f,r]$ where $f\in F$ and
$r\in R$. Intuitively, $H_2(\Gamma)$ is generated by relations which
are products of simple commutators $[a_1,b_1]\dots[a_g,b_g]$, where
$a_i,b_i\in F$ are words in $f_1,\dots,f_N$.  Such {\em commutator
relations\/} correspond to maps of a closed orientable surface
$\Sigma_g$ into the classifying space $B\Gamma$ of $\Gamma$. If
$\Gamma\xrightarrow{\rho} G$ is a homomorphism into $G$ and $\tG
\longrightarrow G$ is a central extension (such as a covering group of
a Lie group), then the obstruction is calculated for each commutator
relation 
\begin{equation*}
w = [a_1,b_1]\dots[a_g,b_g]\in [F,F]\cap R  
\end{equation*}
corresponding to a $2$-cycle $z$, as follows.  (Here each $a_i,b_i\in
F$ is a word in the free generators $f_1,\dots, f_N$.)  Lift each 
generator $\rho(f_i)$ to $\trho(f_i)\in\tG$ and evaluate the word
$w(f_1,\dots,f_N)$ on the lifts $\trho(f_i)$ to obtain an element in
the kernel $K$ of $\tG \longrightarrow G$ (since $w\in R$).
Furthermore since $w\in [F,F]$ and two lifts differ by an element of
$K\subset \mathsf{center}(\tG)$, this element is independent of the
chosen lift $\trho$. This procedure defines an element of
\begin{equation*}
H^2(\Gamma,K) \;\cong\; \Hom\bigg(\frac{[F,F]\cap R}{[F,R]},K \bigg) 
\end{equation*}
which evidently vanishes if and only if $\rho$ lifts.
(Compare Milnor~\cite{Milnor}).   
For more discussion of lifting homomorphisms to $\SLtC$, 
compare Culler~\cite{Culler}, Kra~\cite{Kra}, Goldman~\cite{Topcomps}
or Patterson~\cite{Patterson}. According to Patterson~\cite{Patterson},
the first result of this type, due to H.\ Petersson~\cite{Petersson},
is that a Fuchsian subgroup of $\PSLtR$ lifts to $\SLtR$ if and only
if it has no elements of order two.) 
\index{lifting representations to $\SLtC$}

A representation $\Gamma\longrightarrow\PSLtC$ is {\em irreducible\/} if one (and hence every) lift
$\hat\Gamma\longrightarrow\SLtC$ is irreducible.

\subsection{The $3$-holed sphere.}\label{sec:threeholedsphere}
The basic building block for hyperbolic surfaces is the 
three-holed sphere $\Sigma_{0,3}$.


\subsubsection*{Geometric version of Theorem~A.}
\index{character variety of $\F_2$, geometric version}

Theorem~\ref{thm:vogt} has a suggestive interpretation in terms of the
three-holed sphere $\Sigma_{0,3}$, or ``pair-of-pants.'' 
Namely, the fundamental group
\begin{equation*}
\pi_1(\Sigma_{0,3})\cong \F_2 
\end{equation*}
admits the {\em redundant geometric presentation\/} 
\begin{equation*}
\pi = \pi_1(\Sigma_{0,3})
= \langle X,Y,Z \mid  X Y Z = 1\rangle,
\end{equation*}
where $X,Y,Z$ correspond to the three components of $\partial\Sigma_{0,3}$.
Denoting the corresponding trace functions by lower case, for example
\begin{align*}
\Hom(\pi,G)  &\xrightarrow{x} \C\\
\rho &\longmapsto \tr\big(\rho(X)\big),
\end{align*}
Theorem~A asserts that the $\SLtC$-character ring of $\pi$
is the polynomial ring $\C[x,y,z]$. 

\begin{theorem}
The equivalence class of a flat $\SLtC$-bundle over $\Sigma_{0,3}$ with 
irreducible holonomy is 
determined by the equivalence classes of its restrictions to the three
components of $\partial\Sigma_{0,3}$. Furthermore any triple of isomorphism 
classes of flat $\SLtC$-bundles over 
$\partial\Sigma_{0,3}$ 
whose holonomy traces satisfy
\begin{equation*}
x^2 + y^2 + z^2 - x y z \neq 4 
\end{equation*}
extends to a flat $\SLtC$-bundle over
$\Sigma_{0,3}$.
\end{theorem} 

\subsubsection*{The hexagon orbifold.}\label{sec:hex}
\index{hexagon orbifold}
Every irreducible representation $\rho$ corresponds to a geometric object
in $\Hth$, a {\em triple of geodesics.\/} Any two of
these geodesics admits a unique common perpendicular geodesic.
These perpendiculars cut off a hexagon bounded by geodesic segments,
with all six angles right angles. Such a {\em right hexagon\/} in $\Hth$
is an alternate geometric object corresponding to $\rho$.

The surface $\Sigma_{0,3}$ admits an {\em orientation-reversing involution \/} 
\begin{equation*}
\Sigma_{0,3}
\xrightarrow{\iota_\hexagon} 
\Sigma_{0,3}
\end{equation*}
whose restriction to each boundary component is a reflection. The
quotient $\hexagon$ by this involution is a disc, combinatorially
equivalent to a hexagon.  The three boundary components map to three
intervals $\partial_i(\hexagon)$, for $i=1,2,3$, in the boundary
$\partial \hexagon$.  The other three edges in $\partial \hexagon$
correspond to the three arcs comprising the fixed point set
$\Fix(\iota_\hexagon)$. The orbifold structure on $\hexagon$ is
defined by mirrors on these three arcs on $\partial \hexagon$. The
quotient map
\begin{equation*}
\Sigma_{0,3} \xrightarrow{\Pi_\hexagon} \hexagon
\end{equation*}
is an orbifold covering-space, representing
$\Sigma_{0,3}$ as the orientable double covering 
of the orbifold $\hexagon$.


\begin{figure}
\centerline{\epsfxsize=1.5in 
\epsfbox{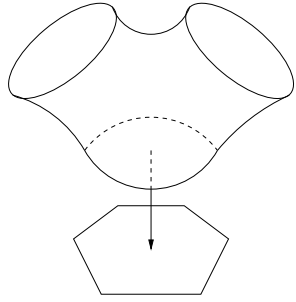}}
\caption{The  three-holed sphere double covers a hexagon orbifold}
\label{fig:hexpants}
\end{figure}
\index{right hexagon orbifold}
The orbifold fundamental group is
\begin{align*}
\hat\pi := \pi_1(\hexagon) &= \langle 
\iota_{YZ}, \iota_{ZX}, \iota_{XY} \mid 
\iota_{YZ}^2 = \iota_{ZX}^2 = \iota_{XY}^2 = 1 \rangle\\
&\cong \Z/2 \star \Z/2 \star \Z/2.
\end{align*}
The covering-space  $\Sigma_{0,3} \xrightarrow{\Pi_\hexagon} \hexagon$
induces the embedding of fundamental groups:
\begin{align*}
\pi_1(\Sigma_{0,3}) &\xrightarrow{(\Pi_\hexagon)_*} \pi_1(\hexagon) \\ 
X &\longmapsto  \iota_{ZX}\iota_{XY} \\
Y &\longmapsto  \iota_{XY}\iota_{YZ} \\
Z &\longmapsto  \iota_{YZ}\iota_{ZX}.
\end{align*}

\begin{theorem}\label{thm:factor}
Let $\pi\xrightarrow{\rho}\PGLtC$ be an irreducible representation.
Then there exists a unique representation 
$\hat\pi\xrightarrow{\hat\rho}\PGLtC$ such that
$\rho = \hat\rho \circ (\Pi_\hexagon)_*.$
\end{theorem}
\noindent
Every element of order two in $\PGLtC$ is reflection about some geodesic.
Therefore a representation $\hat\rho$  corresponds exactly to an ordered
triple of geodesics in $\Hth$. 
Denote this ordered triple of geodesics in $\Hth$ corresponding to $\rho$ 
by $\iota^\rho$.

\begin{corollary}\label{cor:factor}
Irreducible representations $\pi\xrightarrow{\rho}\PGLtC$ correspond to
triples $\iota^\rho$ of geodesics in $\Hth$, which share neither a common endpoint  
nor a common orthogonal geodesic.
\end{corollary}
\noindent
The proofs of Theorem~\ref{thm:factor} and Corollary~\ref{cor:factor}
occupy the remainder of this section.

%

\subsubsection*{Involutions in $\PGLtC$.}
\index{involutions}
We are particularly interested in projective transformations of $\CP^1$
of order two, which we call {\em involutions.\/} Such an involution
is given by a matrix $\xi\in\GLtC$ such that $\xi^2$ does
act identically on $\CP^1$ but $\xi$ does {\em not\/} act identically
on $\CP^1$. 
Thus $\xi$ is a matrix whose square is a scalar matrix but
$\xi$ itself is not scalar. Since $\det(\xi)\neq 0$, replacing $\xi$ 
by 
\begin{equation*}
\det(\xi)^{-1/2} \xi 
\end{equation*}
--- for either 
choice of $\det(\xi)^{-1/2}$ --- 
ensures that $\det(\xi)=1$. 
Then the scalar matrix $\xi^2 = \pm \Id$.
If $\xi^2 = \Id$, then $\det(\xi) = 1$ implies
$\xi = -\Id$, a contradiction. Hence $\xi^2 = -\Id$, 
and $\xi$ must have distinct reciprocal 
eigenvalues $\pm i$.  Thus $\xi$ is conjugate to
\begin{equation*}
\bmatrix i & 0 \\ 0 & -i \endbmatrix.
\end{equation*}
The corresponding projective transformation $\P(\xi)$ has two fixed points.
The orbit of any point not in $\Fix\big(\P(\xi)\big)$  has cardinality two. 

\begin{proposition}\label{prop:matrixinvolution}
Let $\xi\in\Mat$. The following conditions are equivalent:
\begin{itemize}
\item $\P(\xi)\in\Inv$; 
\item $\xi$ is conjugate to $\bmatrix i & 0 \\ 0 & -i \endbmatrix$;
\item $\det(\xi) = 1$ and $\tr(\xi)=0$;
\item $\xi^2 = -\Id$ and $\xi\neq \pm i\Id$;
\item $\xi^2 = -\Id$ and $\xi$ is not a scalar matrix.
\end{itemize}
\end{proposition}
\noindent
The proof is left as an exercise. Denote the collection of such
matrices by
\begin{align*}
\widetilde\Inv & := \SLtC \cap \sltC \\ & = \{\xi\in\Mat \mid \det(\xi) = 1, \tr(\xi) = 0\}.
\end{align*}
Notice that $\widetilde\Inv$ is invariant under $\pm\Id$, and the quotient 
\begin{equation*}
\Inv := \widetilde\Inv/\{\pm\Id\} \subset\PGLtC
\end{equation*}
consists of all projective involutions of $\CP^1$.
It naturally identifies with the 
collection of unordered pairs of distinct points in $\CP^1$,
that is, the quotient
\begin{equation*}
\big(\CP^1 \times \CP^1 \setminus \Delta_{\CP^1} \big)/ \mathfrak{S}_2
\end{equation*}
of the complement in $\CP^1\times\CP^1$ of the 
diagonal 
\begin{equation*}
\Delta_{\CP^1} \subset \CP^1 \times \CP^1 
\end{equation*}
by the symmetric group $\mathfrak{S}_2$.
In \S\ref{sec:Hth}, we interpret $\widetilde\Inv$ as the space
of {\em oriented\/} geodesics in hyperbolic $3$-space $\Hth$.

\subsubsection*{Involutions and the complex projective line.}\label{sec:closedInvolutions}
\index{involutions}
Denote by $\overline{\Inv}$ the closure of $\Inv$ in the projective
space $\P(\sltC)$. The complement $\overline{\Inv}\setminus\Inv$
corresponds to $\CP^1$  embedded as
the diagonal $\Delta_{\CP^1}$ in the above
description.
For example the elements of $\overline{\Inv}$
corresponding to $0,\infty\in\CP^1$ are the respective lines
\begin{equation*}
\bmatrix 0 & * \\ 0 & 0 \endbmatrix,  
\bmatrix 0 & 0 \\ * & 0 \endbmatrix \subset \Mat.  
\end{equation*}
The closure corresponds to the full quotient space
\begin{equation*}
\big(\CP^1 \times \CP^1 \big)/ \mathfrak{S}_2. 
\end{equation*}
An element $\xi\in\PGLtC\setminus\{\Id\}$ stabilizes a unique element
$\iota_\xi\in\overline{\Inv}$. If $\xi$ is semisimple ($\#\Fix(\xi) = 2$),
then $\iota_\xi$ is the unique involution with the same fixed
points. Otherwise $\xi$ is parabolic ($\#\Fix(\xi) = 1$), and $\iota_\xi$ 
corresponds to the line 
\begin{equation*}
\Fix\big(\Ad(\xi)\big) = \Ker\big(\Id - \Ad(\xi)\big)\subset \sltC,
\end{equation*}
the Lie algebra centralizer of $\xi$ in $\sltC$.
Further discussion of semisimple elements in $\SLtC$ and $\PSLtC$
is given in \S\ref{sec:Hth}. 

Here is an elegant matrix representation.
If $\xi\in\SLtC$ is semisimple and $\neq\Id$, 
then the two lifts of $\iota_\xi\in\Inv$ to $\widetilde{\Inv}\subset\SLtC$ differ by $\pm\Id$.
Since $\xi$ is semisimple, its 
{\em traceless projection\/}
\begin{equation*}
\xi' := \xi - \frac12 \tr(\xi) \Id 
\end{equation*}
satisfies
\begin{itemize}
\item $\tr(\xi') = 0$;
\item $\xi'$ commutes with $\xi$;
\item $\det(\xi') \neq 0$ (semisimplicity).
\end{itemize}
Choose $\delta\in\C^*$ such that 
\begin{equation*}
\delta^2 = \det(\xi') = \frac{4 - \tr(\xi)^2}4.
\end{equation*}
Then $\delta^{-1}\xi'\in\widetilde{\Inv}$  and represents the involution
$\iota_\xi$ centralizing $\xi$: 
\begin{equation}\label{eq:invformula} 
\widetilde{\iota_\xi} \;=\; \pm \frac2{\sqrt{4 - \tr(\xi)^2}} \bigg(\xi - \frac{\tr(\xi)}2 \Id\bigg).
\end{equation}
This formula will be used later in \eqref{eq:hat}.

\subsubsection*{$3$-dimensional hyperbolic geometry.}\label{sec:Hth}
\index{hyperbolic three-space}

The group $\GLtC$ acts by orien\-ta\-tion-preserving isometries
on {\em hyperbolic $3$-space\/} $\Hth$. The kernel of the action 
equals the center of $\GLtC$, the group $\C^*$ of nonzero
scalar matrices. The quotient
\begin{equation*}
\PGLtC := \GLtC/\C^*
\end{equation*}
acts effectively on $\Hth$. 
The restriction of the quotient homomorphism 
\begin{equation*}
\GLtC\longrightarrow\PGLtC 
\end{equation*}
to $\SLtC\subset\GLtC$
defines an isomorphism
\begin{equation*}
\PSLtC \xrightarrow{\cong} \PGLtC.
\end{equation*}
The projective line $\CP^1$ identifies naturally
with the ideal boundary $\partial\Hth$.
The center of $\SLtC$ consists of $\pm\Id$, which is the kernel of the
actions on $\Hth$ and $\CP^1$. The only element of order two in $\GLtC$ is
$-\Id$, and an element of even order $2k$ in $\PGLtC$ corresponds to
an element of order $4k$ in $\GLtC$.
Elements of odd order $2k+1$ in $\PGLtC$ have two lifts to $\SLtC$,
one of order $2k+1$ and the other of order $2(2k+1)$.

We use the {\em upper-half-space model\/} of $\Hth$ as follows.
The algebra $\HH$ of {\em Hamilton quaternions\/} is
the $\R$-algebra generated by $1,i,\jj$ subject to the relations
\begin{equation*}
i^2 = \jj^2 = -1, i \jj + \jj i = 0.
\end{equation*}
\noindent 
$\HH$ contains the smaller subalgebra $\C$ having basis $\{1,i\}$. 
Define 
\begin{equation*}
\Hth := \{ z + u \jj \in \HH \mid   z\in \C, u\in\R, u > 0\} 
\end{equation*}
where 
\begin{equation*}
\xi = \bmatrix a & b \\ c & d \endbmatrix \in \GLtC
\end{equation*}
acts by
\begin{equation*}
z + u \jj \longmapsto 
\big(a (z + u \jj) + b\big)
\big(c (z + u\jj) + d\big)^{-1}. 
\end{equation*}
$\PSLtC$ is the group of orientation-preserving {\em isometries\/}
of $\Hth$ with respect to the {\em Poincar\'e metric\/}
\begin{equation*}
u^{-2} \big(\vert dz\vert^2 + du^2 \big)
\end{equation*}
of constant curvature $-1$.
The restriction of $\GLtC$ to $\partial\Hth$ identifies with 
the usual projective
action of $\PGLtC$ on 
\begin{equation*}
\partial\Hth := \CP^1 = \C \cup \{\infty\},
\end{equation*}
the space of complex lines ($1$-dimensional linear subspaces) in $\C^2$.

Oriented geodesics in $\Hth$ correspond to ordered pairs of distinct points
in $\CP^1$, via their endpoints.
Unoriented geodesics correspond to unordered pairs.
For example geodesics with an
endpoint at $\infty$ are represented by vertical rays $z + \R_+\jj$,
where $z\in\C$ is the other endpoint. The unit-speed parametrization
of this geodesic is 
\begin{align*}
\R &\longrightarrow \Hth \\
t &\longmapsto z + e^t \jj 
\end{align*}
Distinct $z_1,z_2\in\C$ span a geodesic in $\Hth$ whose unit-speed
parametrization is:
\begin{align*}
\R &\longrightarrow \Hth \\
t & \;\longmapsto\; \frac{z_1 \,+\, z_2}2  + \frac{z_2 - z_1}2\,  \bigg(\tanh(t) \;+\; \operatorname{sech}(t)\jj\bigg).
\end{align*}
A geodesic $l\subset \Hth$  
corresponds uniquely to the involution $\iota = \iota_l\in\PSLtC$ for which
\begin{equation*}
l = \Fix(\iota). 
\end{equation*}
\noindent
For example, if $z_1,z_2\in\C$, the involution in $\PGLtC$ fixing $z_1,z_2$ 
is given by the pair of  matrices
\begin{equation}\label{eq:involutionformula}
\pm \frac{i}{z_1-z_2} \bmatrix z_1 + z_2 & -2 z_1 z_2 \\
2 & -(z_1+z_2) \endbmatrix \in \SLtC. 
\end{equation}
If $z_2 = \infty$, the corresponding matrices are:
\begin{equation}\label{eq:involutionformula2}
\pm 
i \bmatrix -1  & 2 z_1  \\ 0 & 1 \endbmatrix \in \SLtC. 
\end{equation}
\noindent 
Compare Fenchel~\cite{Fenchel} for more details.

Let $\xi\in\SLtC$ be {\em non-central:\/} $\xi \neq \pm \Id$. Then the following
conditions are equivalent:
\begin{itemize}
\item
$\xi$ has two distinct eigenvalues;
\item
$\tr(\xi)\neq \pm 2$;
\item
the corresponding collineation of $\CP^1$ has two fixed points;
\item
the correponding orientation-preserving isometry of $\Hth$ leaves invariant
a unique geodesic $\ell_\xi$, each of whose endpoints is fixed;
\item
a unique involution $\iota_\xi$ centralizes $\xi$.
\end{itemize}  
(In the standard terminology, the corresponding isometry of $\Hth$
is either {\em elliptic\/} or {\em loxodromic.})

We shall say that $\xi$ is {\em semisimple.\/}
Otherwise $\xi$ is {\em parabolic:\/} it has a repeated eigenvalue (necessarily
$\pm 1$, because $\det(\xi)=1$), and fixes a {\em unique\/} point on $\CP^1$.

Suppose $\xi\in\SLtC$ and the corresponding 
isometry $\P(\xi)\in\PSLtC$ leaves invariant a geodesic $l\subset\Hth$.  Then
the restriction $\P(\xi)|_l$ is an isometry of $l\approx\R$.  

Any isometry
of $\R$ is either a translation of $\R$, a reflection in a point of
$\R$, or the identity. We distinguish these three cases as follows.
For concreteness choose coordinates so that $l$ is represented by
the imaginary axis $\R_+\jj\subset\Hth$ in the upper-half-space model.
The endpoints of $l$ are $0,\infty$:
\begin{itemize}
\item{\bf $\P(\xi)|_l$ acts by translation:}
Then $\P(\xi)$ is {\em loxodromic,\/} represented by
\begin{equation}\label{eq:diagonalmatrix}
\bmatrix \lambda & 0 \\ 0 & \lambda^{-1} \endbmatrix
\end{equation}
where $\lambda\in\C^*$ is a nonzero complex number and 
$\vert\lambda\vert \neq 1$. The fixed point set is 
\begin{equation*}
\Fix\big(\P(\xi)\big) = \{0,\infty\}.  
\end{equation*}
The restriction of 
$\P(\xi)$ to $l$ is translation along $l$ by distance 
$2\log\vert\lambda\vert$, in the direction from $0$ (its repellor)
to $\infty$ (its attractor) if $\vert\lambda\vert > 1$.
(If $\vert\lambda\vert < 1$, then $0$ is the attractor and $\infty$ is the
repellor.)
\item{\bf $\P(\xi)|_l$ acts identically:}
Now $\P(\xi)$ is {\em elliptic\/} and is represented by the diagonal matrix
\eqref{eq:diagonalmatrix}, except now $\vert\lambda\vert = 1$.
If $\lambda = e^{i\theta}$, then $\P(\xi)$ represents a rotation through angle
$2\theta$ about $l$. In particular if $\lambda = \pm i$, then $\P(\xi)$ is the
involution fixing $l$. Although $\P(\xi)$ has order two in $\PGLtC$, its matrix
representatives in $\SLtC$ each have order $4$. 
(Compare Proposition~\ref{prop:matrixinvolution}.)

\item{\bf $\P(\xi)|_l$ acts by reflection:}
In this case $\P(\xi)$ interchanges the two endpoints $0,\infty$ and is necessarily
of order two. Its restriction $\P(\xi)|_l$ to $l$ fixes the point $p = \Fix\big(\P(\xi)\big)\cap l$,
and is reflection in $p$. The corresponding matrix is:
\begin{equation*} 
\bmatrix 0 & -\lambda  \\  \lambda^{-1} & 0 \endbmatrix
\end{equation*}
where $\lambda\in\C^*$ and $p = \vert\lambda\vert\jj$
is the fixed point of $\P(\xi)_l$. Necessarily $\P(\xi)\in\Inv$ and 
\begin{equation*}
\Fix\big(\P(\xi)\big) = \{\pm i \lambda\}.
\end{equation*}

\end{itemize}

\subsubsection*{Dihedral representations.}\label{sec:dihedral}
\index{dihedral representations}
The following lemma is crucial in the proof of 
Theorem~\ref{thm:factor}.

\begin{lemma}\label{lem:refl}
Suppose that $\xi\in\SLtC\setminus\{\pm\Id\}$ and $\iota\in\Inv$.
\begin{enumerate}
\item 
Suppose that $\#\Fix\big(\P(\xi)\big) = 2$.
Let $\ell_\xi\subset\Hth$ denote the unique $\xi$-invariant geodesic
(the geodesic with endpoints $\Fix\big(\P(\xi))$\big).
Then
\begin{equation}\label{eq:dih}
\iota \xi \iota = \xi^{-1} 
\end{equation}
if and only if 
$\iota$ preserves $\ell_\xi$ and its restriction acts by reflection.
In that case $\iota$ interchanges the two elements of $\Fix\big(\P(\xi)\big)$.
\item 
Suppose that $\#\Fix\big(\P(\xi)\big) = 1$.
Then 
$\iota \xi \iota = \xi^{-1}$ if and only if 
$\Fix\big(\P(\xi)\big)\subset \Fix(\iota)$. 
\end{enumerate}
\end{lemma}

\begin{proof}
Consider first the case that $\xi$ is semisimple, that is, when 
$\#\Fix\big(\P(\xi)\big) = 2$. Let $\ell_\xi\subset\Hth$ be the $\xi$-invariant geodesic
with endpoints $\Fix\big(\P(\xi)\big)$. Then $\iota$ interchanges the two elements
of $\Fix\big(\P(\xi)\big)$. 
In terms of the linear representation, 
$\xi$ preserves a decomposition 
into eigenspaces 
\begin{equation*}
\C^2 = L_1\oplus L_2 
\end{equation*}
where each line $L_i\subset\C^2$ corresponds to a fixed point in $\CP^1$,
and $\iota$ interchanges $L_1$ and $L_2$.

When $\#\Fix\big(\P(\xi)\big) = 1$, the corresponding matrix has a unique eigenspace,
which we take to be the first coordinate line. Then $\xi$ is represented
by the upper-triangular matrix
\begin{equation*}
\pm \bmatrix 1 & u \\ 0 & 1 \endbmatrix 
\end{equation*}
and $\iota$ is also represented by an upper-triangular matrix, of the
form
\begin{equation*}
\pm \bmatrix i & w \\ 0 & -i \endbmatrix.
\end{equation*}
Rewrite \eqref{eq:dih} as:
\begin{equation*}
(\iota \xi)^2 = \Id 
\end{equation*}
so that $\iota \xi\in\Inv$. Thus $\xi$ factors as the product of
two involutions
\begin{equation*}
\xi = \iota (\iota \xi).
\end{equation*}
Conversely if $\iota,\iota'\in\Inv$, then the product 
$\xi := \iota\iota'$ satisfies \eqref{eq:dih}.
\end{proof}

\subsubsection*{Geometric interpretation of the Lie product.}
\index{Lie product, geometric interpretation}
These ideas provide an elegant formula for the common orthogonal
of the invariant axes of elements of $\PSLtC$. 
Suppose $\xi,\eta\in\SLtC$. Then the {\em Lie product\/}
\begin{equation}\label{eq:LieProduct}
\Lie(\xi,\eta) := \xi \eta - \eta \xi
\end{equation}
has trace zero, and vanishes if and only if $\xi, \eta$ commute.
Furthermore \eqref{eq:twokindsofcommutator} and 
Proposition~\ref{prop:irreducibility}, \eqref{item:irr} 
$\Longleftrightarrow$
\eqref{item:trcomm}
imply that 
$\Lie(\xi,\eta)$ is invertible if and only if 
$\langle \xi, \eta\rangle$ acts irreducibly (as defined in \S\ref{sec:Coxeter}).
Suppose $\langle \xi,\eta\rangle$ acts irreducibly, so that $\Lie(\xi,\eta)$ defines
an element $\lambda\in\PSLtC$. 
Since $\tr\big(\Lie(\xi,\eta)\big) = 0$, the isometry $\lambda$ has order two,
that is, lies in $\Inv$.

Now 
\begin{align*}
\tr\big(\xi\, \Lie(\xi,\eta)\big) & = \tr\big(\xi(\xi\eta)\big) - \tr\big(\xi(\eta\xi)\big) \\
& = \tr\big(\xi(\xi\eta)\big) - \tr\big((\xi\eta)\xi\big) \\ & = 0
\end{align*}
which implies that $\xi\lambda$ also has order two, that is, $\lambda \xi\lambda = \xi^{-1}$.
Lemma \ref{lem:refl} implies that $\lambda$ acts by reflection on the invariant axis $\ell_\xi$.
Similarly $\lambda$ acts by reflection on the invariant axis $\ell_\eta$.
Hence the fixed axis $\ell_\lambda$ is orthogonal to both $\ell_\xi$ and $\ell_\eta$:

\begin{proposition}\label{prop:Lie}
If $\xi,\eta\in\GLtC$, then
the Lie product $\Lie(\xi,\eta)$ represents the common orthogonal geodesic 
$\perp(
\ell_{\P(\xi)},
\ell_{\P(\eta)})$ to the   
invariant axes $\ell_{\P(\xi)}, \ell_{\P(\eta)}$ of $\P(\xi)$ and $\P(\eta)$ respectively.
\end{proposition}
Compare Marden~\cite{Marden} and the references given there.
\index{Lie product}

\subsubsection*{Geometric proof of Theorem~\ref{thm:factor}.}
\index{extension to Coxeter group, geometric proof}

\begin{proof}[Proof of Theorem~\ref{thm:factor}]

Abusing notation, write $X,Y,Z$ for 
\begin{equation*}
\rho(X),\rho(Y),\rho(Z)\in\PGLtC 
\end{equation*}
respectively.
We seek respective involutions 
\begin{equation*}
\rho(\iota_{XY}),\,\rho(\iota_{YZ}),\,\rho(\iota_{ZX}), 
\end{equation*}
which we respectively denote
$\rxy,\,\ryz,\,\rzx$.
These involutions will be the ones fixing the respective
pairs. For example we take $\rxy$ to be the involution
fixing $\perp(l_X,l_Y)$, and similarly for $\ryz$ and $\rzx$.

Suppose that $\langle X,Y\rangle\subset\SLtC$ acts irreducibly on $\C^2$.
Write $Z = Y^{-1} X^{-1}$ so that 
\begin{equation*}
X Y Z = \Id .
\end{equation*}
Since $\langle X,Y\rangle$ acts irreducibly, none of $X,Y,Z$ 
act identically. 

Let $\rxy\in\Inv$  be the unique involution such that
\begin{align}
\rxy X  \rxy &= X^{-1} \label{eq:conjinv} \\
\rxy Y  \rxy &= Y^{-1} \notag
\end{align}
respectively.
(By Proposition~\ref{prop:Lie}, it is represented by the Lie product $\Lie(X,Y)$.)
The involution $\rxy$ will be specified by its fixed
line $\lxy = \Fix(\rxy)$, which is defined as follows:
If both $X,Y$ are semisimple, then 
Lemma~\ref{lem:refl} implies $\rxy$ is the involution
fixing the unique common orthogonal geodesic 
to the invariant axes of $X$ or $Y$.
If both $X,Y$ are parabolic, then $\rxy$ is the involution
in the geodesic bounded by the fixed points of $X,Y$.
Finally consider the case when one element is semisimple and the other element is parabolic.
Then $\lxy$ is the unique geodesic,  
for which one endpoint is the fixed point of the 
parabolic element, and which is  
orthogonal to the invariant axis of the semisimple element.

Similarly define lines
$\lyz,\lzx$ with respective involutions
$\ryz,\rzx\in\Inv$. 
The triple $(\rxy,\ryz,\rzx)$ defines the homomorphism
$\hat\rho$ of Theorem~\ref{thm:factor}.

{\bf Claim:\/} $X = \rzx \rxy$. 
To this end we show $X\rxy$ equals $\rzx$. 
First, $X\rxy$ fixes $\Fix(X)$, since both $X$ and $\rxy$ fix $\Fix(X)$.
By \eqref{eq:conjinv}, 
\begin{equation*}
\rxy X Y \rxy = X^{-1} Y^{-1} = X^{-1} (XY)^{-1} X.
\end{equation*}
Equivalently,
\begin{equation*}
\rxy Z^{-1} \rxy = X^{-1} Y^{-1} = X^{-1} Z X,
\end{equation*}
which implies
\begin{equation*}
(X\rxy) Z^{-1} (X\rxy)^{-1} = Z. 
\end{equation*}
Now Lemma~\ref{lem:refl} (1) implies that $X\rxy$ preserves
$\Fix(Z)$ and its restriction to the corresponding line
is a reflection. Thus $X\rxy$ is itself an involution $\ell_Z$
with $\Fix(X\rxy)$ orthogonal to the axis of $Z$.

Since $X\rxy$ fixes $\Fix(X)$, it follows that $X\rxy = \rzx$ as claimed.
Similarly $Y\ryz = \rxy$ and $Z\rzx = \ryz$,
completing the proof of Theorem~\ref{thm:factor}.
\end{proof}

\subsection{Orthogonal reflection groups.}
An algebraic proof of Theorem~\ref{thm:factor} involves three-dimensional inner product
spaces and is described 
in Goldman~\cite{Topcomps}. 
This proof exploits the isomorphism $\PSLtC\longrightarrow\SOthC$.

\subsubsection*{The $3$-dimensional orthogonal representation of $\PSLtC$.}
\index{three-dimensional orthogonal representation}

Let $W = \C^2$ with a nondegenerate symplectic form $\omega$.
The {\em symmetric square \/} $\Sym^2(W)$ is a 3-dimensional vector space based on monomials
$e\cdot e, e\cdot f, f\cdot f$, where $e,f$ is a basis of $W$,
and $x \cdot y$ denotes the {\em symmetric product\/} of $x,y$ (the image of the
tensor product $x \otimes y$ under symmetrization).
$\Sym^2(W)$ inherits a symmetric inner product defined by:
\begin{equation*}
(u_1\cdot u_2, v_1\cdot v_2 ) \;\longmapsto\; 
\frac12\,\big(\, \omega(u_1,v_1)\omega(u_2,v_2) + \omega(u_1,v_2)\omega(u_2,v_1)\,\big).
\end{equation*}
If $e,f\in W$ is a symplectic basis for $W$, the corresponding inner product for
$\Sym^2(W)$ has matrix
\begin{equation*}
\bmatrix 0 & 0 & 1 \\ 0 &  -1/2 & 0 \\ 1 & 0 & 0\endbmatrix
\end{equation*}
with respect to the above basis 
of $\Sym^2(W)$. In particular the inner product is nondegenerate.
Every $\xi\in\SLtC$ induces an isometry of $\Sym^2(W)$ with respect to this inner product.
This correspondence defines a local isomorphism
\begin{equation*}
\SLtC \xrightarrow{\Sym^2} \SOthC 
\end{equation*}
with kernel $\{\pm\Id\}$ and a resulting {\em isomorphism\/}
$\PSLtC\longrightarrow\SOthC$.

If $\xi\in\SLtC$, then
\begin{equation}\label{eq:trSym2}
\tr\big(\Sym^2(\xi)\big) = \tr(\xi)^2 - 1.
\end{equation}
For example, the diagonal matrix
\begin{equation*}
\xi = \bmatrix \lambda & 0 \\ 0 & \lambda^{-1} \endbmatrix 
\end{equation*}
induces the diagonal matrix
\begin{equation*}
\Sym^2(\xi) = \bmatrix \lambda^2 & 0 & 0  \\ 0 & 1 & 0 \\ 0 & 0 &\lambda^{-2} \endbmatrix
\end{equation*}
and 
\begin{equation*}
\tr\big(\Sym^2(\xi)\big) = \lambda^2 +  1 + \lambda^{-2} = (\lambda + \lambda^{-1})^2 -1.
\end{equation*}

Alternatively, this is the adjoint representation of $\SLtC$ on its Lie algebra
$\slt\cong\Sym^2(W)$. Here the standard basis of $\C$ is
\begin{equation*}
e = \bmatrix 1 \\ 0 \endbmatrix, \quad 
f = \bmatrix 0 \\ 1 \endbmatrix
\end{equation*}
and the monomials correspond to
\begin{equation*}
e\cdot e = \bmatrix 1 & 0 \\ 0 & 0 \endbmatrix, \quad
e\cdot f = \frac12\bmatrix 0 & 1 \\ 1 & 0 \endbmatrix, \quad 
f\cdot f = \bmatrix 0 & 0 \\ 0 & 1 \endbmatrix.
\end{equation*}
The inner product corresponds to the {\em trace form\/}
\begin{equation*}
(X,Y) \longmapsto\; \frac12\, \tr(XY) 
\end{equation*}
which is $1/8$ the Killing form on $\sltC$.

\subsubsection*{$3$-dimensional inner product spaces.}
\index{three-dimensional inner product spaces.}
Let $e_1,e_2,e_3$ denote the standard basis of $\C^3$.
A $3\times 3$ symmetric matrix $B$ determines an inner product $\BB$ on $\C^3$ 
by the usual rule:
\begin{equation*}
(v,w) \stackrel{\BB}\longmapsto  v^\dag B w.
\end{equation*}
We suppose that $\BB$ is nonzero on $e_1,e_2,e_3$; 
in fact, let's normalize $\BB$ so that its basic values
are $1$:
\begin{equation*}
\BB(e_i,e_i) = 1 \text{~for~} i=1,2,3. 
\end{equation*}
In other words, the diagonal entries satisfy $B_{11} =B_{22}= B_{33} = 1$.

Let $R_i = R_i^{(B)}$ denote the orthogonal reflection in $e_i$ defined by $\BB$:
\begin{equation*}
v \stackrel{R_i}\longmapsto   v - 2\, \BB(v,e_i) e_i
\end{equation*}
with corresponding matrix: 
\begin{equation*}
R_i \;:=\; \Id - 2 e_i (e_i)^\dag B. 
\end{equation*}
($e_i(e_i)^\dag B$ is the $3\times 3$ matrix with the same $i$-th row as $B$ and
the other two rows zero.) Since 
\begin{equation*}
\Id = \BB(e_i,e_i) = (e_i)^\dag B e_i \qquad \text{~(matrix multiplication),} 
\end{equation*}
\begin{align*}
R_i^\dag B R_i - B  
&  = \big(\Id - 2 B e_i(e_i)^\dag\big) B \big(\Id - 2 e_i (e_i)^\dag B\big)  - B \\
& = - 2 Be_i(e_i)^\dag B - 2 Be_i(e_i)^\dag B  + 4 Be_i(e_i)^\dag B e_i (e_i)^\dag B  \\ 
& = - 2 Be_i(e_i)^\dag B - 2 Be_i(e_i)^\dag B  + 4 Be_i(e_i)^\dag B \\
& = 0
\end{align*}
so {\em $R_i$ is orthogonal with respect to $\BB$.\/}
 
Thus the matrix $B$ determines a triple of involutions
$R_1^{(B)}, R_2^{(B)}, R_3^{(B)}$  
in the {\em orthogonal group of $\BB$\/}:
\begin{equation*}
\OB := \{ \xi \in \GLtC \mid  \xi^\dag B \xi = B \}. 
\end{equation*}
In other words, $B$ defines a representation $\hat\rho := \hat\rho^{(B)}$ 
of the free product 
\begin{equation*}
\hat\pi := \Z/2 \star  \Z/2 \star  \Z/2 
\end{equation*}
in $\OB$, taking the free generators
$\iota_{XY}, \iota_{YZ}, \iota_{ZX}$
of $\hat\pi$ into $R_1^{(B)},R_2^{(B)},R_3^{(B)}$
respectively. The restriction $\rho :=\rho^{(B)}$ of $\hat\rho^{(B)}$ 
to the index-two subgroup 
\begin{equation*}
\Z\star\Z\cong \pi \subset \hat\pi 
\end{equation*}
(compare \S\ref{sec:hex})
assumes values in the subgroup 
\begin{equation*}
\SOB := \SLthC \cap \OB.  
\end{equation*}
When $\BB$ is nondegenerate, then $\SOB\cong \SOthC$ 
(a specific isomorphism corresponds to an orthonormal basis for $\BB$).
There are exactly $4$ lifts $\tilde\rho$ 
of $\rho$ to the double covering-space
\begin{equation*}
\SLtC \longrightarrow \SOthC \xrightarrow{\cong} \SOB
\end{equation*}
To see this, for each generator $X,Y,Z$ of $\pi$, its $\rho$-image has exactly two lifts, differing by $\pm\Id$. 
Lifting the generators to $\widetilde{\rho(X)},\widetilde{\rho(Y)},\widetilde{\rho(Z)}$ respectively,
exactly half of the eight choices satisfy  
\begin{equation}\label{eq:liftedequation}
\widetilde{\rho(X)}
\widetilde{\rho(Y)}
\widetilde{\rho(Z)} = \Id,
\end{equation}
(as desired), 
and for the other four choices the product equals $-\Id$.

Choose one of the four lifts
satisfying \eqref{eq:liftedequation}, and denote it $\tilde\rho$.
If $i\neq j$, the trace of $R_iR_j\in\SLthC$ equals $4 (B_{ij})^2 -1$. 
For example, take $i=1, j=2$:
\begin{align*}
\widetilde{\rho(Z)} = R_1 R_2 & = 
\bmatrix -1 & -2 B_{12}& -2 B_{13} \\ 0 & 1 & 0 \\ 0 & 0 & 1 \endbmatrix
\bmatrix 1 & 0 & 0 \\ -2 B_{12} & -1& -2 B_{23} \\ 0 & 0 & 1 \endbmatrix \\
& = \bmatrix 4 (B_{12})^2 -1 & -2 B_{12}& 4 B_{23}B_{12} -2 B_{13} \\ 
-2 B_{12} & -1 & -2 B_{23} \\ 0 & 0 & 1 \endbmatrix
\end{align*}
has trace $4(B_{12})^2 -1$.

This calculation gives another proof of surjectivity in 
Theorem~\ref{thm:vogt}
(as in \cite{Topcomps}).  
For $(x,y,z)\in\C^3$, the matrix
\begin{equation}\label{eq:bilinearform}
B = \bmatrix 1 & z/2 & y/2 \\ z/2 & 1 & x/2 \\ y/2 & x/2 & 1 \endbmatrix 
\end{equation}
defines a bilinear form $\BB$ and a representation
$\rho^{(B)}$ as above.

The corresponding $\SLtC$-traces of the $\trho$-images of the generators $X,Y,Z$ of $\pi$ satisfy:
\begin{align*}
\tr\big(\trho(X)\big)  & = \pm 2 B_{23} \\
\tr\big(\trho(Y)\big)  & = \pm 2 B_{13} \\
\tr\big(\trho(Z)\big)  & = \pm 2 B_{12} 
\end{align*}
because  (using \eqref{eq:trSym2}): 
\begin{align*}
\tr\big(\tilde{\rho}(X)\big)^2  & = 1 + \tr\bigg(\Sym^2\big(\tilde\rho(X)\big)\bigg) = 1 + \tr(R_2 R_3) = 4 (B_{23})^2 \\
\tr\big(\tilde{\rho}(Y)\big)^2  & = 1 + \tr\bigg(\Sym^2\big(\tilde\rho(Y)\big)\bigg) = 1 + \tr(R_3 R_1) = 4 (B_{31})^2 \\
\tr\big(\tilde{\rho}(Z)\big)^2  & = 1 + \tr\bigg(\Sym^2\big(\tilde\rho(Z)\big)\bigg) = 1 + \tr(R_1 R_2) = 4 (B_{12})^2 
\end{align*}
Now adjust the lifts as above to arrange a representation
$\pi\xrightarrow{\tilde\rho}\SLtC$ with
\begin{align*}
\tr\big(\tilde\rho(X)\big) & = x \\
\tr\big(\tilde\rho(Y)\big) & = y \\
\tr\big(\tilde\rho(Z)\big) & = z
\end{align*}
as desired.

When $\kappa(x,y,z) =2$, the matrix $B$ is singular and we obtain {\em reducible\/} representations.
There are two cases, depending on whether $\rank(B) = 2$ or $\rank(B) = 1$. 
(Since $B\neq 0$, its rank cannot be zero.)

\subsection{Real characters and real forms.}
\index{real characters}
A {\em real character\/} $(x,y,z)\in\R^3$ corresponds to a representation
of a rank-two free group in one of the two {\em real forms\/}
$\SUt, \SLtR$ of $\SLtC$. This was first stated and proved 
in general by Morgan-Shalen~\cite{MorganShalen}. 
Geometrically, $\SUt$-representations are those which fix a point
in $\Hth$, and $\SLtR$-representations are those which preserve a plane
$\Ht\subset\Hth$ as well as an orientation on the plane.
\index{$\SUt$-representations}

\begin{theorem}\label{thm:realcharacters}
Let $(x,y,z)\in\R^3$ and
\begin{equation*}
\kappa(x,y,z) := x^2 + y^2 + z^2 - x y z - 2. 
\end{equation*}
Let $\pi\xrightarrow{\rho}\SLtC$ be a representation with character
$(x,y,z)$. Suppose first that $\kappa(x,y,z)\neq 2$. 
\begin{itemize}
\item
If $-2\le x,y,z \le 2$ and $\kappa(x,y,z)<2$, then $\rho(\pi)$ fixes
a unique point in $\Hth$ and is conjugate to a $\SUt$-representation.
\item
Otherwise $\rho(\pi)$ preserves a unique plane in $\Hth$ and its restriction
to that plane preserves orientation.
\end{itemize}
If $\kappa(x,y,z) = 2$, then $\rho$ is reducible and one of the following 
must occur:
\begin{itemize}
\item
$\rho(\pi)$ acts identically on $\Hth$, in which case $\rho(\pi) \subset \{\pm\Id\}$
is a {\em central representation.\/}
\item
$\rho(\pi)$ fixes a line in $\Hth$, in which case $-2\le x,y,z\le 2$ and 
$\rho$ is conjugate to a representation taking values in $\SOt = \SUt \cap \SLtR$.
\item
$\rho(\pi)$ acts by transvections along a unique line in $\Hth$, in which case 
\begin{equation*}
x,y,z \in \R \setminus \big(-2,2\big). 
\end{equation*}
Then $\rho$ is conjugate to a representation taking values in $\SOoo \subset \SLtR$. 
\item
$\rho(\pi)$ fixes a unique point on $\partial_\infty\Hth$.
\end{itemize}
\end{theorem}
Recall that $\SOoo$ is isomorphic to the multiplicative group $\R^*$ of nonzero real numbers,
and is conjugate to the subgroup of $\SLtR$ consisting of diagonal matrices.

Corollary~\ref{cor:factor} associates to a generic 
representation $\rho$ an ordered triple $\iota^\rho$ of geodesics in $\Hth$.
When $\rho$ is irreducible, the corresponding cases for  $\iota^\rho$ are the following:
\begin{itemize}
\item
If $\rho$ fixes a unique point $p\in\Hth$,
then the three lines are distinct and intersect in $p$.
Conversely if the three lines are concurrent, then $\rho$ is conjugate to an $\SUt$-representation.
\item
If $\rho$ preserves a unique plane $P$, then the three lines are distinct. 
There are two cases:
\begin{itemize}
\item
The three lines are orthogonal to $P$;
\item
The three lines lie in $P$.
\end{itemize}
\end{itemize}
The first case, when the lines are orthogonal to $P$, occurs when $\kappa(x,y,z) <2$. 
In this case the corresponding involutions preserve orientation on $P$.
The second case, when the lines lie in $P$, occurs when $\kappa(x,y,z) > 2$.
In that case the involutions restrict to reflections in geodesics in $P$
which reverse orientation.

\subsubsection*{Real symmetric $3\times 3$ matrices.}
\index{three-by-three real symmetric matrices}

We deduce these facts from the classification given in Theorem~\ref{thm:vogt}.
First assume that $(x,y,z)$ is an {\em irreducible\/} character,
that is, $\kappa(x,y,z)\neq 2$. Theorem~\ref{thm:vogt} implies that
$(x,y,z)$ is the character of the representation $\rho$
given by \eqref{eq:explicitrep}, and all such characters are $\PGLtC$-conjugate.

The matrix $B$ defining the bilinear form $\BB$ in 
\eqref{eq:bilinearform} satisfies: 
\begin{equation*}
4\, \det(B) \;=\; 2 - \kappa(x,y,z) 
\end{equation*}
so $\BB$ is degenerate if and only if $\kappa(x,y,z) = 2$,
in which case $\rho$ is reducible.

Suppose that $\BB$ is nondegenerate, so that 
either 
$\kappa(x,y,z) > 2$ or
$\kappa(x,y,z) < 2$. 
Suppose first that $\kappa(x,y,z) > 2$. Then
$\det(B) < 0$. Since the diagonal entries of $B$ are positive, 
$\BB$ is indefinite of signature $(2,1)$. In particular, it cannot
be negative definite. 
In this case the triple of lines corresponding to $\rho$ are all
coplanar. The corresponding involutions reverse orientation on $P$
and act by reflections of $P$ in the three geodesics respectively.

When $\kappa(x,y,z) < 2$, there are two cases: either $\BB$ is positive
definite (signature $(3,0)$) or indefinite (signature $(1,2)$).
The restriction of $\BB$ to the coordinate plane spanned by 
$e_i$ and $e_j$ is given by the $2\times 2$ symmetric matrix
\begin{equation*}
\bmatrix 1 & B_{ij} \\  B_{ij} & 1 \endbmatrix  
\end{equation*}
which is positive definite if and only if $-1 < B_{ij} < 1$.
Thus $\BB$ is positive definite if and only if 
$-2 < x,y,z < 2$.

Otherwise $\BB$ is indefinite and $\rho$ corresponds to a
representation in $\SOot$.
In this case the triple of lines in $\Hth$ are all orthogonal
to the invariant plane $P$ in $\Hth$. 
The corresponding three involutions preserve orientation on  $P$ and
act by symmetries about points in $P$.

\subsubsection*{The two-dimensional normal form.}
\index{normal form for characters of $\F_2$}
Another approach to finding a representation with given traces 
involves a direct computation with the explicit normal form 
\eqref{eq:explicitrep} as follows. Let $(x,y,z)$ be as above.
First solve  $z = 2\cos(\theta)$ to obtain representative matrices:
\begin{equation*}
\xi_x := \bmatrix x &  -1  \\ 1 & 0 \endbmatrix,\;
\eta_{y,\theta} := \bmatrix 0 &  e^{-i\theta}  \\ -e^{i\theta} & y \endbmatrix,\; 
\end{equation*}
with a slight change of notation from \eqref{eq:explicitrep}.

A Hermitian form on $\C^2$ is given by a Hermitian $2\times 2$-matrix $H$. 
A complex $2\times 2$ matrix $H$ is Hermitian $\Longleftrightarrow H = \bar{H}^\dag$.
The corresponding Hermitian form on $\C^2$ is:
\begin{equation*}
(u,v) \longmapsto   \bar{v}^\dag H u,
\end{equation*}
where $u,v\in\C^2$.
A linear transformation $\C^2\xrightarrow{\xi}\C^2$ preserves $H$ 
$\Longleftrightarrow \bar{\xi}^\dag H  \xi = H$.

The $\xi_x$-invariant Hermitian forms comprise the real vector space with basis
\begin{equation*}
\bmatrix 2 & x \\ x  & 2 \endbmatrix,  \bmatrix 0 &  i \\  - i & 0 \endbmatrix 
\end{equation*}
and the $\eta_{y,\theta}$-invariant Hermitian forms comprise the real vector space with basis
\begin{equation*}
\bmatrix 2 & -y \sec(\theta) \\ -y \sec(\theta) & 2 \endbmatrix,
\bmatrix 0 & i - y \tan(\theta) \\ -i -y  \tan(\theta) & 0 \endbmatrix.
\end{equation*}
The $\rho$-invariant Hermitian forms thus comprise the intersection of these two 
vector spaces, the vector space with basis
\begin{equation*}
H = \bmatrix 2\sin(\theta) & x \sin(\theta) -i (y + xz/2) \\ 
x \sin(\theta) + i (y + xz/2) & 2\sin(\theta) \endbmatrix. 
\end{equation*}
This Hermitian matrix is definite since its determinant is positive:
\begin{equation*}
\det(H) = 4\sin^2(\theta) - x^2 \sin^2(\theta) - (y - xz/2)^2 =
2 - \kappa(x,y,z) > 0   
\end{equation*}
(since $\sin^2(\theta) = 1 - (z/2)^2$ and $\kappa(x,y,z)<2$).

\clearpage
\begin{figure}[t]
\centerline{\epsfxsize=1.5in 
\epsfbox{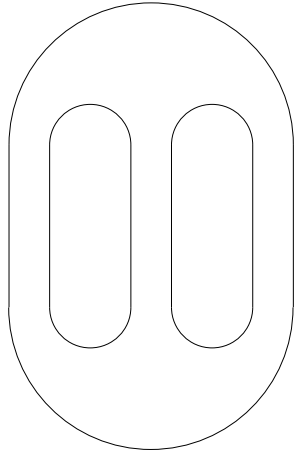}}
\caption{A ribbon graph for a three-holed sphere}
\label{fig:ribbongraph2}
\end{figure}

\begin{figure}
\centerline{\epsfxsize=1.5in 
\epsfbox{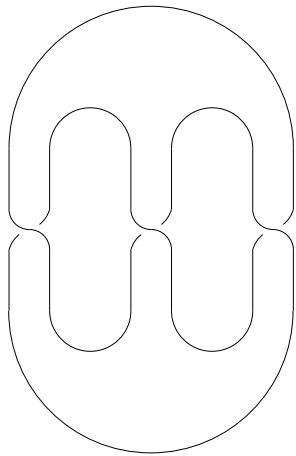}}
\caption{A ribbon graph for a one-holed torus}
\label{fig:ribbongraph3}
\end{figure}

\begin{figure}
\centerline{\epsfxsize=1.5in 
\epsfbox{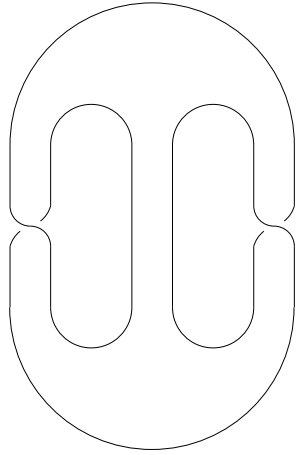}}
\caption{A ribbon graph for a two-holed cross-surface}
\label{fig:ribbongraph5}
\end{figure}

\begin{figure}
\centerline{\epsfxsize=1.5in 
\epsfbox{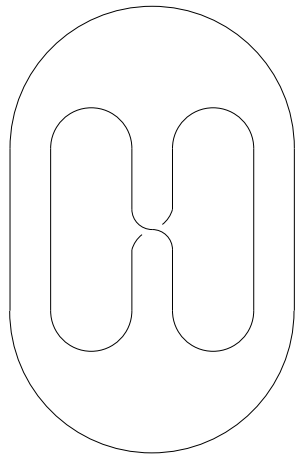}}
\caption{A ribbon graph for a one-holed Klein bottle}
\label{fig:ribbongraph4}
\end{figure}
\index{ribbon graphs}
\clearpage

\section{Hyperbolic structures on surfaces of $\chi=-1$}
We apply this theory to compute, in trace coordinates,
the deformation spaces of hyperbolic structures on compact 
connected surfaces $\Sigma$ with $\chi(\Sigma) = -1$. 
Equivalently, such surfaces are characterized by the condition
that $\pi_1(\Sigma)$ is a free group of rank two.
There are four possibilities:
\begin{itemize}
\item $\Sigma$ is homeomorphic to a three-holed sphere (a ``pair-of-pants'' or ``trinion'')
$\Sigma_{0,3}$;
\item $\Sigma$ is homeomorphic to a one-holed torus $\Sigma_{1,1}$;
\item $\Sigma$ is homeomorphic to a one-holed Klein bottle $C_{1,1}$;
\item $\Sigma$ is homeomorphic to a two-holed projective plane $C_{0,2}$.
\end{itemize}
Each of these surfaces can be realized as a ribbon graph with
three bands connecting two $2$-cells. The number of boundary
components and the orientability can be read off from the
parities of the number of twists in the three bands.

\subsection{Fricke spaces.}
\index{Fricke spaces}

The {\em Fricke space\/} $\Fricke(\Sigma)$ is the space of
isotopy classes of marked hyperbolic structures on $\Sigma$ 
with $\partial\Sigma$ geodesic. 
The group of isometries of $\Ht$ equals $\PGLtR$,
which embeds in $\PSLtC$.
Its identity component $\PSLtR$ consists of the
isometries of $\Ht$ which {\em preserve an orientation\/}
on $\Ht$.
The {\em holonomy map\/} embeds
$\Fricke(\Sigma)$ in the deformation space
\begin{equation*}
\Hom\big(\pi_1(\Sigma),\PGLtR\big)//\PGLtR.
\end{equation*}
Since $\partial\Sigma\neq\emptyset$, 
$\pi_1(\Sigma)$ is a free group, and
the problem of lifting a representation of $\pi_1(\Sigma)$ to $\GLtR$ is unobstructed. 

The various lifts are permuted by the group $H^1(\Sigma;\Z/2)$,
which is isomorphic to $\Z/2\oplus\Z/2$ when $\chi(\Sigma)=-1$. 
In terms of trace coordinates on the $\R$-locus of 
the character variety this action is given by:
\begin{equation*}
\bmatrix x \\ y \\ z \endbmatrix, \; 
\bmatrix x \\ -y \\ -z \endbmatrix, \; 
\bmatrix -x \\ y \\ -z \endbmatrix, \; 
\bmatrix -x \\ -y \\ z \endbmatrix.
\end{equation*}

\begin{theorem}\label{thm:frickechione}
Using trace coordinates of the boundary, the Fricke space of the
three-holed sphere $\Sigma_{0,3}$ identifies with the quotient of the four octants 
\begin{align*}
(-\infty,-2] \times (-\infty,-2] \times (-\infty,-2] 
& \coprod 
(-\infty,-2] \times [2,\infty)\times [2,\infty) \\
& \coprod 
[2,\infty) \times [2,\infty) \times (-\infty,-2] \\
& \coprod 
[2,\infty) \times (-\infty,-2] \times [2,\infty)
\subset \R^3
\end{align*}
by $H^1(\Sigma,\Z/2)$.
The octant 
\begin{equation*}
(-\infty,-2] \times (-\infty,-2] \times (-\infty,-2] 
\end{equation*}
defines a slice for the $H^1(\Sigma,\Z/2)$-action.
\end{theorem}
\noindent
The proof will be given in \S\ref{sec:subsecthreeholedsphere}. 

\begin{theorem}\label{thm:frickechione2}
The Fricke space of the one-holed torus $\Sigma_{1,1}$ identifies with 
the quotient of 
\begin{equation*}
\kappa^{-1}\big((-\infty,-2]\big) =  \{(x,y,z)\in\R^3 \mid x^2 + y^2 + z^2 - x y z \le 0\}
\end{equation*}
by $H^1(\Sigma,\Z/2)$. The region
\begin{equation*}
\{(x,y,z)\in (2,\infty)^3 \mid x^2 + y^2 + z^2 - x y z \le 0\}
\end{equation*}
is a connected component of 
$\kappa^{-1}\big((-\infty,-2]\big)$, 
defines a slice for the $H^1(\Sigma,\Z/2)$-action,
and hence identifies with the Fricke space of $\Sigma$.
\end{theorem}
\noindent
The proof will be given in \S\ref{sec:oneholedtorus}. 

\subsection{Two-dimensional hyperbolic geometry}
We take for our model of the {\em hyperbolic plane\/} $\Ht$ the subset of $\Hth$ comprising
quaternions $z + u\jj$, where $u>0$ and $z\in\R$. A matrix
\begin{equation*}
A = \bmatrix a & b \\ c & d \endbmatrix \in \GLtC 
\end{equation*}
determines a projective transformation of $\CP^1 =\partial\Hth$ which extends to
an orientation-preserving isometry of $\Hth$.
This isometry preserves $\Ht$ if and only if 
it is a scalar multiple of a real matrix. 
As usual, normalize $A\in\GLtC$ by dividing by a square root $\sqrt{\det(A)}\in\C^*$.
An element of 
\begin{equation*}
\SLtC \;\cap\; \C^*\GLtR 
\end{equation*}
 is either:
\begin{itemize}
\item a real matrix of determinant $1$, or
\item a purely imaginary matrix $i A'$ where $A'\in\GLtR$ satisfies $\det(A') = 1$. 
\end{itemize}
In the first case the corresponding orientation-preserving isometry of $\Hth$ preserves orientation
on $\Ht$, and in the second case its restriction  {\em reverses orientation\/} on $\Ht$.
The traces of their representative matrices in $\SLtC$
distinguish these cases. We emphasize that these representatives are only determined up to $\pm 1$.
Suppose that $A\in \PSLtC$ preserves a plane $P\subset\Hth$ and let $\tilde{A}\in\SLtC$ be a lift of $A$.
Then:

\begin{itemize}
\item The restriction of $A$ to $P$ preserves orientation 
$\Longleftrightarrow$
$\tr(\tilde{A}) \in\R$;
\item The restriction of $A$ to $P$ reverses orientation 
$\Longleftrightarrow$
$\tr(\tilde{A}) \in i\R$.
\end{itemize}
Observe that an element $A\in \SLtC$ whose trace is both real and purely imaginary --- that is, equals zero ---
is an involution in a line $\ell := \Fix(A)$. The $A$-invariant planes fall into two types:
those which contain $\ell$, upon which $A$ reverses orientation, and those which are orthogonal to $\ell$, 
upon which $A$ preserves orientation.

\subsubsection*{The hyperbolic plane and involutions of $\Hth$.}
\index{half-planes}
The proof of Theorem~\ref{thm:frickechione} 
requires an algebraic representation of half-planes in $\Ht$. 
Given an orientation on $\Ht$, and an oriented geodesic $\ell\subset\Ht$,
there is a well-determined half-plane bounded by $\ell$, defined as follows.
Let $x\in\ell$. Choose the unit vector $v_\ell$ tangent to $\ell$ at $x$ 
determined by the orientation of $\ell$. 
The choice of half-plane $\hh\subset \Ht \setminus \ell$ is determined
by the normal vector $\nu$ to $\ell$ at $x$ pointing outward from $\hh$.
Choose $\hh$ so that the basis $\{v_\ell,\nu\}\subset T_x\Ht$ is positively
oriented.

The points of $\Ht$ identify with geodesics in $\Hth$ which are orthogonal
to $\Ht$; the endpoints of such geodesics are complex-conjugate elements
of $\CP^1 \setminus \RP^1$. 
An involution which interchanges these endpoints 
is given by matrices 
$\pm I_{x + \jj u}$, where the $2\times 2$ real matrix
\begin{equation}\label{eq:pointmatrix}
I_{x + \jj u} := \frac{1}{u} \bmatrix  x & -(x^2 + u^2) \\ 1 & -x \endbmatrix 
\end{equation}
has determinant $1$ and trace $0$.
(Apply \eqref{eq:involutionformula}, taking
$z_1 = x + i u$ and $z_2 = x - i u$.)
The space of such matrices has two components, depending on the signs
of the off-diagonal elements. 
A matrix 
\begin{equation*}
A\in\SLtR\cap\sltR 
\end{equation*}
equals $I_{x+ \jj u}$, for some $x + \jj u\in\Ht$
if and only if $A_{12} < 0 < A_{21}$, in which case
\begin{equation*}
u = (A_{21})^{-1},\; x = (A_{21})^{-1} A_{11}.
\end{equation*}
The above inequalities determine one of the two sheets of the
two-sheeted hyperboloid in $\sltR$ defined by the
unimodularity condition $\det(A) = 1$.

This gives a convenient form of the Klein hyperboloid model for $\Ht$,
as the quadric in $\sltR\cong\R^3$ defined by
\begin{equation*}
1 = \det(A) = -\frac12 \tr(A^2).
\end{equation*}

Next we represent oriented geodesics by involutions.
A geodesic in $\Ht$ determines an involution 
by \eqref{eq:involutionformula}
and \eqref{eq:involutionformula2}, 
where $z_1,z_2$ are distinct points in $\RP^1$.
Such a matrix is purely imaginary, has trace zero
and determinant one. Multiplying by $i$, we obtain
an element $A$ of $\sltR$ which has determinant $-1$.
Such a matrix has well-defined $1$-dimensional 
eigenspaces with eigenvalues $\pm i$. These eigenspaces
determine the respective fixed points in $\RP^1$.
Replacing $A$ by $-A$ interchanges the $\pm i$-eigenspaces.
In this way, we identify the 
space of 
oriented geodesics in $\Ht$ with 
\begin{equation*}
\{ A \in \sltR \mid \det(A) = -1 \} .
\end{equation*}
This is just the usual hyperboloid model. 
Think of $\sltR$ as a $3$-dimensional real
inner product space under the inner product
\begin{equation*}
\langle A, B \rangle := \frac12 \tr(AB).
\end{equation*}
The corresponding quadratic form relates to the
{\em determinant\/} by
\begin{equation*}
\langle A, A \rangle = \frac12 \tr(A^2) = -\det(A).
\end{equation*}
This quadratic form is readily seen to have signature $(2,1)$
since its value on 
\begin{equation*}
\bmatrix a & b \\ c & -a \endbmatrix
\end{equation*}
equals $ a^2 +  b c$.
Then $\Ht$ corresponds to one component of the 
two-sheeted hyperboloid (say the one with $b < 0 < c$)
\begin{equation*}
\{ v\in \sltR \mid \langle v,v\rangle = -1  \}
\end{equation*}
and the space of oriented geodesics corresponds to
{\em de Sitter space\/}
\begin{equation*}
\dst :=  \{ v\in \sltR \mid \langle v,v\rangle = 1  \}.
\end{equation*}
\index{de Sitter space}
A vector $v\in\dst$ determines a half-plane $\Hh(v)$ by:
\begin{equation}\label{eq:halfplane}
\Hh(v) := \{ w\in\Ht \mid \langle w,v\rangle \ge 0  \}.
\end{equation}
In particular, $\Hh(-v)$ is the half-plane complementary 
to $\Hh(v)$.

For example, the half-plane corresponding to
\begin{equation*}
\bmatrix 1 & 0 \\ 0 & -1 \endbmatrix 
\end{equation*}
consists of all $x + u\jj$ where $x \ge 0$, 
as can easily be verified using 
\eqref{eq:pointmatrix}.

The main criterion for disjointness of half-planes is the following lemma,
whose proof is an elementary exercise and left to the reader.
(Recall that two geodesics in $\Ht$ are {\em ultraparallel\/} if and only
if they admit a common orthogonal geodesic; equivalently distances between 
their respective points have a positive lower bound.) 

\begin{lemma} Let $v_1,v_2\in\dst$ determine geodesics $\ell_1,\ell_2\subset\Ht$
and half-planes  $\Hh_i := \Hh(v_i)$ with $\partial \Hh_i = \ell_i$.
The following conditions are equivalent:
\begin{itemize} 
\item $\vert\langle v_1, v_2\rangle\vert  > 1$;
\item The invariant geodesics $\ell_1$ and $\ell_2$ are ultraparallel.
\end{itemize} 
In this case, the following two further conditions are equivalent:
\begin{itemize} 
\item $\langle v_1, v_2 \rangle  > 1$;
\item Either $\Hh_1 \subset \Hh_2$ or $\Hh_2 \subset \Hh_1$.
\end{itemize} 
Contrariwise, the following two conditions are equivalent:
\begin{itemize} 
\item $\langle v_1, v_2 \rangle  <  -1$;
\item Either $\Hh_1$ and $\Hh_2$ are disjoint or their
complements are disjoint. 
\end{itemize} 
\end{lemma}

\subsubsection*{Hyperbolic isometries.}
\index{hyperbolic isometries}
An element $A\in\SLtR$ is {\em hyperbolic\/} if it satisfies
any of the following equivalent conditions:
\begin{itemize}
\item $\tr(A) > 2$  or $\tr(A) < -2$;
\item $A$ has distinct real eigenvalues;
\item The isometry of $\Ht$ defined by $A$ has exactly two fixed points
on $\partial \Ht$;
\item The isometry of $\Ht$ defined by $A$ leaves invariant a 
(necessarily unique) geodesic $\ell_A$, upon which it acts by 
a nontrivial translation.
\end{itemize}
A geodesic $\ell$ in $\Ht$ is specified by its {\em reflection
$\rho_\ell$,\/} an isometry of $\Ht$ whose fixed point set equals
$\ell$. If $v\in\dst$ is a vector corresponding to $\ell$, then
$\rho_\ell$ is the restriction to $\Ht$ of the orthogonal reflection in
$\SOto$ fixing $v$:
\begin{equation*}
u \stackrel{\rho_v}\longmapsto  - u + 2 \langle u,v\rangle v. 
\end{equation*}
\eqref{eq:invformula} implies
the following useful formula for the invariant axis of a hyperbolic element:
\begin{lemma}
Let $A$ be hyperbolic. Then 
\begin{equation}\label{eq:hat}
\widehat{ A } \; := \;  
\frac{2 A - \tr(A) \Id}{\sqrt{\tr(A)^2 - 4}} 
\;\in\; \widetilde{\Inv}\cap\SLtR 
\end{equation}
defines the reflection in the invariant axis of $A$.
\end{lemma}
\noindent
Notice that 
\begin{equation*}
\widehat{-A} \; = \; \widehat{A^{-1}} \; = \;   -\widehat{A}
\end{equation*}
so that $A$ and $A^{-1}$ determine
complementary half-planes.

For example  
\begin{equation*}
A =  \bmatrix e^{l/2} & 0 \\ 0 & e^{-l/2} \endbmatrix 
\end{equation*}
represents translation along a geodesic (the imaginary axis in
$\Ht$) by distance $l > 0$ from $0$ to $\infty$. 
The corresponding reflection is
\begin{equation*}
\widehat{A} =   \bmatrix 1 & 0 \\ 0 & -1 \endbmatrix.
\end{equation*}
which determines the half-plane:
\begin{equation*}
\Hh(\widehat{A}) = \{ x + u \jj \in\Ht \mid  x\ge 0, u > 0\}
\end{equation*}
as above.

\subsection{The three-holed sphere.}\label{sec:subsecthreeholedsphere}
\index{Fricke space of three-holed sphere}
We now show that a representation corresponding to a character
$(x,y,z)\in \R^3$ satisfying $x, y, z< -2$
is the holonomy representation of a hyperbolic structure on a 
three-holed sphere. 
We find matrices $X,Y,Z$ of the desired type and compute the corresponding
reflections $\widehat{X},\widehat{Y},\widehat{Z}$. Then we show that
the corresponding half-planes are all disjoint (after possibly replacing
$\widehat{X},\widehat{Y},\widehat{Z}$ by their negatives).
From this we construct a developing map for a hyperbolic structure
on $\Sigma$.
For details on geometric structures on manifolds and their
developing maps, see
Goldman~\cite{geost,puncturedtorus,PGOM} or Thurston~\cite{Thurston}.

\begin{lemma}
Suppose $X,Y,Z\in\SLtC$ satisfy $X Y Z = \Id$ and have real traces
$x,y,z<-2$ respectively.
Then the inner  products
\begin{equation*}
\langle \widehat{X}, \widehat{Y}\rangle, 
\langle \widehat{Y}, \widehat{Z}\rangle, 
\langle \widehat{Z}, \widehat{X}\rangle 
\end{equation*}
are all $< -1$.
\end{lemma}
\begin{proof}
The proof breaks into a series of calculations. 
By symmetry it suffices to prove
$\langle \widehat{X}, \widehat{Y}\rangle < -1$. 
By the definition  \eqref{eq:hat}
\begin{align}
\langle \widehat{X}, \widehat{Y}\rangle  & \;=\;  \frac12
\tr\bigg( \frac{2 X - x\Id}{\sqrt{x^2-4}}\; \frac{2 Y - y\Id}{\sqrt{y^2-4}}\bigg)
\notag \\
& \; = \;  \frac{\tr\big(\, 4 X Y - 2 x Y - 2 y X + x y \Id\big)}
{2\sqrt{(x^2-4)(y^2-4)}}  \notag\\
& \; = \; 
\frac{2 z - x y}{\sqrt{(x^2-4)(y^2-4)}} \label{eq:tracesdot}
\end{align}
since $\tr(XY)= z$, $\tr(Y)= y$, $\tr(X) = x$ and $\tr(\Id) = 2$.
Because 
\begin{equation*}
x,y,z< -2 \;\Longrightarrow\; 2 z - x y < 0, 
\end{equation*}
the calculation above implies
\begin{equation}\label{eq:dotneg}
\langle \widehat{X}, \widehat{Y}\rangle  < 0.
\end{equation}
Now
\begin{equation*}
x^2 + y^2 + z^2 - x y z - 4 \;>\; 4 + 4 + 4 - 8 -4 \;=\;  0 
\end{equation*}
implies that 
\begin{equation*}
(2 z - x y)^2 - (x^2 -4)(y^2 -4) \;=\; 4 (x^2 + y^2 + z^2 - x y z - 4)  \;>\; 0  
\end{equation*}
and
\begin{equation}\label{tzmxy}
\bigg(\frac{2 z - x y}{\sqrt{(x^2-4)(y^2-4)}}\bigg)^2 > 1.
\end{equation}
Thus \eqref{eq:tracesdot} and  \eqref{tzmxy} imply
\begin{equation*}
\langle \widehat{X}, \widehat{Y}\rangle^2 > 1
\end{equation*}
whence \eqref{eq:dotneg} implies:
\begin{equation*}
\langle \widehat{X}, \widehat{Y}\rangle <  -1 
\end{equation*}
as claimed.
\end{proof}

\begin{proof}[Conclusion of proof of Theorem~\ref{thm:frickechione}]

Thus the half-planes $\Hx,\Hy,\Hz$ are either all disjoint or their complements
are all disjoint. Replacing 
$\widehat{X},\widehat{Y},\widehat{Z}$ by their negatives if necessary,
assume that the complements to
$\Hx,\Hy,\Hz$ are pairwise disjoint.

The intersection
\begin{equation*}
\Delta_\infty := \Hx \cap \Hy \cap \Hz
\end{equation*}
is bounded by the three geodesics 
\begin{equation*}
\ell_X = \partial\Hx, \quad
\ell_Y = \partial\Hy, \quad
\ell_Z = \partial\Hz 
\end{equation*}
and three segments of $\partial\Ht$. 
When some of $\rho(X),\rho(Y),\rho(Z)$
are parabolic, then these segments degenerate into ideal points.
If $a,b$ are lines or ideal points, denote their common orthogonal
segment by $\perp(a,b)$. Define:
\begin{align*}
\sigma_{XY} &:= \perp(\ell_X,\ell_Y) \\  
\sigma_{YZ} &:= \perp(\ell_Y,\ell_Z) \\  
\sigma_{ZX} &:= \perp(\ell_Z,\ell_X). 
\end{align*}
Let $\hexagon_\rho\subset\Delta_\infty$ 
denote the right hexagon bounded by
$\sigma_{XY}, \sigma_{YZ}, \sigma_{ZX}$ 
and segments of 
$\ell_X,\ell_Y,\ell_Z$ as in
Figure~\ref{fig:pantsdomain}. 
Map the abstract hexagon $\hexagon$ of \S\ref{sec:hex}
to $\hexagon_\rho$ so that 
\begin{align*}
\partial_1(\hexagon)  &\longmapsto \ell_X \\
\partial_2(\hexagon)  &\longmapsto \ell_Y \\
\partial_3(\hexagon)  &\longmapsto \ell_Z 
\end{align*}
and the other three edges of $\partial\hexagon$
map homeomorphically 
to $\sigma_{XY}, \sigma_{YZ}, \sigma_{ZX}$ 
respectively. This mapping embeds $\hexagon$ into $\Ht$.

A fundamental domain for the action of 
\begin{equation*}
\pi_1(\Sigma) = \langle X,Y,Z \mid X Y Z = 1\rangle 
\end{equation*}
on the universal covering surface
\begin{equation*}
\tilde\Sigma := \bigg(\hexagon \times \hat\pi\bigg)/\sim
\end{equation*}
is the union 
\begin{equation*}
\Delta := \hexagon \cup \iota_{XY}(\hexagon).  
\end{equation*}
We shall extend the 
embedding $\hexagon\hookrightarrow\Ht$ 
to a local diffeomorphism (a {\em developing map\/})
\begin{equation*}
\tilde\Sigma \longrightarrow \Ht
\end{equation*}
which is $\pi_1(\Sigma)$-equivariant with respect to the
action $\rho$ on $\Ht$. Pull back the hyperbolic structure
from $\Ht$ to obtain a $\pi_1(\Sigma)$-invariant hyperbolic
structure on $\Sigma$, as desired.

Extend $\hexagon\hookrightarrow\Ht$
to $\Delta\longrightarrow\Ht$ as follows.
Map the reflected image of $\hexagon$ to
$\iota_{XY}\hexagon_\rho$.  
Then 
\begin{equation*}
X = \iota_{ZX}\iota_{XY} 
\end{equation*}
identifies
the two sides of $\Delta$ corresponding to
$\sigma_{ZX}$ and $\iota_{XY}\sigma_{ZX}$,
and
\begin{equation*}
Y = \iota_{XY}\iota_{YZ} 
\end{equation*}
identifies the two sides of $\Delta$ corresponding to
$\iota_{XY}\sigma_{YZ}$ and $\sigma_{YZ}$.
(Compare Figure~\ref{fig:pantsdomain}.)
This defines a hyperbolic structure on $\Sigma$
with geodesic boundary developing to $\ell_X$,
$\ell_Y$ and $\ell_Z$.

This completes the proof that every character
$(x,y,z)\in \big(-\infty,2]^3$ is the holonomy
of a hyperbolic structure on $\Sigma_{0,3}$.

\end{proof}

\begin{figure}[h]
\centerline{\epsfxsize=2.5in \epsfbox{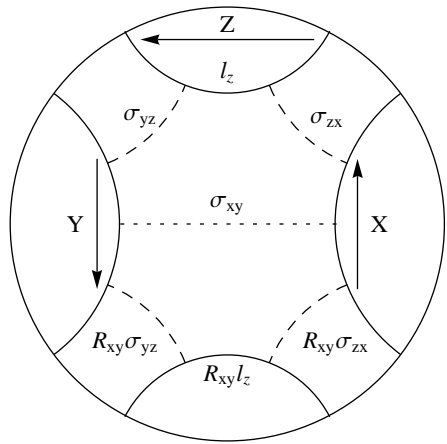}}
\caption{Fundamental domain for hyperbolic structure on $\Sigma_{0,3}$.}
\label{fig:pantsdomain}
\end{figure}
 
\subsection{The one-holed torus.} \label{sec:oneholedtorus} 
\index{Fricke space of one-holed torus}
Now consider the case $\Sigma \approx \Sigma_{1,1}$.
Present $\pi=\pi_1(\Sigma)$ as freely generated by $X,Y$
corresponding to simple closed curves which intersect
transversely in one point. Then the boundary $\partial \Sigma$
corresponds to the commutator $K = [X,Y]$ and we obtain the presentation
\begin{equation*}
\pi \;=\; \langle X,Y,Z,K \;\mid\; X Y Z = \Id, \quad X Y  = K Y X \rangle
\end{equation*}
The corresponding trace functions are
\begin{align*}
x([\rho]) & :=\tr\big(\rho(X)\big) \\
y([\rho]) & :=\tr\big(\rho(Y)\big) \\
z([\rho]) & :=\tr\big(\rho(Z)\big) \\
k([\rho]) & :=\tr\big(\rho(K)\big)  \\ 
& = \kappa(x,y,z) = x^2 + y^2 + z^2 - x y z -2,
\end{align*}
which we denote by $x,y,z,k$ (without reference to $\rho$) when the context is clear.

The goal of this section is to prove Theorem~\ref{thm:frickechione2}.

\begin{lemma}\label{lem:biggerthantwo}
Suppose $(x,y,z)\in\R^3$ satisfies 
$x^2 + y^2 + z^2 - x y z < 4$. 
Then either:
\begin{itemize} \item $(x,y,z)\in[-2,2]^3$ 
or \item 
$\vert x\vert, \vert y\vert, \vert z\vert > 2.$ 
\end{itemize}
\end{lemma}
\noindent
In the first case $(x,y,z)$ is the character of an 
$\SUt$-representation as in Theorem~\ref{thm:realcharacters}.
\begin{proof}
Rewriting the hypothesis as
\begin{equation*}
(x^2 -4) (y^2 -4) > (2 z - xy)^2, 
\end{equation*}
it follows that $(x^2 -4) (y^2 -4) > 0$.
By symmetry $(y^2 -4) (z^2 -4) > 0$ and $(z^2 -4) (x^2 -4) > 0$
as well. 
Thus none of $\vert x\vert$, $\vert y\vert$, $\vert z\vert$ equal
$2$, and either 
\begin{equation*}
x^2 -4, \quad y^2 -4, \quad z^2 -4 
\end{equation*}
are all positive or all negative.
If 
\begin{equation*}
\vert x\vert, \vert y\vert, \vert z\vert < 2 
\end{equation*}
then $(x,y,z)$ is an $\SUt$-character
as in Theorem~\ref{thm:realcharacters};
otherwise 
\begin{equation*}
\vert x\vert, \vert y\vert, \vert z\vert > 2
\end{equation*}
as desired.
This completes the proof of Lemma~\ref{lem:biggerthantwo}.
\end{proof}

\begin{figure}[h]
\centerline{\epsfxsize=2.5in 
\epsfbox{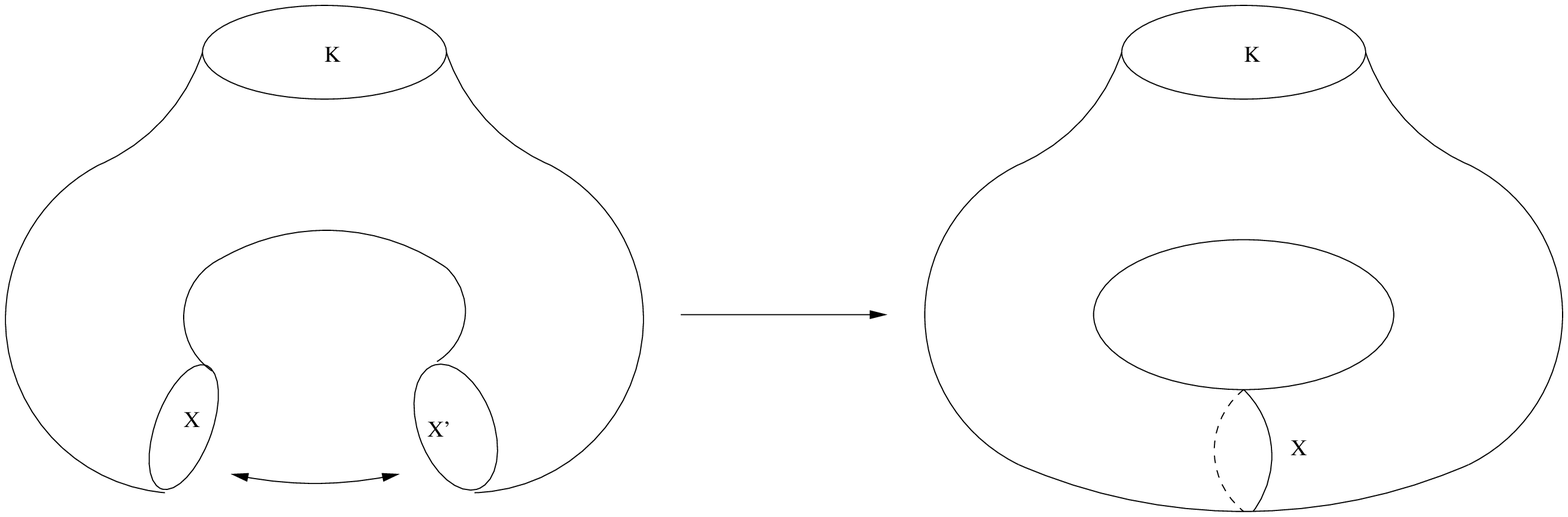}}
\caption{The one-holed torus as an identification space.
The identification map $Y$ conjugates $X$ to a boundary element
of $\pi_1(\Sigma')$, but with the opposite orientation.}
\label{fig:oneholedtorus}
\end{figure}

Denote by $X\subset\Sigma$ 
a simple closed curve corresponding to the 
generator $X\in\pi_1(\Sigma)$. 
The surface-with-boundary $\Sigma' := \Sigma | X$ 
obtained by {\em splitting\/} $\Sigma$
along $X$ is homeomorphic to a three-holed sphere.
Denoting the quotient map by $\Sigma'\xrightarrow{\Pi}\Sigma$,
the three components of $\partial\Sigma'$ are the connected
preimage $\partial' := \Pi^{-1}(\partial\Sigma)$ and the two components
$X_\pm$ of the preimage  
$\Pi^{-1}(X)$. 
Choose 
arcs from the basepoint to $X_+$  and represent the
boundary generators of $\pi_1(\Sigma')$ by the elements
$\partial',X_+,X_-$ subject to the relation
$X_-X_+\partial' = \Id$.
The quotient map induces a monomorphism
\begin{align*}
\pi_1(\Sigma') & \stackrel{\Pi_*}\hookrightarrow \pi_1(\Sigma) \\
X_+  & \longmapsto X \\
X_- & \longmapsto Y X^{-1} Y ^{-1} \\
\partial' & \longmapsto \partial = X^{-1} Y X Y^{-1}.
\end{align*}
Compare Figure~\ref{fig:oneholedtorus}.

\begin{lemma}
The composition $\rho\circ\Pi_*$ is the holonomy
representation of a hyperbolic structure on
$\Sigma'\approx\Sigma_{0,3}$.
\end{lemma}
\begin{proof}
By Theorem~\ref{thm:frickechione}, it suffices to show
that the images of the boundary generators 
$\partial',X_-,X_+\in\pi_1(\Sigma')$  under
$\rho\circ\Pi_*$ have trace $\le -2$.
By Lemma~\ref{lem:biggerthantwo},
\begin{equation*}
\tr\big(\Pi_*\circ\rho(X_-)\big) =
\tr\big(\Pi_*\circ\rho(X_+)\big) = x 
\end{equation*}
is either $>2$ or $<-2$. In the former case,
replace $\rho(X)$ by $-\rho(X)$ to assume that
$x< -2$. Now by assumption
\begin{equation*}
\tr\big(\Pi_*\circ\rho(\partial')\big) =
\tr\big(\rho(K)\big) = k \le - 2 
\end{equation*}
so that all three boundary generators
of $\pi_1(\Sigma')$ have trace $\le -2$,
as desired.
\end{proof}
\begin{proof}[Conclusion of proof of Theorem~\ref{thm:frickechione2}]
Thus we obtain a hyperbolic structure on $\Sigma'$
with geodesic boundary. Choosing a developing map
with holonomy $\Pi_*\circ\rho$, the isometry $\rho(Y)$
realizes the identification of $\xi_-$ with $\xi_+$ for
which $\Pi$ is the quotient map.
The resulting quotient space is homeomorphic to $\Sigma$
and inherits a hyperbolic structure from the one on $\Sigma'$
and the identification. Therefore $\rho$ is the holonomy
representation of a hyperbolic structure on $\Sigma$.
This concludes the proof of
Theorem~\ref{thm:frickechione2}.
\end{proof}
Compare Goldman~\cite{puncturedtorus} for a different proof.

The algebraic methods discussed here easily imply several
other qualitative geometric facts:

\begin{proposition}\label{prop:axescross}
Suppose that $x,y,z\in\R$ satisfy 
$\kappa(x,y,z)\le -2$ and $(x,y,z)\neq (0,0,0)$.
Then $(x,y,z)$ is the character of a representation
$\pi\xrightarrow{\rho}\SLtR$
and $\rho(X)$ and $\rho(Y)$ are hyperbolic elements
whose axes cross.
\end{proposition}
\begin{proof}
Since $\kappa(x,y,z)\le -2 < 2$, Lemma~\ref{lem:biggerthantwo}
applies. The representation $\rho$ with character $(x,y,z)$
is conjugate to either an $\SUt$-representation or an
$\SLtR$-representation. Since 
\begin{equation*}
\tr\big(\rho([X,Y])\big) \le -2, 
\end{equation*}
the only possibility for an $\SUt$-representation occurs if
$\kappa(x,y,z)=-2$. Then $\rho$ is conjugate to 
the quaternion representation

\begin{equation*}
\rho(X) = \bmatrix 0 & -1 \\ 1 & 0 \endbmatrix, \quad
\rho(Y) = \bmatrix i & 0 \\ 0 & -i \endbmatrix,
\end{equation*}
and $(x,y,z)=(0,0,0)$, a contradiction. 
(Compare \S 2.6 of Goldman~\cite{puncturedtorus}.)
Thus $(x,y,z)$ corresponds to an $\SLtR$-representation
$\rho$. Lemma~\ref{lem:biggerthantwo} implies that $\rho(X)$ and $\rho(Y)$
are both hyperbolic. 
Proposition~\ref{prop:Lie} implies that the 
involution fixing the common orthogonal
$\perp(\ell_{\rho(X)},\ell_{\rho(Y)})$ of their respective invariant
axes is given by the Lie product $\Lie\big(\rho(X),\rho(Y)\big)$.
In particular their axes cross if and only if 
$\perp(\ell_{\rho(X)},\ell_{\rho(Y)})$ is orthogonal to the real
plane $\Ht\subset \Hth$, that is, if 
$\Lie\big(\rho(X),\rho(Y)\big)$ defines an orientation-preserving 
involution of $\Ht$. This occurs precisely when the matrix
$\Lie\big(\rho(X),\rho(Y)\big)$ has positive determinant.
By \eqref{eq:twokindsofcommutator}, the Lie product
$\Lie\big(\rho(X),\rho(Y)\big)$ has positive determinant
if and only if the commutator trace 
\begin{equation*}
\tr([\rho(X),\rho(Y)])< 2, 
\end{equation*}
as assumed. The proof of Proposition~\ref{prop:axescross} is complete.
\end{proof}

\subsection{Fenchel-Nielsen coordinates.}

In an influential manuscript written in the early 20th century but
only recently published, Fenchel and Nielsen~\cite{FenchelNielsen}
gave geometric coordinates for Fricke space. 
We briefly relate these coordinates to trace coordinates 
for the surfaces of Euler characteristic $-1$.

\subsubsection*{Pants decompositions.}
\index{pants decompositions}
Decompose $\Sigma$ into a union of three-holed spheres (``pants'')
\begin{equation*}
P_1,\dots, P_l 
\end{equation*}
along a system of $N$ disjoint simple curves 
\begin{equation*}
\gamma_1,\dots, \gamma_N\subset\mathsf{int}(\Sigma).  
\end{equation*}
Let $\partial_1,\dots,\partial_n$ denote components of $\partial\Sigma$.
For a given marked hyperbolic structure on $\Sigma$,
the curves $\gamma_i,\partial_j$ may be taken to be simple closed
geodesics. Theorem~A implies that the isometry type of the 
complementary subsurfaces $P_k$ are determined by the lengths of
the three closed geodesics representing $\partial P_k$.
The resulting map
\begin{align*}
\Fricke(\Sigma) & \xrightarrow{F} (\R_+)^N \times (\R_{\ge 0})^n \\
\langle M\rangle & \longmapsto  \ell_M(\gamma_i) \times \ell_M(\partial_j) 
\end{align*}
which associates to a hyperbolic surface $M$ the lengths of the geodesics 
$\gamma_i,\partial_j$ is onto. Its fibers correspond to the various ways
in which the subsurfaces $P_k$ are identified along interior curves $\gamma_i$. 

Choose a section $\sigma$ of the map $F$ as follows.
Each interior curve $\gamma$ bounds two subsurfaces, which we denote $P'$ and $P''$.
The corresponding boundary curves are denoted $\gamma'\subset P'$ and
$\gamma''\subset P''$ respectively. 
The {\em twist parameter\/}  $\tau_i\in\R$ represents the displacement 
between points on the marked surfaces $P, P'$ corresponding to the section $\sigma$.
This realizes the Fenchel-Nielsen map $F$ as a principal $\R^N$-bundle over
$\Fricke(\Sigma)$. Wolpert~\cite{Wolpert1,Wolpert2,Wolpert3} shows that, when $\Sigma$ 
is closed and orientable, the Fenchel-Nielsen coordinates on $\Fricke(\Sigma)$ are
{\em canonical\/} or {\em Darboux\/} coordinates for the symplectic structure arising
from the Weil-Petersson K\"ahler form on Teichm\"uller space, and indeed $F$ is
a moment map for a completely integrable system on $\Fricke(\Sigma)$.

In the orientable case, $\Sigma \approx \Sigma_{g,n}$, then
\begin{equation}\label{eq:numberofinteriorcurves} 
N = 3 (g -1) + n.  
\end{equation}
Since $\chi(\Sigma) = 2-2g + n$ and each $P_k$ has Euler characteristic $-1$, 
the number $l$ of subsurfaces $P_k$ equals 
\begin{equation*}
l = -\chi(\Sigma) = 2 + 2g + n. 
\end{equation*}
Consider the set $S$ of pairs $(\alpha,C)$, where $\alpha$ is one of the $N+n$ curves
$\partial_i,\gamma_j$ and $C\subset P_k$ is a collar neighborhood of $\alpha$ inside
$P_k$ for some $k$. 
Each $C$ lies in exactly one $P_k$ and each subsurface $P_k$ contains
exactly three pairs $(\alpha,C)$, the cardinality of $S$ equals $3 l$. Furthermore
the number of collars $C$ equals $2 N + n$, since each $\gamma_i$ is two-sided in $\Sigma$
and each $\partial_j$ is one-sided. Computing the cardinality of $S$ in two ways:
\begin{equation*}
2 N + n = 3 l = 3 (2 g - 2 + n) 
\end{equation*}
implies \eqref{eq:numberofinteriorcurves}.

The nonorientable case, say $\Sigma\approx C_{k,n}$, 
reduces to the orientable case by cutting along a disjoint
family of simple loops:
$k$ of them reverse orientation and 
$N = k + 2 - n$ preserve orientation.
This follows easily from the classification of
surfaces: $C_{k,n}$ can be obtained from the planar
surface $\Sigma_{0,k+n}$ by attaching $k$ cross-caps
(copies of $C_{0,k}$) to $k$ of the components of 
$\Sigma_{0,k+n}$. In the nonorientable surface 
$\Sigma \approx C_{k,n}$ are $k$ disjoint orientation-reversing
simple loops $s_1,\dots, s_k$ so that the surface 
$\Sigma'$ obtained by splitting $\Sigma$ along 
$s_1,\dots, s_k$ identifies to $\Sigma$. 
Denote the resulting quotient map by
\begin{equation*}
\Sigma' \approx \Sigma_{0,n+k} \xrightarrow{\phi} \Sigma \approx C_{k,n}.
\end{equation*}
Let $s_i'\subset \Sigma'$ denote the preimage $\phi^{-1}(s_i)$.
Given a hyperbolic structure on $\Sigma'$, there is a uniqe way of
extending this hyperbolic structure to $\Sigma$ as follows.  As usual,
assume that each $s_i'\subset \Sigma'$ is a closed geodesic. Choose
$\epsilon>0$ sufficiently small so that all the $\epsilon$-collars 
$N_\epsilon(s_i')$ of $s_i'$ are disjoint. Denote the complement
of these collars by
\begin{equation*}
\Sigma'' := \Sigma' \setminus \bigcup_{i=1}^k N_\epsilon(s_i').
\end{equation*}
Represent the geodesic $s_i'$ as the quotient $\tilde{s}/\langle \xi\rangle$,  
where $\xi\in\PSLtR$ is hyperbolic and $\tilde{s}\subset\Ht$ is the $\xi$-invariant
geodesic. That is, $\xi$ is a transvection along the geodesic $\tilde{s}$. 
Let $\sqrt{\xi}$ denote the unique {\em glide-reflection\/} whose square is
$\xi$; it is the composition of reflection in $\tilde{s}$ with the transvection
of displacement $\ell(\xi)/2$ where $\ell(\xi)$ is the displacement of $\xi$.
If a matrix representative of $\xi$ has trace $x > 2$, then 
$x = 2\cosh\big(\ell(\xi)/2\big)$ and a matrix representing the glide-reflection
$\sqrt{\xi}$ equals
\begin{equation*}
\frac1{\sqrt{x-2}} \big(\xi - \Id\big).
\end{equation*}
Let $N_\epsilon(\tilde{s})\subset\Ht$ be the tubular neighborhood of width $\epsilon$
about $\tilde{s}$. The quotient 
$N_\epsilon(\tilde{s})/\langle\sqrt{\xi}\rangle$ is a cross-cap bounded by a hypercycle
(equidistant curve).
The union
\begin{equation*}
\Sigma'' \bigcup N_\epsilon(\tilde{s})/\langle\sqrt{\xi}\rangle
\end{equation*}
is the desired hyperbolic structure on $\Sigma$.

\subsubsection*{Fenchel-Nielsen coordinates on $\Sigma_{1,1}$.}
\index{Fenchel-Nielsen coordinates for one-holed torus}
We relate the Fricke trace coordinates to Fenchel-Nielsen coordinates as follows. 
We suppose that the boundary $\partial\Sigma$ has length $b\ge 0$,
the case $b = 0$ corresponding to the complete finite-area structure
(where the holonomy around $\partial\Sigma$ is parabolic).

Suppose that $X\subset \Sigma$ has length $ l> 0$, and has holonomy represented by:
\begin{equation*}
\tilde\rho(X) := \bmatrix e^{l/2} & 0 \\ 0  & e^{-l/2}  \endbmatrix
\end{equation*}
where $x = 2 \cosh(l/2)$.  Then 
\begin{equation*}
\Fix\big(\rho(X)\big) = \{0,\infty\} .
\end{equation*}
A fundamental domain for the cyclic group $\langle \rho(X)
\rangle$ is bounded by the geodesics with endpoints 
$\pm e^{-l/2} $ and $\pm e^{l/2} $ respectively.

Normalize the twist parameter $\tau$ so that $\tau = 0$ corresponds to the case that the 
invariant axes $\ell_{\rho(X)}$, $\ell_{\rho(Y)}$ are orthogonal. 
In that case take $\Fix(\rho(Y)) = \pm 1$ and define 
\begin{equation*}
\tilde\rho_0(Y) :=   \bmatrix \cosh(\mu/2) & \sinh(\mu/2) \\ \
\sinh(\mu/2)     & \cosh(\mu/2)  \endbmatrix 
\end{equation*}
where $y = 2 \cosh(\mu/2)$. 

A fundamental domain for the cyclic group $\langle \rho(Y)\rangle$ 
is bounded by the geodesics with endpoints 
$-e^{\pm \mu/2} $ and $e^{\pm \mu/2} $
respectively. For this representation, 
\begin{equation*}
z = \tr\big(\tilde\rho(X)\tilde\rho_0(Y)\big) = x y/2 = 2 \cosh(l/2)\cosh(\mu/2). 
\end{equation*}
Let $\tau\in\R$ be the twist 
parameter for Fenchel-Nielsen flow. (Compare Wolpert~\cite{Wolpert1}).
The orbit of the Fenchel-Nielsen twist deformation is defined by the representation
\begin{equation*}
\rho(Y) \;:=\; \rho_0(Y) \exp \big(\,(\tau/2)\, \widehat{\tilde\rho(X)}\,\big)
\end{equation*}
where
\begin{equation*}
\widehat{\tilde\rho(X)} = \bmatrix 1 & 0 \\ 0 & -1 \endbmatrix 
\end{equation*}
defines the one-parameter subgroup
\begin{equation*}
\exp \big(\,(\tau/2)\, \widehat{\tilde\rho(X)}\,\big)
\;=\; \bmatrix e^{\tau/2} & 0 \\ 0 & e^{-\tau/2} \endbmatrix.
\end{equation*}
Now $x = 2 \cosh(l/2)$ is constant but
\begin{align*}
y \;=\; \tr\big(\tilde\rho(Y)\big) 
& \;=\; \tr\bigg( \tilde\rho_0(Y)\,  \bmatrix e^{\tau/2} & 0 \\ 0 & e^{-\tau/2} \endbmatrix\bigg)
\\ & \;=\; 2 \cosh(l/2)\, \cosh(\tau/2).
\end{align*}
Similarly,
\begin{align*}
z \;=\; \tr\bigg(\tilde\rho(X)\,\tilde\rho(Y)\bigg) 
& = \tr\bigg( 
\bmatrix e^{l/2} & 0 \\ 0 & e^{-l/2} \endbmatrix \, 
\tilde\rho_0(Y)\,  \bmatrix e^{\tau/2} & 0 \\ 0 & e^{-\tau/2} \endbmatrix\bigg) \\
& \;=\; 2 \cosh(\mu/2) \cosh\big((l + \tau)/2\big).
\end{align*}
Now the commutator trace $\tr\big(\tilde\rho([X,Y])\big)$ equals
\begin{equation*}
-2 \cosh(b/2) = \kappa(x,y,z) = 2 - \sinh^2(l/2) \sinh^2(\mu/2) 
\end{equation*}
whence
\begin{equation*}
\cosh^2(\mu/2) = 1 - 4\csch^2(l/2)\sinh^2(b/4).  
\end{equation*}
Therefore the Fricke trace coordinates are expressed in terms of Fenchel-Nielsen coordinates by:
\begin{align*}
x & \;=\;  2 \cosh(l/2) \\
y & \;=\;  2 \sqrt{1 - 4\csch^2(l/2)\sinh^2(b/4)}\, \cosh(\tau/2) \\
z & \;=\;  2 \sqrt{1 - 4\csch^2(l/2)\sinh^2(b/4)}\, \cosh\big((\tau+l)/2).
\end{align*}

\subsection{The two-holed cross-surface.}\label{sec:twoholedcross}

Following John H.\ Conway's suggestion, we call a surface homeomorphic
to a real projective plane a {\em cross-surface.\/} Suppose that $\Sigma = 
C_{0,2}$ is a 2-holed cross-surface (Figure~\ref{fig:ribbongraph5}). 
Then $\pi_1(\Sigma)$ is freely generated by two orientation-reversing 
simple loops $P,Q$ on the interior.
These loops correspond to the two $1$-handles in Figure~\ref{fig:ribbongraph5}. 
The two boundary components $\partial_\pm$ of $\Sigma$ correspond to 
elements
\begin{equation*}
R := P^{-1}Q^{-1}, \quad R' := Q P^{-1},
\end{equation*}
obtaining a redundant geometric presentation of $\pi_1(\Sigma)$:
\begin{equation*}
\pi \;=\; \langle P,Q,R,R' \;\mid\; PQR =  P Q^{-1} R' = \Id \rangle.
\end{equation*}
The characters of the generators of this presentation define a presentation of the 
character ring
\begin{equation*}
\C[f_P,f_Q,f_R,f_{R'}] \,/\, \big( f_R + f_{R'} -  f_P f_Q \big)
\end{equation*}
the relation being \eqref{eq:basic}. 
Of course $p,q,r$ (respectively $p,q,r'$)
are free generators for the character ring
(a polynomial ring in three variables).

The Fricke space of $\Sigma$ was computed by Stantchev~\cite{Stantchev};
compare also the forthcoming paper by Goldman-McShane-Stantchev-Tan~\cite{GMST}. 
For a given hyperbolic structure on $\Sigma$, the holonomy transformations
$\rho(P)$ and $\rho(Q)$ reverse orientation, their traces are
purely imaginary, and the traces of $\rho(R)$ and $\rho(R')$ are real. 
For this reason we write
\begin{align*} 
i p = f_P &  = \tr\big(\rho(P)\big) \in i\R \\
i q = f_Q &  = \tr\big(\rho(Q)\big) \in i\R\\
r   = f_R    & = \tr\big(\rho(R)\big) \in \R  \\
r'  = f_{R'} & = \tr\big(\rho(R')\big) \in \R 
\end{align*}
where $p,q,r,r'\in \R$ and
\begin{equation*}
r' := r + p q \in\R.
\end{equation*}
By an analysis similar to that of $\Sigma_{0,3}$
and $\Sigma_{1,1}$, the Fricke space of $\Sigma$ identifies with
\begin{equation*}
\{ (p,q,r) \in \R^3  \;\mid\;  r \le -2,\; pq + r \ge 2 \}.
\end{equation*}

Compare \cite{Stantchev,GoldmanStantchev,GMST} for further details.

\subsection{The one-holed Klein bottle.}

Now suppose $\Sigma$ is a one-holed Klein bottle.
(Compare Figure~\ref{fig:ribbongraph4}.)
Once again we choose free generators $P,Q$ for $\pi$ corresponding
to the two 1-handles in Figure~\ref{fig:ribbongraph4} which
reverse orientation. The boundary component $D$ corresponds
to $P^2Q^2$ and, writing $R = (P Q)^{-1}$, we obtain 
a redundant geometric presentation 
\begin{equation*}
\pi \;=\; \langle P,Q,R, D \;\mid\;  P Q R = \Id, D  = P^2 Q^2 \rangle
\end{equation*}
and the character ring has presentation
\begin{equation*}
\C[f_P,f_Q,f_R,f_D]/\big( f_D - (f_P f_Q f_R - f_P^2 - f_Q^2 + 4) \big)
\end{equation*}
Since $P,Q$ reverse orientation on $\Ht$,
the functions $f_P,f_Q$ are purely imaginary and $f_R,f_D$ are real.
Thus we write 
\begin{align*}
ip & = f_P = \tr\big(\rho(P)\big) \in i\R \\
iq & = f_Q = \tr\big(\rho(Q)\big) \in i\R \\
r & = f_R = \tr\big(\rho(R)\big) \in \R \\
d & = f_D = \tr\big(\rho(D)\big) \in \R 
\end{align*}
and the Fricke space of $\Sigma$ identifies with
\begin{equation*}
\{ (p,q,r) \in \R^3  \;\mid\; p^2 + q^2 - pqr \ge 0 \}.
\end{equation*}
See Stantchev~\cite{Stantchev} and \cite{GMST} for details.

\section{Three-generator groups and beyond}

Let $\Sigma$ be a compact connected surface-with-boundary.
Suppose $\partial\Sigma\neq\emptyset$.
Then the fundamental group $\pi_1(\Sigma)$ is free of rank $3$
if and only if the Euler characteristic $\chi(\Sigma) = -2$.
Such a surface is homeomorphic to one of the four topological
types:
\begin{itemize}
\item
A $4$-holed sphere $\Sigma_{0,4}$;
\item
A $2$-holed torus $\Sigma_{1,2}$; 
\item
A $3$-holed cross-surface (projective planes) $C_{0,3}$;
\item
A $2$-holed Klein bottle $C_{1,2}$.
\end{itemize}
In this 
section
we only consider the orientable topological types,
namely $\Sigma_{0,4}$ and $\Sigma_{1,2}$.
We relate their character varieties to those of nonorientable
surfaces $C_{0,2}$ and $C_{1,1}$, 
discussed in 
\S\ref{sec:C02} and \S\ref{sec:C11}.

\subsection{The $\SLtC$-character ring of $\F_3$.}

Representations $\rho$ of the free group
$\langle X_1, X_2, X_3\rangle$ 
of rank three correspond to arbitrary triples 
\begin{equation*}
\big( \rho(X_1), \rho(X_2), \rho(X_3) \big) \;\in\; \SLtC^3.
\end{equation*}
As before we consider the quotient space
(in the sense of Geometric Invariant Theory) 
under the action of $\SLtC$ by inner automorphisms,
the {\em character variety.\/} 
Its coordinate ring is by definition the 
subring of invariants (the {\em character ring\/}) 
\begin{equation*}
\C[\SLtC^3]^{\PSLtC} \;\subset\; \C[\SLtC^3] 
\end{equation*}
of the induced {\em effective\/} $\PSLtC$-action on the ring of 
coordinate ring $\C[\SLtC^3]$.

We saw in \S 2 that for a free group of rank two,
the character variety is an affine space and 
the charcter ring is a polynomial ring.
The situation in rank three is more complicated.
The character variety $V_3$ is a six-dimensional
hypersurface in $\C^7$, which admits a branched double
covering onto the six-dimensional affine space $\C^6$.

Explicitly,
the character ring $\RR_3$ is generated by eight trace functions
\begin{equation*}
t_1, t_2, t_3, t_{12}, t_{23}, t_{13}, t_{123}, t_{132}
\end{equation*}
defined by
\begin{align*}
t_i (\rho) &\; :=\; \tr\big( \rho(X_i)\big) \\
t_{ij} (\rho) &\;:=\; \tr\big( \rho(X_iX_j)\big) \\
t_{ijk} (\rho) &\;:=\; \tr\big( \rho(X_iX_jX_k)\big) 
\end{align*}
subject to two relations expressing the sum and product of traces
of the length 3 monomials in terms of traces of monomials of length 1
and 2:
\begin{align}
t_{123} \,+\, t_{132} & \;=\; t_{12}t_3 + t_{13}t_2 + t_{23}t_1 - t_1t_2t_3 
\label{eq:sum1}  \\
t_{123}\,\,\,  t_{132} & \;=\;  (t_1^2 + t_2^2 +t_3^2) \,+\,
(t_{12}^2 + t_{23}^2 \,+\, t_{13}^2) \  -  \label{eq:product1}\\  
& \qquad \  ( t_1t_2t_{12} + t_2t_3t_{23} + t_3t_1t_{13} ) \,+\,
t_{12} t_{23} t_{13} \,-\, 4. \notag 
\end{align}
We call \eqref{eq:sum1} the {\em Sum Relation\/} and 
\eqref{eq:product1} the {\em Product Relation\/} respectively.
They imply that the {\em triple traces\/} $t_{123}$ and $t_{132}$
are the respective roots $\lambda$ of the irreducible monic quadratic equation:
\begin{equation*}
\lambda^2 - f_\Sigma \lambda + f_\Pi = 0 
\end{equation*}
where the coefficients:
\begin{equation*}
f_\Sigma \;:=\; t_{12} t_3 + t_{23} t_1 +t_{13} t_2 - t_1 t_2 t_3 
\end{equation*}
and 
\begin{align*}
f_\Pi & \; := \;  (t_1^2 + t_2^2 + t_3^2) \\
& \quad + (t_{12}^2 + t_{23}^2 + t_{13}^2) \\
& \qquad - (t_1t_2t_{12} + t_2t_3t_{23} + t_3t_1t_{13}) \\
& \qquad \quad + t_{12}t_{23}t_{13} - 4 
\end{align*}
are the polynomials appearing in the right-hand sides 
of \eqref{eq:sum1} and \eqref{eq:product1} respectively.
\index{sum relation}
\index{product relation}

\subsubsection*{$V_3$ is a hypersurface in $\C^7$.}
\index{character variety of $\F_3$ is a hypersurface}
Eliminating $t_{132}$ in \eqref{eq:sum1} as
\begin{equation}\label{eq:eliminate132}
 t_{132}  \;=\; t_{12}t_3 + t_{13}t_2 + t_{23}t_1 - t_1t_2t_3  - t_{123},
\end{equation}
realizes $V_3$ as the hypersurface in $\C^7$ consisting of all 
\begin{equation*}
\big(t_1,t_2,t_3,t_{12},t_{23},t_{13},t_{123}\big) \in \C^7
\end{equation*}
satisfying
\begin{align*}
t_{123} \;
\big( t_{12}t_3 + t_{13}t_2 + t_{23}t_1 & - t_1t_2t_3  - t_{123}\big)  = 
  (t_1^2 + t_2^2 +t_3^2)  \,+\,
(t_{12}^2 + t_{23}^2 \,+\, t_{13}^2) \\ &  -  
 ( t_1t_2t_{12} + t_2t_3t_{23} + t_3t_1t_{13} ) \,+\, 
t_{12} t_{23} t_{13} \,-\, 4. 
\end{align*}

\subsubsection*{$V_3$ double covers $\C^6$.}\label{sec:doublecoveringcv}
\index{character variety of $\F_3$ is a double covering}
The double covering of $V_3$ over $\C^6$ arises
from the composition
\begin{equation}\label{eq:doublecovering}
V_3 \hookrightarrow \C^8 \xrightarrow{\Pi} \C^6 
\end{equation}
where
\begin{align*}
\C^8 & \xrightarrow{\Pi}  \C^6 \\
\bmatrix 
t_1 \\ t_2 \\ t_3 \\ 
t_{12} \\ t_{23} \\ t_{13} \\ t_{123} \\ t_{132}
\endbmatrix
&\longmapsto 
\bmatrix 
t_1 \\ t_2 \\ t_3 \\ 
t_{12} \\ t_{23} \\ t_{13}
\endbmatrix
\end{align*}
is the coordinate projection.

\begin{proposition}\label{prop:onto}
The composition \eqref{eq:doublecovering}
\begin{align*}
V & \xrightarrow{\bt}  \C^6 \\
[\rho] &\longmapsto \bmatrix 
t_1(\rho) \\ t_2(\rho) \\ t_3(\rho) \\ 
t_{12}(\rho) \\ t_{23}(\rho) \\ t_{13}(\rho)
\endbmatrix
\end{align*}
is onto. Furthermore it is a double covering
branched along the discriminant hypersurface
in $\C^6$ defined by 
\begin{align*}
\big(t_{12}t_3 + t_{13}t_2 + t_{23}t_1 - t_1t_2t_3\big)^2 & = 
4\bigg( 
(t_1^2 + t_2^2 +t_3^2) \,+\,
(t_{12}^2 + t_{23}^2 \,+\, t_{13}^2) \\ &  -  
( t_1t_2t_{12} + t_2t_3t_{23} + t_3t_1t_{13} ) \,+\,
t_{12} t_{23} t_{13} \,-\, 4 \bigg).
\end{align*}
\end{proposition}
\noindent
The goal of this section is to prove the identities
\eqref{eq:sum1}, \eqref{eq:product1} and Proposition~\ref{prop:onto}.

\subsubsection*{Proof of the Sum Relation.}
\index{sum relation}
To prove these identities, we temporarily introduce the following 
notation. Let $\xi = \rho(X_1)$, $\eta = \rho(X_2)$, $\zeta = \rho(X_3)$,
so that \eqref{eq:sum1} becomes:
\begin{align}
\tr(\xi\eta\zeta) + \tr(\xi\zeta\eta) & = \tr(\xi\eta)\tr(\zeta) \notag\\ 
& \qquad - \tr(\zeta)\tr(\xi)\tr(\eta) + \tr(\zeta\xi)\tr(\eta) \notag\\ 
& \qquad \qquad + \tr(\eta\zeta)\tr(\xi). \label{eq:sum2}
\end{align}
 
To prove \eqref{eq:sum2}, 
apply the Basic Identity \eqref{eq:basic} three times: 
\begin{align}
\tr(\xi\eta\zeta) + \tr(\xi\eta\zeta^{-1}) & = \tr(\xi\eta)\tr(\zeta) \label{eq:xyz}\\ 
\tr(\zeta^{-1}\xi\eta) + \tr(\zeta^{-1}\xi\eta^{-1}) & = \tr(\zeta^{-1}\xi)\tr(\eta) 
				\notag \\
& = (\tr(\zeta)\tr(\xi)-\tr(\zeta\xi))\tr(\eta)  \label{eq:zxy}\\ 
\tr(\eta^{-1}\zeta^{-1}\xi) + \tr(\eta^{-1}\zeta^{-1}\xi^{-1}) & = 
		\tr(\eta^{-1}\zeta^{-1}) \tr(\xi) \notag\\
& = \tr(\eta\zeta) \tr(\xi)\label{eq:yzx}
\end{align}
Now add \eqref{eq:xyz}, subtract \eqref{eq:zxy} and
add \eqref{eq:yzx} to obtain: 
\begin{align}
\big(\tr(\xi\eta\zeta) & + \tr(\xi\eta\zeta^{-1})\big)   - \big(\tr(\zeta^{-1}\xi\eta) + \tr( \zeta^{-1} \xi \eta^{-1}) \big) \notag \\
& \quad + \big( \tr(\eta^{-1} \zeta^{-1} \xi) + \tr( \xi \zeta \eta )\big)  \notag \\
& \qquad = \tr(\xi\eta)\tr(\zeta) -  \big(\tr(\zeta)\tr(\xi)-\tr(\zeta\xi)\big)\tr(\eta)  \notag \\
& \qquad \qquad  + \tr(\eta\zeta) \tr(\xi) \label{eq:sumrearranging}
\end{align}
The right hand side of \eqref{eq:sumrearranging} is the right-hand side of
\eqref{eq:sum2}. The left-hand side of \eqref{eq:sumrearranging} equals:
\begin{align*}
\tr(\xi\eta\zeta) + 
 \big(\tr(\xi\eta\zeta^{-1})  & - \tr(\zeta^{-1}\xi\eta) \big)  \\ +
 \big(-\tr( \zeta^{-1} \xi \eta^{-1})  &+ \tr(\eta^{-1} \zeta^{-1} \xi) \big) \\
 + \tr( \xi \zeta \eta )  & \;=\;  \tr(\xi\eta\zeta) + \tr( \xi \zeta \eta ),
\end{align*}
the left-hand side of \eqref{eq:sum2}, from
which \eqref{eq:sum2} follows.


\subsubsection*{Proof of the Product Relation.}
\index{product relation}
 
We derive this formula in several steps. Directly applying the
Basic Identity \eqref{eq:basic}:
\begin{align}\label{eq:zxzy}
\tr(\zeta\xi\zeta\eta) & = \tr(\zeta\xi) \tr(\zeta\eta) - \tr(\xi\eta^{-1}) \\
& = \tr(\zeta\xi)\tr(\zeta\eta) - \left( \tr(\xi)\tr(\eta) - \tr(\xi\eta) \right) \notag\\
& = \tr(\zeta\xi)\tr(\zeta\eta) - \tr(\xi)\tr(\eta) + \tr(\xi\eta) \notag 
\end{align}
Apply 
a calculation similar to 
\eqref{eq:xyxiy} to $\xi,\zeta^{-1}$:
\begin{equation}\label{eq:xzixizi}
\tr(\xi\zeta^{-1}\xi^{-1}\zeta^{-1})  =  
\tr(\xi)\tr(\zeta)\tr(\zeta\xi) - \tr(\zeta\xi)^2 - \tr(\xi)^2 + 2
\end{equation}

\begin{align}\label{eq:xyzxzy}
\tr(\xi\eta\zeta\xi\zeta\eta) & = \tr(\xi\eta) \tr(\zeta\xi\zeta\eta)  \\ 
& \qquad - \tr(\xi\zeta^{-1}\xi^{-1}\zeta^{-1}) \notag\\
& = \tr(\xi\eta) \left( \tr(\zeta\xi)\tr(\zeta\eta) - \tr(\xi)\tr(\eta) + \tr(\xi\eta) \right) \notag \\
& \qquad - \left(  \tr(\xi)\tr(\zeta)\tr(\zeta\xi) - \tr(\zeta\xi)^2 - \tr(\xi)^2 + 2 \right) \notag\\
& \qquad\qquad\qquad\text{(by \eqref{eq:zxzy} and \eqref{eq:xzixizi})}\notag\\
& = \tr(\xi\eta)\tr(\zeta\xi)\tr(\eta\zeta) \notag \\ 
& \qquad \quad - \tr(\xi)\tr(\eta)\tr(\xi\eta) - \tr(\zeta)\tr(\xi)\tr(\zeta\xi) \notag \\ 
& \qquad \qquad + \tr(\xi\eta)^2 + \tr(\xi)^2 - 2.  \notag   
\end{align}

\noindent 
Finally, applying \eqref{eq:xyzxzy} and the Commutator Identity
\eqref{eq:commutator} to $\eta,\zeta$:
\begin{align*}
\tr(\xi\eta\zeta)\tr(\xi\zeta\eta) & = \tr(\xi\eta\zeta\xi\zeta\eta) + \tr(\eta\zeta\eta^{-1}\zeta^{-1}) \\
& = 
\Big( \tr(\xi\eta)\tr(\zeta\xi)\tr(\eta\zeta)  - \tr(\xi)\tr(\eta)\tr(\xi\eta) \\ 
& \quad  - \tr(\zeta)\tr(\xi)\tr(\zeta\xi) + \tr(\xi\eta)^2 + \tr(\zeta\xi)^2 + \tr(\xi)^2 - 2 \Big) \\
& \qquad + \Big(
\tr(\eta)^2 + \tr(\zeta)^2 + \tr(\eta\zeta)^2 - \tr(\eta)\tr(\zeta)\tr(\eta\zeta) - 2  \Big) \end{align*}
obtaining \eqref{eq:product1}.

\subsubsection*{Proof that $\bt$ is onto.}
\index{character variety of $\F_3$ is a hypersurface}


Now we prove Proposition~\ref{prop:onto}.
For a more general treatment see Florentino~\cite{Florentino}.

Theorem~\ref{thm:vogt} guarantees $\xi_1,\xi_2\in\SLtC$  such that
\begin{align}\label{eq:initial}
\tr(\xi_1) & = t_1, \notag \\ 
\tr(\xi_2) &= t_2,   \\ 
\tr(\xi_1\xi_2) &= t_{12}\notag.
\end{align}
We seek $\xi_3\in\SLtC$ such that
\begin{align}\label{eq:rankthreeequations}
\tr(\xi_3) & = t_3, \notag\\  
\tr(\xi_2\xi_3) &= t_{23},\notag \\  
\tr(\xi_1\xi_3) &= t_{13}. 
\end{align}
To this end, consider the affine subspace $\W$ of $\Mat$ consisting
of matrices $\omega$ satisfying
\begin{align}\label{eq:threeconditions}
\tr(\omega) & = t_3, \notag\\ 
\tr(\xi_2\omega) & = t_{23}, \\  
\tr(\xi_1\omega) & = t_{13}.\notag   
\end{align}
Since the bilinear pairing 
\begin{align*}
\Mat \times \Mat & \longrightarrow \C \\
(\xi,\eta) & \longmapsto \tr(\xi\eta)
\end{align*}
is nondegenerate, each of the three equations in \eqref{eq:threeconditions}
describes an affine hyperplane in $\Mat$.
We first suppose that 
$(t_1,t_2,t_{12})$ describes an irreducible character, that is
$\kappa(t_1,t_2,t_{12})\neq 2$:
\begin{equation}\label{eq:irreducible}
4 - t_1^2 -t_2^2 - t_{12}^2 + t_1t_2t_{12} \neq 0. 
\end{equation}
Our goal will be to find an element 
$\xi_3\in\W$ such that $\det(\xi_3) = 1$.

\begin{lemma}
There exist $\xi_1,\xi_2\in\SLtC$ satisfying \eqref{eq:initial} 
such that $\W$ is a (nonempty) affine line.
\end{lemma}

\begin{proof}
Since $\kappa(t_1,t_2,t_{12})\neq 2$,
Proposition~\ref{prop:irreducibility}, 
\eqref{item:trcomm} %
implies that
the pair $\xi_1,\xi_2$ generates an irreducible representation.

We claim that $\{\Id,\xi_1,\xi_2\}$ is a linearly independent subset of the
$4$-dimensional vector space $\Mat$. Otherwise the nonzero element
$\xi_1$ is a linear combination of $\xi_2$ and $\Id$. Let $v\neq 0$ be an
eigenvector of $\xi_2$. Then the line $(v)$ spanned by $v$
is invariant under $\xi_1$ as well, and hence under the
group generated by $\xi_1$ and $\xi_2$. This contradicts irreducibility
of the representation generated by $\xi_1$ and $\xi_2$.

Since $\{\Id,\xi_1,\xi_2\}$ is linearly independent,
the three linear conditions of \eqref{eq:threeconditions} are independent.
Hence $\W\subset\Mat$ is an affine line.
\end{proof}
\noindent 
Let $\omega_0,\omega_1\in\W$ be distinct elements in this line. Then the function
\begin{align*}
\C &\longrightarrow \C \\
 s &\longmapsto \det \big(s \omega_1 + (1-s) \omega_0\big) 
\end{align*}
is polynomial of degree $\le 2$, and is thus onto unless it is constant. 

We shall show that this map is onto, and therefore $\W\cap \SLtC \neq\emptyset$.
The desired matrix $\xi_3$ will be an element of $\W\cap \SLtC$.

\begin{lemma} Let $\omega_0,\omega_1\in\Mat$. Then
\begin{align*}
\det \big(s \omega_1 + (1-s) \omega_0\big) & =
\det\big(\omega_0  + s (\omega_1-\omega_0)\big) \\ 
& = \det(\omega_0)  
 +\;  s\; \bigg( \tr(\omega_0)\tr(\omega_1-\omega_0)-\tr\big(\omega_0(\omega_1-\omega_0)\big)\bigg) \\
& \qquad \qquad\qquad  +\,s^2\; \det(\omega_1-\omega_0)
\end{align*}
\end{lemma}
\begin{proof}
Clearly
\begin{equation}\label{eq:trlinear}
\tr\big(\omega_0  \,+\; s\, (\omega_1-\omega_0)\big)  =
\tr(\omega_0) \,+\; s \, \tr(\omega_1-\omega_0).
\end{equation}
Now
\begin{equation}\label{eq:detmat}
\det(\omega) = \frac{\tr(\omega)^2-\tr(\omega^2)}2 
\end{equation}
whenever $\omega\in\Mat$.
Now apply \eqref{eq:detmat} to \eqref{eq:trlinear}
taking 
\begin{equation*}
\omega = \omega_0 + s (\omega_1-\omega_0) 
\end{equation*}
\end{proof} 
\noindent
Thus the restriction $\det|_\W$ is constant only if $\det(\omega_1-\omega_0) = 0$.

Choose a solution $\xi$ of $\xi + \xi^{-1} = t_{12}$. 
Work in the slice
\begin{equation*}
\xi_1 = \bmatrix t_1 & -1 \\ 1 & 0 \endbmatrix,\;
\xi_2 = \bmatrix 0 & \xi \\ -\xi^{-1} & t_2 \endbmatrix. 
\end{equation*}
The matrix $\omega_0\in\Mat$ defined by:
\begin{equation*}
\omega_0  = 
\bmatrix t_3 & 
\big( (t_{13}-t_1t_3)\xi + t_{23}\big)\xi/(\xi^2-1) \\
\big( (t_{13}-t_1t_3) + t_{23}\xi\big)/(\xi^2-1) & 0 
\endbmatrix, 
\end{equation*}
satisfies  \eqref{eq:threeconditions}.
Any other $\omega\in\W$ must satisfy
\begin{align}\label{eq:threelinearconditions}
\tr\big(\omega-\omega_0\big) & = 0, \notag\\ 
\tr\big(\xi_2(\omega-\omega_0)\big) & = 0, \\  
\tr\big(\xi_1(\omega-\omega_0)\big) & = 0.\notag   
\end{align}
\begin{lemma} 
Any solution $\omega-\omega_0$ of \eqref{eq:threelinearconditions} is a multiple
of
\begin{equation*}
\Lie(\xi_1,\xi_2) = \xi_1\xi_2 - \xi_2\xi_1 =
\bmatrix \xi^{-1} - \xi & -t_2 + t_1\xi \\
-t_2 + \xi^{-1}t_1 & 
 \xi  - \xi^{-1}  \endbmatrix.
\end{equation*}
\end{lemma} 
\begin{proof} 
The first equation in \eqref{eq:threelinearconditions} 
asserts that $\omega-\omega_0$ lies in the subspace $\slt$, upon which the
trace form is nondegenerate. The second and third equations assert
that $\omega-\omega_0$ is orthogonal to $\xi_1$ and $\xi_2$. By \eqref{eq:irreducible},
$\xi_1,\xi_2$ and $\Id$ are linearly independent in $\Mat$, so the solutions
of \eqref{eq:threelinearconditions} form a one-dimensional linear subspace.
The Lie product
\begin{equation*}
\Lie(\xi_1,\xi_2) = \xi_1 \xi_2 - \xi_2 \xi_1  
\end{equation*}
is nonzero and lies in $\slt$. Furthermore,
for $i=1,2$, 
\begin{equation*}
\tr( \xi_i \xi_1 \xi_2) = \tr( \xi_i \xi_2 \xi_1)
\end{equation*}
implies that $\Lie(\xi_1,\xi_2)$ is orthogonal to $\xi_1$ and $\xi_2$. The lemma follows.
\end{proof} 
\noindent
Parametrize $\W$  explicitly as
$
\omega \;=\; \omega_0 \,+\, s\, \Lie(\xi_1,\xi_2).
$
By \eqref{eq:twokindsofcommutator},
\begin{equation*}
\det\big(\Lie(\xi_1,\xi_2)\big) = 4 - (t_1^2 +t_2^2 +t_{12}^2 - t_1t_2t_{12}) =
2 - \kappa(t_1,t_2,t_{12}) \neq 0. 
\end{equation*}
By \eqref{eq:irreducible}, the polynomial
\begin{equation*}
\W \xrightarrow{\det} \C 
\end{equation*}
is nonconstant, and hence onto.
Taking 
\begin{equation*}
\omega_1\in(\det|_\W^{-1})(1),  
\end{equation*}
the proof of Proposition~\ref{prop:onto} is complete assuming
\eqref{eq:irreducible}.

The case when $4 - t_1^2 -t_2^2 - t_{12}^2 + t_1t_2t_{12} = 0$ remains.
Then 
\begin{equation*}
t_i = a_i + (a_i)^{-1} 
\end{equation*}
for $i=1,2$, for some $a_1,a_2\in\C^*$.
Then either
\begin{equation}\label{eq:bothplus}
t_{12} = a_1a_2 + (a_1a_2)^{-1} 
\end{equation}
or
\begin{equation}\label{eq:plusminus}
t_{12} = a_1(a_2)^{-1} + (a_1)^{-1}a_2. 
\end{equation}
In the first case \eqref{eq:bothplus}, set
\begin{align*}
\xi_1 &:=  \bmatrix a_1 & t_{13} - a_1 t_3 \\ 0 & (a_1)^{-1} \endbmatrix, \\
\xi_2 &:=  \bmatrix a_2 & t_{23} - a_2 t_3  \\ 
                   0  & (a_2)^{-1} \endbmatrix, \\
\xi_3 &:=  \bmatrix t_3 & -1 \\ 1 & 0 \endbmatrix
\end{align*}
and in the second case \eqref{eq:plusminus}, set
\begin{align*}
\xi_1 &:=  \bmatrix (a_1)^{-1} & t_{13} - (a_1)^{-1} t_3 
\\ 0 & a_1 \endbmatrix, \\
\xi_2 &:=  \bmatrix a_2 & t_{23} - a_2 t_3  \\ 
                   0  & (a_2)^{-1} \endbmatrix, \\
\xi_3 &:=  \bmatrix t_3 & -1 \\ 1 & 0 \endbmatrix
\end{align*}
obtaining $(\xi_1,\xi_2,\xi_3)\in\SLtC^3$ 
explicitly solving \eqref{eq:initial} and \eqref{eq:rankthreeequations}.
The proof of Proposition~\ref{prop:onto} is complete.

\subsection{The four-holed sphere.}
\index{four-holed sphere}
Let $\Sigma \approx \Sigma_{0,4}$ be the four-holed sphere, with boundary components 
$A,B,C,D$ subject to the relation
\begin{equation*}
A B C D = \Id. 
\end{equation*}
The fundamental group is freely generated by 
\begin{equation*}
A = X_1, B = X_2, C = X_3
\end{equation*}
which represent three of the boundary components.
The fourth boundary component is represented by an
element
\begin{equation*}
D  := (X_1 X_2 X_3)^{-1},
\end{equation*}
satisfying the relation
\begin{equation*}
A B C D = \Id.
\end{equation*}
The resulting redundant presentation of the free group is:
\begin{equation*}
\pi = \langle A, B, C, D \mid A B C D = \Id \rangle. 
\end{equation*}
The elements 
\begin{align*}
X & := X_1 X_2 \\ 
Y & := X_2 X_3 \\ 
Z & := X_1 X_3 \\ 
\end{align*}
correspond to simple loops on $\Sigma$ separating $\Sigma$ into two $3$-holed spheres.
(Compare Figure~\ref{fig:quadsphere}.)
The (even more redundant) presentation
\begin{align*}
\pi \;=\; \langle A, B, C, D, X, Y, Z  \,\mid\, & A B C D = \Id, \\
&\; X = A B,\, Y = B C,\, Z = C A \rangle 
\end{align*}

gives regular functions
\begin{align*}
a & = t_1 \\
b & = t_2 \\
c & = t_3 \\
x & = t_{12} \\
y & = t_{23} \\
z & = t_{13} \\
d & = t_{123}.
\end{align*}
generating the character ring.
Using \eqref{eq:sum1} to eliminate $t_{132}$
as in \eqref{eq:eliminate132}, the product relation
\eqref{eq:product1} implies:
\begin{align}
x^2 + y^2 + z^2 + xyz & =  (ab+cd) x \notag \\ 
&  \qquad + (ad+bc)y  \notag \\
&  \qquad \quad + (ac+bd)z  \notag \\
&  \qquad \qquad  + (4 - a^2 - b^2 - c^2 - d^2 - abcd).
\label{eq:definingeq1}
\end{align}
This leads to a presentation of the character ring 
as a quotient of the polynomial ring  
$\C[a,b,c,d,x,y,z]$ by the principal ideal $(\Phi)$ 
generated by 
\begin{align}
\Phi(a,b,c,d;x,y,z) & =x^2 + y^2 + z^2 + xyz  \notag\\ 
& \quad -  (ab+cd) x  - (ad+bc)y   - (ac+bd)z   \notag\\
& \qquad  + a^2 + b^2 + c^2 + d^2 + abcd -4. \label{eq:definingquartic}
\end{align}
Thus the {\em $\SLtC$-character variety\/} is a quartic hypersurface in $\C^7$,
and for fixed boundary traces $(a,b,c,d)\in\C^4$,
the {\em relative $\SLtC$-character variety\/} is the cubic surface in $\C^3$
defined by \eqref{eq:definingeq1}, as was known to Fricke and Vogt.
(Compare Benedetto-Goldman~\cite{Bengol},
Goldman~\cite{erg}, Goldman-Neumann~\cite{goldmanneumann}, 
Cantat-Loray~\cite{CantatLoray} and Cantat~\cite{Cantat}, 
Iwasaki~\cite{Iwasaki}.)

\begin{figure}[h]
\centerline{\epsfxsize=2.0in \epsfbox{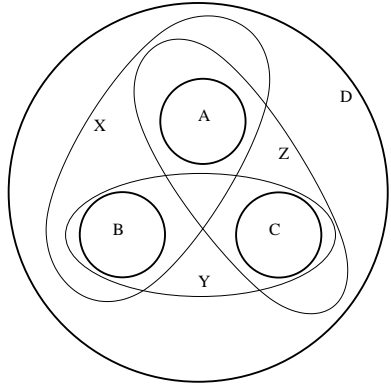}}
\caption{Seven simple curves on $\Sigma_{0,4}$.}
\label{fig:quadsphere}
\end{figure}

\subsubsection*{The Fricke space of $\Sigma_{0,4}$.}
\index{Fricke space of $4$-holed sphere}
We identify the Fricke space $\Fricke(\Sigma_{0,4})$ in terms of 
trace coordinates.
\begin{theorem}\label{thm:FrickeFourHoles}
The Fricke space of a four-holed sphere with boundary
traces $a,b,c,d>2$ is defined by the inequalities in $\R^4\times\R^3$
for $(a,b,c,d)\in\R^4$ and $(x,y,z)\in\R^3$ : 

\begin{align*}
\begin{cases}
a,b,c,d & \ge 2, \quad x  < -2, \\ 
F^-, F^+ & > 0  \\
F^- F^+  & = \frac{(x^2 + a^2 + b^2 - a b x - 4) 
(x^2 + c^2 + d^2 - c d x - 4)}{x^2 -4} \\
\end{cases}
\end{align*}
where 
\begin{align*} 
F^- & = 
\sqrt{2-x} \bigg( y - z - \frac{(a-b)(d-c)}{2-x}  \bigg)  - 
\sqrt{-2-x} \bigg( y + z - \frac{(a+b)(d+c)}{-2-x}  \bigg)  \\
F^+ & = 
\sqrt{2-x} \bigg( y - z - \frac{(a-b)(d-c)}{2-x}  \bigg)  + 
\sqrt{-2-x} \bigg( y + z - \frac{(a+b)(d+c)}{-2-x}  \bigg).
\end{align*}
\end{theorem}

\begin{proof} 
For a given hyperbolic structure on $\Sigma$, the holonomy generators
$\rho(A)$, $\rho(B)$, $\rho(C)$, $\rho(D) \in\PSLtR$ are hyperbolic or parabolic. 
Choose lifts \begin{equation*}
\trho(A), \trho(B),\trho(C),\trho(D) \in\SLtR 
\end{equation*}
which have positive trace.
Since $\rho$ is a representation $\pi_1(\Sigma)\longrightarrow\PSLtR$,
\begin{equation*}
\trho(A) \trho(B) \trho(C) \trho(D) = \pm \Id.
\end{equation*} 
We claim that $\trho(A) \trho(B) \trho(C) \trho(D) = \Id.$
Since 
\begin{align*}
\tr\big(\trho(A)\big) & \ge 2 \\
\tr\big(\trho(B)\big) & \ge 2 \\
\tr\big(\trho(C)\big) & \ge 2 \\
\tr\big(\trho(D)\big) & \ge 2,
\end{align*}
each of
$ \trho(A), \trho(B),\trho(C),\trho(D)$ lies in a unique
one-parameter subgroup of $\SLtR$. The corresponding embeddings 
define trivializations of the corresponding flat $\PSLtR$-bundle over 
each component of $\partial\Sigma$ as in Goldman~\cite{thesis,Topcomps}. 
Namely, since each component $\partial_i(\Sigma)$ is a closed $1$-manifold,
lifting a homeomorphism $\R/\Z \longrightarrow \partial_i(\Sigma)$ to
\begin{equation*}
\R \longrightarrow \widetilde{\partial_i(\Sigma)} 
\end{equation*}
the flat bundle over $\partial_i(\Sigma)$ with holonomy $\gamma_i$
lifts to the quotient of the trivial principal $\PSLtR$-bundle
$\R\times \PSLtR$ by the $\Z$-action generated by
\begin{equation*}
(t, g) \longmapsto (t + 1, \gamma_i g).
\end{equation*}
The corresponding trivialization is covered by the $\Z$-equivariant
isomorphism
\begin{equation*}
(t, g )  \longmapsto \bigg(t , \exp\big(-t \log(\gamma_i) g\big)\bigg)
\end{equation*}
where
\begin{equation*}
\bigg\{  \exp\big(t \log(\gamma_i)\big) \bigg\}_{t\in\R} 
\end{equation*}
is the unique one-parameter subgroup of $\PSLtR$ containing
$\gamma_i$ as above.

Since $\chi(\Sigma) = -2$, the Euler class of the representation
$\rho$ equals $-2$ and is even. The obstruction to lifting a representation
to the double covering space 
\begin{equation*}
\SLtR\longrightarrow\PSLtR  
\end{equation*}
is the second Stiefel-Whitney class, which is the reduction of the Euler 
class modulo $2$. 
Therefore $\trho$ defines a representation and 
$\trho(A) \trho(B) \trho(C) \trho(D) = \Id $
as claimed.
\index{relative Euler class}

Furthermore, if $\rho$ is a Fuchsian representation, then $X$ is
represented by a unique closed geodesic on $\Sigma$ and
\begin{equation*}
\rho(X)=\rho(A)\rho(B)  
\end{equation*}
is hyperbolic. The relative Euler classes of
the restriction of $\rho$ to the subsurfaces complementary to $X$ sum
to $\pm 2$. Since they are constrained to equal $-1,0,+1$, they both
must be equal to $+1$ or both equal to
$-1$. (Compare~\cite{thesis,Topcomps}.)  It follows that the trace $x
= \tr(X) < -2$.

We study \eqref{eq:definingeq1} using the following identity:
\begin{align}\label{eq:definingeq1a} 
4 &  (4 - x^2)  \bigg\{ x^2 + y^2 + z^2 + xyz \notag
\\ & \quad - \big( (ab+cd) x  + (ad+bc)y  + (ac+bd)z \big) \notag  \\
 & \qquad + \big( a^2 + b^2 + c^2 + d^2 + abcd - 4\big) \bigg\} \notag \\
&  = (2 + x) \bigg\{ (y-z)(2-x) + (a-b)(c-d) \bigg\}^2 \notag \\
&  \quad + (2 - x)  \bigg\{ (y+z)(2+x) - (a+b)(c+d) \bigg\}^2 \notag \\ 
& \qquad - 4 
\kappa_{a,b}(x) \kappa_{c,d}(x).
\end{align}
where 
\begin{equation}\label{eq:defkappa}
\kappa_{p,q}(x) := x^2 + p^2 + q^2 - p q x - 4.
\end{equation}
This function equals $\kappa(p,q,x)-2$, where $\kappa$
is the  commutator trace function defined in \eqref{eq:commutator}. 

When $x\neq \pm 2$, rewrite \eqref{eq:definingeq1} using
\eqref{eq:definingeq1a} as follows:

\begin{align}  
& {\frac {2+x} 4} {\left( (y+z) - {\frac
{( a + b)(d + c)} {2+x}} \right)} ^2  \notag\\
& \qquad + 
{\frac {2-x} 4} {\left( (y-z) - {\frac
{( a - b)(d - c)} {2-x}} \right)}^2 . \notag \\ 
& \qquad \qquad =
{\frac {\kappa_{a,b}(x)\kappa_{c,d}(x)}{4-x^2}} 
\label{eq:conics}  
\end{align}
(Compare (3-3) of Benedetto-Goldman~\cite{Bengol}.)

We fix $a,b,c,d \ge 2$. As $x$ varies, \eqref{eq:conics} defines
a family of conics parametrized by $x$.
For $x<-2$, this conic is a hyperbola, denoted  $H_{a,b,c,d;x}$.
The solutions of \eqref{eq:conics} for $a,b,c,d\ge 2$ and $x<-2$
fall into two connected components corresponding to the two components
of the hyperbolas.

We explicitly describe these components. First observe that if $a,b\ge 2$ and
$x<-2$, then 
\begin{equation*}
\kappa_{a,b}(x) > 16,\qquad \kappa_{c,d}(x) > 16,\qquad 4 - x^2 < 0 
\end{equation*}
so the left-hand side \eqref{eq:conics} is negative.
For notational simplicity denote its opposite by $k = \kappa_{a,b,c,d;x}$:
\begin{equation*}
k = 
\kappa_{a,b,c,d;x} :=  
\frac {\kappa_{a,b}(x)\kappa_{c,d}(x)}{x^2-4} \;>\; 0.
\end{equation*}
Rewrite 
\eqref{eq:conics} as:
\begin{align*}  
& {\frac {2-x} 4} {\left( (y-z) - {\frac{( a - b)(d - c)} {2-x}} \right)}^2 \\
& \qquad - {\frac {-2-x} 4} {\left( (y+z) - {\frac{( a + b)(d + c)} {2+x}} \right)} ^2  
= k
\end{align*}
Factoring the left-hand side of this equation, 
rewrite \eqref{eq:conics} as:
\begin{equation}\label{eq:FactorF}
F^+(y,z)  \,   
F^-(y,z) \; = \; 
k
\end{equation}
where the functions $F^\pm(y,z)$ are defined as:
\begin{align}
F^-(y,z)  = 
F_{a,b,c,d;x}^-(y,z)  & := 
\sqrt{2 - x} \bigg( y - z - \frac{(a-b)(d-c)}{2-x}  \bigg)  \notag \\ 
& \quad - \sqrt{-2-x}  \bigg( y + z - \frac{(a+b)(d+c)}{-2-x}  \bigg). \label{eq:Fminus}\\
F^+(y,z)  = 
F_{a,b,c,d;x}^+(y,z)  & := 
\sqrt{2 - x} \bigg( y - z - \frac{(a-b)(d-c)}{2-x}  \bigg) \notag \\
& \quad +  \sqrt{-2-x}  \bigg( y + z - \frac{(a+b)(d+c)}{-2-x}  \bigg)  \label{eq:Fplus}
\end{align}
For fixed $a,b,c,d\ge 2$ and $x< -2$, the functions $F^\pm(y,z)$ are affine functions of $y,z$.

We identify each of the two components of the hyperbola $H_{a,b,c,d;x}$ defined by \eqref{eq:FactorF}.
One component, denoted $H^+_{a,b,c,d;x}$, is the intersection of $H_{a,b,c,d;x}$ with the
open half plane
\begin{equation*}
F_{a,b,c,d;x}^-(y,z) > 0.
\end{equation*}
Equivalently, $H^+_{a,b,c,d;x}$ is the intersection of $H_{a,b,c,d;x}$ with the open half-plane
\begin{equation*}
F_{a,b,c,d;x}^+(y,z) > 0.
\end{equation*}
Similarly the other component $H^-_{a,b,c,d;x}$ is the intersection of $H_{a,b,c,d;x}$ with the 
open half-plane
\begin{equation*}
F_{a,b,c,d;x}^-(y,z) < 0,
\end{equation*}
or, equivalently,
\begin{equation*}
F_{a,b,c,d;x}^+(y,z) < 0.
\end{equation*}
The union of these hyperbola components correspond to values of the relative
Euler class (compare \cite{Topcomps}) as follows. Either
\begin{equation*}
H^+ := \bigcup_{a,b,c,d\ge 2, x < -2}  H^+_{a,b,c,d;x}
\end{equation*}
or
\begin{equation*}
H^- := \bigcup_{a,b,c,d\ge 2, x < -2}  H^-_{a,b,c,d;x}
\end{equation*}
corresponds to characters of representations with relative Euler class
$0$. The other component corresponds to representations with relative
Euler class $\pm 1$. There is no way to distinguish between relative
Euler class $+1$ and $-1$ since the characters are equivalence classes under 
the group $\PGLtR$, which does not preserve orientation. 
To determine which one is 
which, it suffices to check one single example and use continuity of the integer-valued
relative Euler class.

Here is an example whose relative Euler class is zero. 
Choose
\begin{equation*}
\rho(A) = \rho(D)^{-1}, \qquad \rho(B) = \rho(C)^{-1} 
\end{equation*}
so that the relation
\begin{equation*}
\rho(A)\rho(B)\rho(C)\rho(D) = \Id 
\end{equation*}
is trivially satisfied. Clearly such a 
representation depends only on the pair $\rho(A),\rho(B)$ 
which is arbitrary. Furthermore we restrict the boundary traces
to satisfy:
\begin{equation*}
a = \tr\big(\rho(A)\big) \ge 2, \qquad
b = \tr\big(\rho(B)\big) \ge 2. 
\end{equation*}
This space is connected and contains the character of the 
trivial representation, whose relative Euler class is zero.
Now consider the specific example:
\begin{equation*}
\rho(A) := \bmatrix 1 & 1 \\ 0 & 1 \endbmatrix, \qquad
\rho(B) := \bmatrix 1 & 0 \\ x-2 & 1 \endbmatrix
\end{equation*}
where $x<-2$ is arbitrary. Then $a = b = c = d = 2$ and
\begin{equation*}
y = 2, \qquad z = 4 - x.  
\end{equation*}
In particular $y-z < 0 $ and $y + z > 0$. 
and therefore \eqref{eq:Fminus} implies
\begin{equation*}
F_{a,b,c,d;x}^\pm < 0 
\end{equation*}
proving that this representation has a character in 
$H^-$, proving Theorem~\ref{thm:FrickeFourHoles}.
\end{proof}

\subsection{The two-holed torus.}
\index{two-holed torus}

\begin{figure}[h]
\centerline{\epsfxsize=2.0in \epsfbox{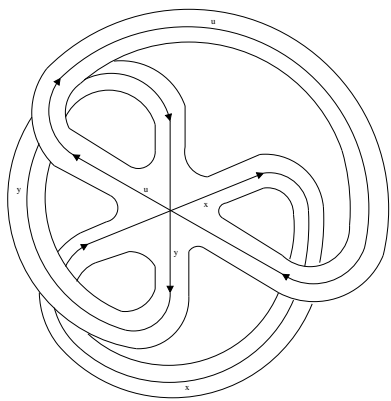}}
\caption{A ribbon graph representing a 2-holed torus}
\label{fig:twoholedtorus}
\end{figure}

The two-holed torus admits a redundant geometric presentation corresponding to the ribbon graph
depicted in Figure~\ref{fig:twoholedtorus}:
\begin{equation*}
\langle A, B, U, X, Y \mid  A = U X Y, \quad B =  U Y X \rangle.
\end{equation*}
The $1$-handles correspond to free generators 
$U,X,Y$ and the boundary components correspond to the triple products
\begin{align*}
A & = U X Y, \\ 
B &=  U Y X.
\end{align*}
Since the curves corresponding to $X,Y,U$ in Figure~\ref{fig:twoholedtorus}
intersect transversely at the basepoint, the double products  
\begin{align*}
V & = U X, \\ 
W & = U Y, \\
Z & = X Y
\end{align*}
are represented by simple loops as well. The 
Sum Relation \eqref{eq:sum1} and the Product Relation 
\eqref{eq:product1} imply that the relative character variety for
$\Sigma_{1,2}$ is defined by 
\begin{align}
a + b & = y v + x w + z u - u x y \notag \\
a b & = x^2 + y^2 + u^2 + v^2 + w^2 + z^2 - x y z - y u w - u x v 
+ v w z - 4. 
\label{eq:twoholedtoruschars}
\end{align}

Button~\cite{Button} gives defining inequalities for the Fricke space,
where the boundary components are mapped to parabolics, as follows.
First consider the $\R$-locus of the character variety, defined as the
set of all $(a,b;u,x,y,z,v,w)\in\R^8$ satisfying  
\eqref{eq:twoholedtoruschars} above, and
\begin{equation*}
a = b = 2. 
\end{equation*}
The regular neighborhood of the union of the loops corresponding to a pair of $X,Y,U$ is
an embedded one-holed torus. For example corresponding to the pair $X,Y$ is a one-holed
torus whose boundary corresponds to the commutator $[X,Y]$, and cuts $\Sigma$ into
the one-holed torus and a three-holed sphere.
This commutator has trace $\kappa(x,y,z)$. Similarly the pair $Y,U$ determines a separating
curve whose corresponding trace function is $\kappa(y,u,w)$ and
the pair $U,X$ determines a separating
curve whose corresponding trace function is $\kappa(u,x,v)$.
The preceding discussions of the Fricke spaces of the one-holed torus and the three-holed sphere imply: 
\begin{equation*}
\kappa(x,y,z) \;<\; -2,\quad \kappa(y,u,w) \;<\; -2,\quad \kappa(u,x,v) \;<\; -2.
\end{equation*}
Button~\cite{Button} shows that these necessary conditions are sufficient,
thus obtaining an explicit description of the Fricke-Teichm\"uller space of $\Sigma_{1,2}$
in terms of traces. (The reader should draw these curves on the ribbon graph
depicted in Figure~\ref{fig:twoholedtorus}.)

\subsection{Orientable double covering spaces.}
\index{orientable double covering spaces}
Let $\Sigma$ be a nonorientable surface of $\chi(\Sigma) = -1$ and
$\hS\xrightarrow{\hP}\Sigma$ be its orientable covering space.
There are two cases:
\begin{itemize}
\item 
$\Sigma\approx C_{0,2}$ and $\hS\approx \Sigma_{0,4}$;
\item 
$\Sigma\approx C_{1,1}$ and $\hS\approx \Sigma_{1,2}$.
\end{itemize}
Then $\pi_1(\Sigma)\cong \F_2$ and 
$\pi_1(\hS) \cong \F_3$. Denote a set of free generators
of $\pi_1(\Sigma)$ by $X_1,X_2$ which correspond to
orientation-reversing loops on $\Sigma$. The image of
\begin{equation*}
\pi_1(\hS) \stackrel{\hP_*}\hookrightarrow \pi_1(\Sigma) = \langle X_1,X_2\rangle 
\end{equation*}
equals the kernel of the homomorphism
\begin{align*}
\pi_1(\Sigma) & \longrightarrow \{\pm 1 \}\\
X_1 & \longmapsto -1 \\ 
X_2 & \longmapsto -1,
\end{align*}
which is freely generated by, for example,
\begin{align*}
Y_1 &  = X_1^2 \\
Y_2 &  = X_1^{-1} X_2^{-1} \\
Y_3 &  = X_2^2.
\end{align*}
The deck transformation of 
$\hS\xrightarrow{\hP}\Sigma$ 
is induced by the restriction of the inner automorphism $\Inn(X_1)$
to $(\hP)_*\big(\pi_1(\Sigma)\big)$:
\begin{align*}
Y_1 & \longmapsto Y_1 \\
Y_2 & \longmapsto Y_3^{-1} Y_2^{-1} Y_1^{-1} \\
Y_3 & \longmapsto Y_1 Y_2 Y_3 Y_2^{-1} Y_1^{-1}. 
\end{align*}
The character ring of $\pi_1(\Sigma)\cong \F_2$ is the polynomial
ring $\C[x_1,x_2,x_{12}]$. The
character ring of $\pi_1(\hS)\cong \F_3$ is the quotient
\begin{equation*}
\C[y_1,y_2,y_3,y_{123},y_{12},y_{23},y_{13}]/(\mathfrak{I})
\end{equation*}
where $(\mathfrak{I})$ is the principal ideal generated by 
\begin{align}
\Phi(y_1,y_2,y_3,y_{123} & ,y_{12},y_{23},y_{13}) \notag \\
\; =\; 
y_{12}^2 +  & y_{13}^2  + y_{23}^2   + y_{12}y_{13}y_{23}  \notag\\ 
& -  (y_1y_2+y_3y_{123}) y_{12}   - (y_1y_{123}+ y_2y_3)y_{23}  \notag \\
& \quad - (y_1y_3+y_2y_{123})y_{13}  + y_1^2 + y_2^2 + y_3^2 + y_{123}^2 
+ y_1y_2y_3y_{123} -4 
\end{align}
where
$\Phi$ is the polynomial \eqref{eq:definingquartic}.
The automorphism 
\begin{equation*}
\Inn(X_1)|_{(\hP)_*\big(\pi_1(\Sigma)\big)} 
\end{equation*}
corresponding to the deck transformation induces the involution
 of character rings: 
\begin{align*}
\RR_3 & \leftrightarrow \RR_3 \\
y_1 & \leftrightarrow y_1 \\
y_2 & \leftrightarrow y_{123} \\
y_3 & \leftrightarrow y_3 \\
y_{12} & \leftrightarrow y_{23} \\
y_{13} & \leftrightarrow y_1y_3 - y_{13} - y_{12}y_{23} + y_{123}y_2 \\
y_{23} & \leftrightarrow y_{12} \\
y_{123} & \leftrightarrow y_2 
\end{align*} 
The covering space $\hP$ induces the embedding of character rings
\begin{align*}
\RR_2 & \hookrightarrow \RR_3 \\
y_1 & \longmapsto x_1^2 -2 \\
y_2 & \longmapsto x_{12} \\
y_3 & \longmapsto x_2^2-2 \\
y_{12} & \longmapsto x_1x_2-x_{12} \\
y_{13} & \longmapsto x_1 x_2 x_{12} - x_1^2 - x_2^2 + 2 \\
y_{23} & \longmapsto x_1x_2-x_{12} \\
y_{123} & \longmapsto x_{12}
\end{align*} 

The two topological types for $\Sigma$ differ by their choice of {\em peripheral structure:\/}
\begin{itemize}
\item $\Sigma \approx C_{0,2}$ has two boundary components corresponding to: 
\begin{equation*}
\delta_1 := Y_2 = X_1^{-1} X_2^{-1}, \;  
\delta_2 = Y_1 Y_2 = X_1X_2^{-1}. 
\end{equation*}
\item $\Sigma \approx C_{1,1}$ has one boundary component corresponding to 
\begin{equation*}
\delta := Y_1 Y_3 = X_1^2 X_2^2. 
\end{equation*}
\end{itemize}

\subsection{The two-holed cross-surface.}\label{sec:C02}
The fundamental group of the two-holed cross-surface $C_{0,2}$
is free of rank two, with presentation
\begin{equation*}
\pi_1(C_{0,2}) \;:=\; \langle U,V,W,W'\; \mid\; W = U V,\, W' = V^{-1} U \rangle \cong \F_2.
\end{equation*}
The free generators $U,V$ correspond to orientation-reversing simple curves
on $C_{0,2}$ and $W,W'$ correspond to the components of $\partial C_{0,2}$.
 
The orientable double covering-space $\widehat{C_{0,2}}\longrightarrow C_{0,2}$ 
is connected, has four boundary components (since $C_{0,2}$ has two boundary 
components, each of which is orientable) and has Euler characteristic
$-2 = 2 \chi(C_{0,2})$. Therefore $\widehat{C_{0,2}}\approx \Sigma_{0,4}$, 
the four-holed sphere, and has presentation
\begin{equation*}
\pi = \langle A, B, C, D \mid A B C D = \Id \rangle.
\end{equation*}
 
The corresponding monomorphism of fundamental groups is:
\begin{align*}
\pi_1(\Sigma_{0,4}) & \hookrightarrow \pi_1(C_{0,2}) \\
A & \longmapsto W = U V \\
B & \longmapsto W' = V^{-1} U \\
C & \longmapsto \Inn(U^{-1})(W')^{-1} = U^{-2} V U \\
D & \longmapsto \Inn(U^{-1})(W)^{-1} = U^{-1} V^{-1}.
\end{align*}
 
The character ring of $\pi_1(C_{0,2})$ is the polynomial ring
$\RR_2 \cong \C[u,v,w]$ and the character ring of $\pi_1(\Sigma_{0,4})$ is the 
quotient of $\C[a,b,c,d,x,y,z]$ by the relation 
defined by \eqref{eq:definingeq1}.
The induced homomorphism of character rings is:

\begin{align*}
\RR_3 & \longrightarrow \RR_2 = \C[u,v,w] \\
a &\longmapsto  w \\
b &\longmapsto  u v - w = w' \\
c &\longmapsto  u v - w = w' \\
d &\longmapsto  w \\
x &\longmapsto  u^2 -2 \\
y &\longmapsto  u^2 + v^2 + w^2 - uvw -2  \\
z &\longmapsto  v^2 - u^2\big(u^2 + v^2 + w^2 - uvw -2\big) - 2,
\end{align*}
evidently satisfying the defining equation
\eqref{eq:definingeq1} for the character variety of $\Sigma_{0,4}$.

\subsection{The one-holed Klein bottle.}\label{sec:C11}
The fundamental group of the one-holed Klein bottle
$C_{1,1}$ is free of rank two, with presentation
\begin{equation*}
\pi_1(C_{1,1}) \;:=\; \langle P,Q,R,D \; \mid\; PQR = P^2 Q^2 D = \Id 
\rangle \cong \F_2.
\end{equation*}
The free generators $P,Q$ correspond to orientation-reversing simple curves
on $C_{1,1}$ and $D$ corresponds to $\partial C_{1,1}$.

The orientable double covering-space $\widehat{C_{1,1}}\longrightarrow C_{1,1}$ 
is connected, has two boundary components (since $\partial C_{1,1})$ is connected
and orientable) and has Euler characteristic
$-2 = 2 \chi(C_{1,1})$. Therefore $\widehat{C_{0,2}}\approx \Sigma_{1,2}$, 
the two-holed torus, and its fundamental group has presentation

This covering-space $\Sigma_{1,2}\longrightarrow C_{1,1}$ induces the monomorphism
\begin{align*}
\pi_1(\Sigma_{1,2}) &\hookrightarrow \pi_1(C_{1,1}) \\
U & \longmapsto P Q\\
X & \longmapsto Q P^{-1}\\
Y & \longmapsto P^2 \\
A & \longmapsto PQ^2P         \sim  P^2 Q^2\\
B & \longmapsto PQP^2QP^{-1}  \sim P^2 Q^2
\end{align*}
where $\pi_1(\Sigma_{1,2})$ is presented as:
\begin{equation*}
\langle A, B, U, X, Y \mid  A = U X Y, \quad B =  U Y X \rangle.
\end{equation*}
 
The character ring of $\pi_1(C_{1,1})$ is the polynomial ring
$\RR_2 \cong \C[p,q,r]$ and the character ring of $\pi_1(\Sigma_{1,2})$ is the 
quotient of $\C[a,b,x,y,z,u,v,w]$ by the relations
defined by \eqref{eq:twoholedtoruschars}.
The covering space $\Sigma_{1,2}\longrightarrow C_{1,1}$ induces the homomorphism
of character rings:
\begin{align*}
\RR_3 &\longrightarrow \RR_2 = \C[p,q,r] \\
u &\longmapsto  r \\
x &\longmapsto  p q - r \\
y &\longmapsto  p^2 -2  \\
v &\longmapsto  q^2 -2 \\
w &\longmapsto  r \\
z &\longmapsto  p( p r - q) - r \\
a &\longmapsto  2 - p^2 - q^2 + p r \\
b &\longmapsto  2 - p^2 - q^2 + p r 
\end{align*}
which evidently satisfies the relations of \eqref{eq:twoholedtoruschars}.

We briefly 
give
a geometric description of 
the deck transformation of the double covering 
of the $\SLtC$-character variety $V_3$ of 
the rank three free group $\F_3$.

Consider the elliptic involution $\iota$ of the torus $\Sigma_{1,0}.$
Writing $\Sigma_{1,0}$ as the quotient $\R^2/\Z^2$, this involution
is induced by the map
\begin{align*}
\R^2 & \longrightarrow \R^2 \\
u &\longmapsto - u.
\end{align*}
This involution has four fixed points, and its quotient orbifold is
$S^2$ with four branch points of order two. 
Choose a small disc $D\subset\Sigma_{1.0}$ 
such that $D$ and its image $\iota(D)$ are disjoint.
Then $\iota$ induces an involution on the complement
\begin{equation*}
\Sigma_{1,0} \setminus (D \cup \iota(D)) \;\approx \; \Sigma_{1,2}. 
\end{equation*}
This involution of the two-holed torus $\Sigma_{1,2}$ induces
the involution of $\pi_1(\Sigma_{1,2}) = \langle U, X, Y\rangle$:
\begin{align*}
U & \longmapsto U^{-1} \\
X & \longmapsto X^{-1} \\
Y & \longmapsto Y^{-1}.
\end{align*}
The quotient orbifold is a disc with four branch points of
order two.  The corresponding involution of character varieties
is the branched double covering \eqref{eq:doublecovering} 
of the character variety $V_3$ over
$\C^6$ described in Proposition~\ref{prop:onto}.
\subsection{Free groups of rank $\ge 3$.}
\index{rank three free groups}
The basic trace identity \eqref{eq:basic},
the Sum Relation \eqref{eq:sum1} and
the Product Relation \eqref{eq:product1}
imply that the trace polynomial $f_w$ of any word 
$w(X_1,X_2,\dots, X_n)$ can be written in terms
of trace polyomials of words
\begin{equation*}
X_{i_1} X_{i_2} \dots 
X_{i_1} X_{i_r} 
\end{equation*}
where $1\le i_1 < i_2 < \dots < i_r \le n$.
The following identity, which may be found in
Vogt~\cite{Vogt}, implies that it suffices
to choose $r\le 3$:

\begin{align*}
2\, t_{1234} & \;=\;
t_1 t_2 t_3 t_4  + t_1 t_{234} + t_2 t_{341} + t_3 t_{412} + t_4 t_{123} \\
& + t_{12} t_{34} + t_{41} t_{23} 
 - t_{13} t_2 t_{24} - t_1 t_2 t_{34} - t_{12} t_3 t_4 - t_4 t_1 t_{23} - t_{41} t_2 t_3
\end{align*}

The $\SLtC$-character variety of a rank $n$ free group has dimension $3n -3$, 
as it corresponds to the quotient of the $3n$-dimensional complex manifold
$\SLtC^n$ by the generically free 
action of the $3$-dimensional group $\PGLtC$.
Thus the transcendence degree of the field of fractions of the 
{\em character ring\/} equals $3n-3$.
In contrast, the above discussion implies that this ring has 
\begin{equation*}
n + \binom{n}{2} + \binom{n}{3} = \frac{n(5 + n^2)}6
\end{equation*}
generators, considerably larger than the dimension $3n-3$ of
the character variety.

%
%


%
%
%
%
%
%
%
%
%

%

\makeatletter \renewcommand{\@biblabel}[1]{\hfill#1.}\makeatother

\newcommand{\bysame}{\leavevmode\hbox to3em{\hrulefill}\,}   

\printindex

\end{document}